\documentclass[reqno,10pt]{amsart}
\usepackage{amsfonts}
\usepackage{a4wide}	
\usepackage{bm}

\numberwithin{equation}{section}

\usepackage{amsmath,amssymb,amsthm,amsfonts}
\usepackage{mathrsfs}
\usepackage{bbm}
\usepackage{hyperref}

\usepackage{color}

\newtheorem{lemma}{Lemma}[section]
\newtheorem{theorem}{Theorem}[section]

\newtheorem{proposition}{Proposition}[section]
\newtheorem{remark}{Remark}[section]

\newcommand{\R}{\mathbb{R}}

\newcommand{\FI}{\mathbf{I}}

\newcommand{\CE}{\mathcal{E}}

\newcommand{\CI}{\mathcal{I}}
\newcommand{\CJ}{\mathcal{J}}
\newcommand{\CK}{\mathcal{K}}

\newcommand{\CR}{\mathcal{R}}

\newcommand{\na}{\nabla}

\newcommand{\al}{\alpha}
\newcommand{\bet}{\beta}
\newcommand{\ga}{\gamma}

\newcommand{\la}{\lambda}
\newcommand{\de}{\delta}

\newcommand{\pa}{\partial}

\newcommand{\eps}{\epsilon}

\newcommand{\eqdef}{\overset{\mbox{\tiny{def}}}{=}}

\allowdisplaybreaks[4]

\makeatletter
\@namedef{subjclassname@2020}{%
  \textup{2020} Mathematics Subject Classification}
\makeatother

\begin{document}

\title[Spatially inhomogeneous Vlasov-Nordstr\"{o}m-Fokker-Planck system]{The spatially inhomogeneous  Vlasov-Nordstr\"{o}m-Fokker-Planck system in the intrinsic weak diffusion regime}

\author[S.-C. Chang]{Shengchuang Chang}
\address[CSC]{School of Mathematics and Statistics and Key Lab NAA-MOE, Central China Normal University, Wuhan 430079, P.R.~China}
\email{csc981020@mails.ccnu.edu.cn}

\author[S.-Q. Liu]{Shuangqian Liu}
\address[SQL]{School of Mathematics and Statistics and Key Lab NAA-MOE,, Central China Normal University, Wuhan 430079, P.R.~China}
\email{sqliu@ccnu.edu.cn}

\author[T. Yang]{Tong Yang}
\address[TY]{Department of Applied Mathematics, The Hong Kong Ploytechnic University, Hung Hom,
Hong Kong, P.R.~China}
\email{t.yang@polyu.edu.hk}

\subjclass[2020]{35Q20, 35B40}
\keywords{Vlasov-Nordstr\"{o}m-Fokker-Planck system, vanishing viscosity, momentum singularity, large time behaviors}
\maketitle





\everymath{\displaystyle}

\maketitle

\begin{abstract}
The spatially homogeneous Vlasov-Nordstr\"{o}m-Fokker-Planck system is known to exhibit nontrivial large time behavior, naturally leading to weak diffusion of the Fokker-Planck operator. This weak diffusion, combined with the singularity of relativistic velocity, present a significant challenge in analysis  for the spatially inhomogeneous counterpart.

In this paper, we demonstrate that the Cauchy problem for the spatially inhomogeneous Vlasov-Nordstr\"{o}m-Fokker-Planck system, without friction, maintains dynamically stable relative to the corresponding spatially homogeneous system. Our results are twofold: (1) we establish the existence of a unique global classical solution and characterize the asymptotic behavior of the spatially inhomogeneous system using a refined weighted energy method; (2) we directly verify the dynamic stability of the spatially inhomogeneous system in the framework of self-similar solutions.

\end{abstract}



\tableofcontents

\section{Introduction}
\subsection{The problem}\label{tp}
The Vlasov-Nordstr\"{o}m-Fokker-Planck (VNFP) system models the kinetic evolution of relativistic, self-gravitating particles undergoing diffusion. In this paper, we are concerned with the Cauchy problem for the following spatially inhomogeneous VNFP  system
in the whole space
\begin{eqnarray}\label{vnfp-ih}
	\left\{
		\begin{array}{l}
\partial_t F+\nabla_p\sqrt{e^{2\phi}+\vert p\vert^2}\cdot\nabla_x F-\nabla_x\sqrt{e^{2\phi}+\vert p\vert^2}\cdot\nabla_p F=e^{2\phi}\nabla_p\cdot(\Lambda_{\phi, p}\nabla_pF),\ t>0,\ x\in\mathbb R^3,\ p\in\mathbb R^3,\\
\partial_t^2\phi-\Delta_x\phi=-e^{2\phi}\int_{\mathbb R^3}\frac{F}{\sqrt{e^{2\phi}+\vert p\vert^2}}{\rm d}p,\ t>0,\ x\in\mathbb R^3,\\
F(0,x,p)=F_0(x,p),\ \phi(0,x)=\phi_0(x),\ \pa_t\phi(0,x)=\phi_1(x),\ x\in\mathbb R^3,\ p\in\mathbb R^3.\\
	\end{array}\right.
\end{eqnarray}
Here, $F=F(t,x,p)\geq0$ is the density distribution function of particles at time $t\geq0$, in the position $x\in\R^3$  with momentum $p\in\R^3$. The scalar gravitational field generated by the particles is denoted by $\phi=\phi(t,x)$.
Additionally, $\Lambda_{\phi,p}$ is the diffusion matrix defined as
\begin{equation*}
	\Lambda_{\phi,p}=\frac{e^{2\phi}\FI+p\otimes p}{\sqrt{e^{2\phi}+\vert p\vert^2}},
\end{equation*}
where $\FI$ is the identity matrix.

The VNFP system was first derived by  Alc\'antara-Calogero \cite{AJ-CS-11}. One can trace its origin back to 1912 when the Finnish theoretical physicist Nordstr\"{o}m proposed his
theory of gravitation. Although the Vlasov-Nordstr\"{o}m system is not a physical model for gravitational field theory, it serves as a useful toy model for the Vlasov-Einstein system. In addition, this model leads to fruitful mathematical exploration of the effects of gravitational radiation on the dynamics of a many-particle system. In 2014, Alc\'{a}ntara-Calogero-Pankavich \cite{ACP-2014} investigated the following spatially homogeneous VNFP system
	\begin{eqnarray}\label{vnpf-ho}
	\left\{
	\begin{array}{l}
		\partial_t\bar F=e^{2\bar\phi}\nabla_p\cdot(\Lambda_{\bar\phi,p}\nabla_p\bar F),\ t>0,\ p\in\mathbb R^3,\\[2mm]
		\frac{d^2\bar\phi}{dt^2}=-e^{2\bar\phi}\int_{\mathbb R^3}\frac{\bar F}{\sqrt{e^{2\bar\phi}+\vert p\vert^2}}{\rm d}p,\ t>0,
	\end{array}\right.
\end{eqnarray}
with the initial data
\begin{equation}\label{ho-id}
	\bar\phi(0)=\phi_{in},\ \bar\phi'(0)=\psi_{in},\ \bar F(0,p)=\bar F_{in}(p),\ p\in\mathbb R^3.
\end{equation}
They showed that the solution of the above system exhibits nontrivial large time behavior. Precisely, unlike the non-relativistic case \cite{CSV-1996},
the authors in \cite{ACP-2014} proved that
\begin{align}\label{lt-ho}
\bar{F}\nrightarrow0,\ \bar\phi(t)\rightarrow-\infty,\ \textrm{as}\ t\rightarrow+\infty.
\end{align}
Actually, Alc\'{a}ntara-Calogero-Pankavich \cite{ACP-2014} also showed that the gravitational field $\bar\phi$ satisfies the following estimate
\begin{equation*}
	-C-\alpha t\leq\bar\phi(t)\leq C-\beta t,
\end{equation*}
for some constants $C$, $\alpha$, $\beta>0$.
This nontrivial large time behavior not only  reveals that the diffusion on the right hand side of \eqref{vnpf-ho}$_{1}$ vanishes
$t\rightarrow+\infty$, but also leads to a momentum singularity for the ``unphysical" stress energy tensor
\begin{align}\label{nt-phi}
\int_{\mathbb R^3}\frac{\bar F}{\sqrt{e^{2\bar\phi}+\vert p\vert^2}}{\rm d}p.
\end{align}
In this paper, we aim to  prove the stability of the spatially inhomogeneous system \eqref{vnfp-ih} around the spatially homogeneous solution $(\bar{F},\bar\phi)$ satisfying \eqref{vnpf-ho} and \eqref{ho-id}, in both the self-similar and non self-similar setting.

Let us now define the perturbation $f=F-\bar F$, $\Phi=\phi-\bar\phi$. Then the new unknown function $(f(t,x,p),\Phi(t,x))$ satisfies
\begin{eqnarray}\label{pt-vnfp}
	\left\{
\begin{array}{l}
	\partial_t f+\nabla_p\sqrt{e^{2\phi}+\vert p\vert^2}\cdot\nabla_x f-\nabla_x\sqrt{e^{2\phi}+\vert p\vert^2}\cdot\nabla_p f=\nabla_x\sqrt{e^{2\phi}+\vert p\vert^2}\cdot\nabla_p\bar F+e^{2\phi}\nabla_p\cdot(\Lambda_{\phi,p}\nabla_pf)\\[2mm]
	\quad+(e^{2\phi}\nabla_p\cdot(\Lambda_{\phi,p}\nabla_p\bar F)-e^{2\bar\phi}\nabla_p\cdot(\Lambda_{\bar\phi,p}\nabla_p\bar F)),\ t>0,\ x\in\mathbb R^3,\ p\in\mathbb R^3,\\[2mm]
	\partial_t^2\Phi-\Delta_x\Phi=-e^{2\phi}\int_{\mathbb R^3}\frac{f}{\sqrt{e^{2\phi}+\vert p\vert^2}}{\rm d}p-\int_{\mathbb R^3}\Big(\frac{e^{2\phi}\bar F}{\sqrt{e^{2\phi}+\vert p\vert^2}}-\frac{e^{2\bar\phi}\bar F}{\sqrt{e^{2\bar\phi}+\vert p\vert^2}}\Big){\rm d}p,
\end{array}\right.
	\end{eqnarray}
with the initial data
\begin{align}\label{pt-id}
	(f(0,x,p),\Phi(0,x),\partial_t\Phi(0,x))&=(F_0-\bar F_{in},\phi_0-\phi_{in},\phi_1-\psi_{in})\notag\\
&\eqdef(f_0(x,p),\Phi_0(x),\Phi_1(x)),\ x\in\mathbb R^3,\ p\in\mathbb R^3.
\end{align}
To cope with the time decay of $\bar\phi$, one can also look for solutions of \eqref{pt-vnfp} and \eqref{pt-id} in the form of
\begin{align}
f(t,x,p)=e^{-3\bar\phi}g(t,x,e^{-\bar\phi}p)=e^{-3\bar\phi}g(t,x,q).\label{sss}
\end{align}
In this setting, we have the following equation for $g(t,x,q)$
\begin{eqnarray}\label{g-eq}
	\left\{
	\begin{array}{l}
		\partial_t g-3\bar\phi' g-\bar\phi' q\cdot\nabla_qg+\nabla_q\sqrt{e^{2\Phi}+\vert q\vert^2}\cdot\nabla_x g-\nabla_x\sqrt{e^{2\Phi}+\vert q\vert^2}\cdot\nabla_q g=\nabla_x\sqrt{e^{2\Phi}+\vert q\vert^2}\cdot\nabla_q\bar G\\[2mm]
		\quad+e^{\bar\phi+2\Phi}\nabla_q\cdot(\Lambda_{\Phi,q}\nabla_qg)+(e^{\bar\phi+2\Phi}\nabla_q\cdot(\Lambda_{\Phi,q}\nabla_q\bar G)-e^{\bar\phi}\nabla_q\cdot(\Lambda_{0,q}\nabla_q\bar G)),\\[2mm]
		\partial_t^2\Phi-\Delta_x\Phi=-e^{\bar\phi+2\Phi}\int_{\mathbb R^3}\frac{g}{\sqrt{e^{2\Phi}+\vert q\vert^2}}{\rm d}q-\int_{\mathbb R^3}\Big(\frac{e^{\bar\phi+2\Phi}\bar G}{\sqrt{e^{2\Phi}+\vert q\vert^2}}-\frac{e^{\bar\phi}\bar G}{\sqrt{1+\vert p\vert^2}}\Big){\rm d}q,
	\end{array}\right.
\end{eqnarray}
with the initial data
\begin{equation}\label{g-id}
	(g(0,x,q),\Phi(0,x),\partial_t\Phi(0,x))=(g_0(x,q),\Phi_0(x),\Phi_1(x)),\ x\in\mathbb R^3,\ q\in\mathbb R^3.
\end{equation}
Here,
\begin{equation*}
\bar F(t,p)=e^{-3\bar\phi}\bar G(t,e^{-\bar\phi}p)=:e^{-3\bar\phi}\bar G(t,q).
\end{equation*}

\subsection{Main results}
Before stating the main results, we introduce the following energy functionals and dissipation rates.
For any $\ga\geq0$, we define the weighted energy functional
\begin{eqnarray}\label{eng-def}
	\begin{split}
\mathcal E^\gamma(t):=\mathcal E^\gamma(f,\Phi)(t)
=&\sum_{\substack{ \vert\alpha\vert+\vert\beta\vert\leq4\\\vert\beta\vert<4}}\int_{\mathbb R^6}e^{(\vert\alpha\vert+3\vert\beta\vert)\bar\phi}(e^{2\phi}+\vert p\vert^2)^\gamma\vert\partial_x^\alpha\partial_p^\beta f\vert^2{\rm d}x{\rm d}p\\&+\sum_{\vert\alpha\vert\leq3}\int_{\mathbb R^3}(\vert\partial_t\partial_x^\alpha\Phi\vert^2+\vert\nabla_x\partial_x^\alpha\Phi\vert^2){\rm d}x+\sum_{|\alpha|=4}\int_{\mathbb R^3}e^{\bar\phi}(\vert\partial_t\partial_x^\alpha\Phi\vert^2+\vert\nabla_x\partial_x^\alpha\Phi\vert^2){\rm d}x,
	\end{split}
\end{eqnarray}
and the dissipation rate
\begin{align*}
 \mathcal D^\gamma(t):=\mathcal	D^\gamma(f,\Phi)(t)=\sum_{\substack{ \vert\alpha\vert+\vert\beta\vert\leq4\\\vert\beta\vert<4}}\int_{\mathbb R^6}&e^{(\vert\alpha\vert+3\vert\beta\vert)\bar\phi+2\phi}(e^{2\phi}+\vert p\vert^2)^{\gamma-\frac{1}{2}}\\&\times(e^{2\phi}\vert\nabla_p\partial_x^\alpha\partial_p^\beta f\vert^2+\vert p\cdot\nabla_p\partial_x^\alpha\partial_p^\beta f\vert^2){\rm d}x{\rm d}p,
\end{align*}
as well as the enhanced dissipation rate
\begin{align*}
\bar{ \mathcal D}^\gamma(t):=\mathcal	D^\gamma(f,\Phi)(t)+\sum_{\substack{1\leq |\al|+|\bet|\leq4\\ |\bet|<4}}\int_{\mathbb R^6} e^{(|\alpha|+3|\beta|)\bar\phi}(e^{2\phi}+|p|^2)^{\ga}|\partial_x^\alpha\partial_p^\beta f|^2{\rm d}x{\rm d}p.
\end{align*}

Furthermore, for convenience of notations, we also define the $(m+n)$-th order energy functional
\begin{equation*}
	\begin{split}
\mathcal E_{m,n}^\gamma(t):=\mathcal	E_{m,n}^\gamma(f,\Phi)(t)
=&\sum_{\vert\alpha\vert=m,\vert\beta\vert=n}\int_{\mathbb R^6}e^{(\vert\alpha\vert+3\vert\beta\vert)\bar\phi}(e^{2\phi}+\vert p\vert^2)^\gamma\vert\partial_x^\alpha\partial_p^\beta f\vert^2{\rm d}x{\rm d}p
\\&+\sum_{\vert\alpha\vert=m}\int_{\mathbb R^3}e^{d\bar\phi}(\vert\partial_t\partial_x^\alpha\Phi\vert^2+\vert\nabla_x\partial_x^\alpha\Phi\vert^2){\rm d}x,
\end{split}
\end{equation*}
where $d=\left\{\begin{array}{l}0,\ {\rm if}\ m \leq3\\
1,\ {\rm if}\ m=4
\end{array}\right.$,
with the corresponding dissipation rate
\begin{eqnarray}
	\begin{split}
\mathcal D_{m,n}^\gamma(t):=\mathcal	D_{m,n}^\gamma(f,\Phi)(t)
=\sum_{\vert\alpha\vert=m,\vert\beta\vert=n}&\int_{\mathbb R^6}e^{(\vert\alpha\vert+3\vert\beta\vert)\bar\phi+2\phi}(e^{2\phi}+\vert p\vert^2)^{\gamma-\frac{1}{2}}\\&\times(e^{2\phi}\vert\nabla_p\partial_x^\alpha\partial_p^\beta f\vert^2+\vert p\cdot\nabla_p\partial_x^\alpha\partial_p^\beta f\vert^2){\rm d}x{\rm d}p,\notag
\end{split}
\end{eqnarray}
and the enhanced dissipation rate
\begin{align*}
	\bar{\mathcal D}^\ga_{m,n}(t):=\mathcal	D^\gamma_{m,n}(t)+\sum_{|\alpha|=m,|\beta|=n}\int_{\mathbb R^6} e^{(|\alpha|+3|\beta|)\bar\phi}(e^{2\phi}+|p|^2)^{\ga}|\partial_x^\alpha\partial_p^\beta f|^2{\rm d}x{\rm d}p.
\end{align*}


In the setting of self-similar solutions, for any $\la\geq0$ and $N\geq0$, we define the energy functional by
\begin{align}
\tilde{\mathcal	E}^\lambda(t):=\tilde{\mathcal	E}^\lambda(g,\Phi)(t)=&\sum_{ \vert\alpha\vert+\vert\beta\vert\leq N}\int_{\mathbb R^6}(e^{2\Phi}+\vert q\vert^2)^\lambda\vert\partial_x^\alpha\partial_q^\beta g\vert^2{\rm d}x{\rm d}q\notag\\&+\sum_{\vert\alpha\vert\leq N}\int_{\mathbb R^3}(\vert\partial_t\partial_x^\alpha\Phi\vert^2+\vert\nabla_x\partial_x^\alpha\Phi\vert^2){\rm d}x,\notag
\end{align}
and the dissipation rate by
\begin{equation*}
	\begin{split}
\tilde{\mathcal D}^\lambda(t):=	\tilde{\mathcal	D}^\lambda(g,\Phi)(t)=\sum_{ \vert\alpha\vert+\vert\beta\vert\leq N}&\int_{\mathbb R^6}e^{\bar\phi}(e^{2\Phi}+\vert q\vert^2)^{\lambda-\frac{1}{2}}\\
&\qquad\times(e^{2\Phi}\vert\nabla_q\partial_x^\alpha\partial_q^\beta g\vert^2+\vert q\cdot\nabla_q\partial_x^\alpha\partial_q^\beta g\vert^2){\rm d}x{\rm d}q,
\end{split}
\end{equation*}
as well as the enhanced dissipation rate
	\begin{align*}
		\bar{\tilde {\mathcal D}}^\gamma(t):=\tilde{\mathcal	D}^\gamma(g,\Phi)(t)+\sum_{|\al|+|\bet|\leq N}\int_{\mathbb R^6} (e^{2\Phi}+|q|^2)^{\la}|\partial_x^\alpha\partial_q^\beta g|^2{\rm d}x{\rm d}q.
\end{align*}
Similarly, we also define the $(m+n)$-th order energy functional
\begin{equation*}
	\begin{split}
\tilde{\mathcal	E}^\lambda_{m,n}(t):=	\tilde{\mathcal	E}^\lambda_{m,n}(g,\Phi)(t)=&\sum_{\vert\alpha\vert=m,\vert\beta\vert=n}\int_{\mathbb R^6}(e^{2\Phi}+\vert q\vert^2)^\lambda\vert\partial_x^\alpha\partial_q^\beta g\vert^2{\rm d}x{\rm d}q\\
&+\sum_{\vert\alpha\vert=m}\int_{\mathbb R^3}(\vert\partial_t\partial_x^\alpha\Phi\vert^2+\vert\nabla_x\partial_x^\alpha\Phi\vert^2){\rm d}x,
\end{split}
\end{equation*}
and the corresponding dissipation rate
\begin{align}
\tilde{\mathcal D}_{m,n}^\lambda(t):=	\tilde{\mathcal	D}^\lambda_{m,n}(g,\Phi)(t)=&\sum_{\vert\alpha\vert=m,\vert\beta\vert=n}\int_{\mathbb R^6}e^{\bar\phi}(e^{2\Phi}+\vert q\vert^2)^{\lambda-\frac{1}{2}}\notag\\&\qquad\qquad\qquad\times(e^{2\Phi}\vert\nabla_q\partial_x^\alpha\partial_q^\beta g\vert^2+\vert q\cdot\nabla_q\partial_x^\alpha\partial_q^\beta g\vert^2){\rm d}x{\rm d}q,\notag
\end{align}
and the enhanced dissipation rate
\begin{align*}
\bar{\tilde{\mathcal D}}^\la_{m,n}=\tilde{\mathcal D}	^\la_{m,n}+\sum_{|\alpha|=m,|\beta|=n}\int_{\mathbb R^6} (e^{2\Phi}+|q|^2)^{\la}|\partial_x^\alpha\partial_q^\beta g|^2{\rm d}x{\rm d}q.
	\end{align*}
The first result about the global stability of the Cauchy problem of \eqref{pt-vnfp} and \eqref{pt-id} is stated as follows.
\begin{theorem}\label{thm1}
Let $0<\delta_1<\frac{1}{2}$ and $\delta_2>\frac{1}{2}$. Assume that $F_0(x,p)=\bar F_{in}(p)+f_0(x,p)\geq0$. There exists a constant $\eps_0>0$ such that if $$\sqrt{\mathcal E^{\delta_2}(f,\Phi)(0)+\left\|\Phi_0\right\|_{L^2_x}^2}\leq\eps_0,$$
then the Cauchy problem \eqref{pt-vnfp} and \eqref{pt-id} admits a unique global solution $(f(t,x,p),\Phi(t,x))$ satisfying $F(t,x,p)=\bar F(t,p)+f(t,x,p)\geq0$. Moreover, it holds that
\begin{equation}\label{eng-tt}
	\begin{split}
\mathcal E^{\delta_1}(t)&+\mathcal E^{\delta_2}(t)+\int_0^t(\bar{\mathcal D}^{\delta_1}(s)+\bar{\mathcal D}^{\delta_2}(s)) {\rm d}s
\leq C(\mathcal E^{\delta_2}(0)+\left\|\Phi_0\right\|_{L^2_x}^2),
\end{split}
\end{equation}
for all $t>0$. In addition, if $\delta_2\geq\frac{3}{2}$ and $\left\|(1+|p|^2)^{\frac{\delta_2}{2}}\nabla_p^4f_0\right\|_{H^1_{x,p}}\leq\eps_0$, then
\begin{equation}\label{long-t}
\left\| e^{5\bar\phi} f\right\|_{L^\infty_{x,p}}=\left\| e^{5\bar\phi}( F-\bar F)\right\|_{L^\infty_{x,p}}\rightarrow0\ as\ t\rightarrow\infty.
\end{equation}
\end{theorem}
Two remarks are necessary regarding the relationship between $\mathcal E^{\delta_1}$ and $\mathcal E^{\delta_2}$ in relation to the initial data and the additional condition on $f_0$ when deriving the large time behavior \eqref{long-t}.
\begin{remark}
Note that $|\bar{\phi}(0)|=|\phi_{in}|$ is finite. Consequently $|\phi_0|$ is also finite so that  we have
 $\mathcal E^{\delta_1}(0)\leq C\mathcal E^{\delta_2}(0)$.
\end{remark}

\begin{remark}
The integrability of the time derivative of the solution is utilized in deriving the large time behavior \eqref{long-t}. This requires a compatibility condition for the time derivative of solution at $t=0$, which is given by
\begin{equation*}
		\begin{split}
			\pa_t f_0\eqdef\pa_t f(0,x,p)&=\nabla_x\sqrt{e^{2\phi_0}+\vert p\vert^2}\cdot\nabla_p f_0-\nabla_p\sqrt{e^{2\phi_0}+\vert p\vert^2}\cdot\nabla_x f_0+\nabla_x\sqrt{e^{2\phi_0}+\vert p\vert^2}\cdot\nabla_p\bar F_0\\
			&\quad+e^{2\phi_0}\nabla_p\cdot(\Lambda_{\phi_0,p}\nabla_pf_0)+(e^{2\phi_0}\nabla_p\cdot(\Lambda_{\phi_0,p}\nabla_p\bar F_0)-e^{2\bar\phi_0}\nabla_p\cdot(\Lambda_{\bar\phi_0,p}\nabla_p\bar F_0)).
		\end{split}
	\end{equation*}
\end{remark}

The second result stated below is about the global stability of the Cauchy problem for \eqref{g-eq} and \eqref{g-id}.
\begin{theorem}\label{thm2}
Let $\frac{1}{2}<\lambda<\frac{3}{2}$ and $N\geq3$.
	Assume that $F_0(x,e^{\phi_{in}}q)=\bar F_{in}(e^{\phi_{in}}q)+e^{-3\phi_{in}}g_0(x,q)\geq0$.
There exists a constant $\eps_1>0$ such that if
$$\sqrt{\tilde{\mathcal E}^\lambda(g,\Phi)(0)+\left\|\Phi_0\right\|_{L^2_x}^2}\leq \eps_1,$$
then the Cauchy problem of \eqref{g-eq} and \eqref{g-id} admits a unique global solution $(g(t,x,q),\Phi(t,x))$ satisfying
$F(t,x,e^{\bar{\phi}}q)=\bar F (t,e^{\bar{\phi}}q)+e^{-3\bar{\phi}}g(t,x,q)\geq0$. Moreover, it holds that
	\begin{equation}\label{eng-es-ss}
		\tilde{\mathcal E}^{\lambda}(g,\Phi)(t)+\int_0^t\tilde{\mathcal D}^{\lambda}(g,\Phi)(s){\rm d}s\leq C(\tilde{\mathcal E}^{\lambda}(g,\Phi)(0)+\left\|\Phi_0\right\|_{L^2_x}^2),
	\end{equation}
	and
	\begin{equation}\label{eng-de-ss}
		\sum_{\vert\alpha\vert+\vert\beta\vert\leq N}\int_{\mathbb R^6}(e^{2\Phi}+\vert q\vert^2)^\lambda\vert\partial_x^\alpha\partial_q^\beta g\vert^2{\rm d}x{\rm d}q\leq	Ce^{\frac{1}{2}\bar\phi(\frac{3}{2}-\lambda)},
	\end{equation}
	for all $t>0$.
\end{theorem}

\subsection{Related works}
We now give a brief review of the mathematical studies on the VNFP system and related models.
Nordstr\"{o}m gravity, an alternative theory of gravity introduced in 1912-1913 \cite{ANH}, uses a scalar field $\phi$ to describe the gravitation   in the sense \eqref{vnfp-ih}$_2$.

When diffusion is ignored, the VNFP system reduces to the Vlasov-Nordstr\"{o}m (VN) system. This is a simplified, toy model where the dynamics of matter are described by the Vlasov equation, but gravitational force between the particles is mediated by a scalar field.
On the contrary, the Vlasov-Einstein system is the physical model within the framework of general relativity. However,  few mathematical results are known for the later system. For example,  the global spherically symmetric solution with small data was established in \cite{AKR-2011,RR-cmp}, which, unfortunately, rules out the propagation of gravitational waves, an intriguing feature of general relativity, cf. \cite{CS03}.
The first mathematical study of the VN system was initiated by Calogero \cite{CS03}, where the VN system was formally derived and the existence of static solutions was constructed. The stability of these static solutions was subsequently investigated in \cite{CSS-09-arma}. Later, the Vlasov-Poisson limit of the VN system was discussed in \cite{CL-04}. Around the same time, classical solutions were obtained in \cite{ CS06,CR-03-cpde}, while weak solutions were established in \cite{CR-04}.
From a more physical perspective, Bauer-Kunze-Rein-Rendall derived a radiation formula analogous to the Einstein quadrupole formula in general relativity, as presented in \cite{BKRR-06}. Recently, global small smooth solutions have been investigated by using a vector field method in \cite{FJS-17}, \cite{FJS-apde}, and \cite{Wang-18}.

The VNFP system was formally derived in \cite{AJ-CS-11}. The first mathematical result by  Alc\'{a}ntara-Calogero-Pankavich in \cite{ACP-2014}  addressed the spatially homogeneous case and established the nontrivial large-time behavior as shown in \eqref{lt-ho}.
Recently, by considering the friction force and a constant background stress energy tensor, global
existence of smooth solution near equilibrium was  proved in \cite{DLY-SIAM} by  energy method. A lot of progress has also been made on the spatially inhomogeneous kinetic equations coupled with Poisson or Maxwell system, see \cite{BJ,D-VML,DL-cmp,Guo-JAMS-2011,YangYu-VMFP} and the references therein. However, to our knowledge, there are no existing results on the spatially inhomogeneous VNFP system. This paper aims to fill in this gap. We expect that the approaches introduced here will lead to further study on the VNFP system, such as inviscid limits and Landau damping.

\subsection{Strategies of the proof}\label{sb-si}
We now  present the basic strategies and ideas used in the proof of Theorem \ref{thm1} and \ref{thm2}.
\begin{itemize}
\item \underline{Perturbation framework.}
Due to the delicate transport terms, the spatially inhomogeneous VNFP system is lack of a nice structure. In the homogeneous case, the non-autonomous heat-type equation
$$
\partial_t\bar F=e^{2\bar\phi}\nabla_p\cdot(\Lambda_{\bar\phi,p}\nabla_p\bar F),\ t>0,\ p\in\R^3,\ \int_0^\infty e^{2\bar\phi}{\rm d}t<+\infty
$$
induces a nontrivial solution. Additionally, the second-order ODE
$$
\frac{d^2\bar\phi}{dt^2}=-e^{2\bar\phi}\int_{\mathbb R^3}\frac{\bar F}{\sqrt{e^{2\bar\phi}+\vert p\vert^2}}{\rm d}p
$$
provides a nontrivial gravitational field as indicated by \eqref{nt-phi}. To address \eqref{vnfp-ih}, the starting point is to show  that the solution of \eqref{vnfp-ih} converges to the one  of \eqref{vnpf-ho} in the perturbation framework.

\item \underline{A refined energy method.} The energy functionals in  the form \eqref{eng-def}, incorporating two kinds of momentum weights, are designed to handle the singularities caused by the stress-energy tensor. Precisely, for large time $t>0$, we can cope with these singularities as follows:
\begin{align}\label{el-sg}
		\int_{\mathbb R^3}\frac{f}{\sqrt{{e^{2\phi}+\vert p\vert^2}}}{\rm d}p&=\int_{\sqrt{{e^{2\phi}+\vert p\vert^2}}\leq1}\frac{f}{\sqrt{{e^{2\phi}+\vert p\vert^2}}}{\rm d}p+\int_{\sqrt{{e^{2\phi}+\vert p\vert^2}}>1}\frac{f}{\sqrt{{e^{2\phi}+\vert p\vert^2}}}{\rm d}p\nonumber\\
		&\leq \int_{\vert p\vert\leq1}\frac{1}{\vert p\vert^{1+\delta_1}}(e^{2\phi}+\vert p\vert^2)^{\frac{\delta_1}{2}}\vert f\vert{\rm d}p+\int_{\vert p\vert^2>\frac{1}{2}}\frac{1}{\vert p\vert^{1+\delta_2}}(e^{2\phi}+\vert p\vert^2)^{\frac{\delta_2}{2}}\vert f\vert{\rm d}p\nonumber\\
		&\leq C\left\|(e^{2\phi}+\vert p\vert^2)^{\frac{\delta_1}{2}}f\right\|_{L^2_p}+C\left\|(e^{2\phi}+\vert p\vert^2)^{\frac{\delta_2}{2}}f\right\|_{L^2_p},
\end{align}
where $0<\delta_1<\frac{1}{2}$ and $\delta_2>\frac{1}{2}$  are required.
On the other hand, we  know in {\it a priorily} that the gravitational field $\phi$ satisfies  \eqref{nt-phi}, leading to vanishing diffusion on the right-hand side of \eqref{vnfp-ih} when time is large. Consequently, both the transport terms $\nabla_p\sqrt{e^{2\phi}+\vert p\vert^2}\cdot\nabla_x f$ and $\nabla_x\sqrt{e^{2\phi}+\vert p\vert^2}\cdot\nabla_p f$  are difficult to control. Inspired by \cite{S. Chaturvedi}, the weights of the power of $e^{\bar{\phi}}$ depend on the order of the $x$ and $p$ derivatives. However, unlike \cite{S. Chaturvedi} where a new dissipation quantity
\begin{equation*}
	\int_{\mathbb R^6}\vert\partial_{x_i}f\vert^2{\rm d}x{\rm d}p
\end{equation*}
is used,  we indeed use the damping structure
\begin{equation}
	\int_{\mathbb R^6}\partial_t\partial_{x_i}f e^{\bar\phi}\partial_{x_i} f{\rm d}x{\rm d}p=\frac{1}{2}\frac{d}{dt}\int_{\mathbb R^6}e^{\bar\phi}\vert\partial_{x_i}f\vert^2{\rm d}x{\rm d}p
-\frac{1}{2}\bar\phi'\int_{\mathbb R^6}e^{\bar\phi}\vert\partial_{x_i}f\vert^2{\rm d}x{\rm d}p. \label{xd-dap}
\end{equation}
Precisely, this requires a $e^{\frac{\bar\phi}{2}}$ weight for each $x$ derivative and a $e^{\frac{3\bar\phi}{2}}$ for momentum derivative in the energy functional \eqref{eng-def}. And this is in consistent with the behavior of Fokker-Planck type operators, cf.\cite{BJ, BJD-22, S. Chaturvedi}.

\item \underline{Time splitting.}
By using the property of $\bar{\phi}$ and the {\it a priori} energy assumption
\begin{equation*}
	\sqrt{\mathcal E^{\delta_1}(t)+\mathcal E^{\delta_2}(t)}\leq2\eps_0,
\end{equation*}
there exists a time $T_c>0$ such that both $\bar{\phi}$ and $\phi$ become monotone
for $t>T_c$. So that the second term on the right hand side of \eqref{xd-dap} induces a damping effect.

\item \underline{Self-similar regime.} As showed in Subsection \ref{tp}, to cope with the singularities that arise in $(e^{2\bar\phi}+\vert p\vert^2)^{-1/2}$, we can utilize the self-similar solution of the perturbation equation \eqref{pt-vnfp} in the form of
$$
f(t,x,p)=e^{-3\bar\phi}g(t,x,e^{-\bar\phi}p)=e^{-3\bar\phi}g(t,x,q).
$$
In this setting, there exists a strong damping due to the monotonicity of $\bar\phi$. Hence, this significantly simplifies the proof for the corresponding energy estimates.

\end{itemize}

\noindent{\it Notations.} We list some notations and norms used in the paper. Throughout this paper,  $C$ denotes some generic positive (generally large) constant and $\la$ or $\ga$ denotes some generic positive constant. $D\lesssim E$ means that  there is a generic constant $C>0$
such that $D\leq CE$. $D\sim E$
means $D\lesssim E$ and $E\lesssim D$.
For multi-indices
$\al=[\al_1, \al_2, \al_3]$, we denote
$
\partial^{\al}_{x}=\partial_{x_{1}}^{\al_{1}}
\partial_{x_{2}}^{\al_{2}}\partial_{x_{3}}^{\al_{3}}
$
and likewise for $
\partial^{\beta}_{p}$, and the length of $\al$ is denoted by $|\al|=\al_1+\al_2+\al_3$.
$\al'\leq\al$ means that no component of $\al'$
is greater than the component of $\al$, and $\al'<\al$ means that
$\al'\leq\al$ and $|\al'|<|\al|$. In addition, we denote the indicator function of the set $A$ by  $\chi_A$.

The rest of this paper is arranged as follows.
Section \ref{gs-vnpf} is devoted to the global existence and long time behavior of the Cauchy problem \eqref{pt-vnfp} and \eqref{pt-id}, specifically for proving the Theorem \ref{thm1}.  The self-similar solution of \eqref{pt-vnfp} and \eqref{pt-id} is constructed in Section \ref{sec-sss}, where the proof of Theorem \ref{thm2} is presented. Finally, we put some useful estimates in Appendix \ref{ad}.

\section{Global stability of the VNFP system}\label{gs-vnpf}
In this section, we will prove that the solution of the VNFP system \eqref{vnfp-ih} converges to the one of the homogeneous VNFP in some time-weighted function space.  For this, the first step is to show the global existence of solution to the perturbed system
\begin{eqnarray}\label{vnfp-s2}
	\left\{
\begin{array}{l}
	\partial_t f+\nabla_p\sqrt{e^{2\phi}+\vert p\vert^2}\cdot\nabla_x f-\nabla_x\sqrt{e^{2\phi}+\vert p\vert^2}\cdot\nabla_p f=\nabla_x\sqrt{e^{2\phi}+\vert p\vert^2}\cdot\nabla_p\bar F+e^{2\phi}\nabla_p\cdot(\Lambda_{\phi,p}\nabla_pf)\\[2mm]
	\quad+(e^{2\phi}\nabla_p\cdot(\Lambda_{\phi,p}\nabla_p\bar F)-e^{2\bar\phi}\nabla_p\cdot(\Lambda_{\bar\phi,p}\nabla_p\bar F)),\ t>0,\ x\in\mathbb R^3,\ p\in\mathbb R^3,\\[2mm]
	\partial_t^2\Phi-\Delta_x\Phi=-e^{2\phi}\int_{\mathbb R^3}\frac{f}{\sqrt{e^{2\phi}+\vert p\vert^2}}{\rm d}p-\int_{\mathbb R^3}\Big(\frac{e^{2\phi}\bar F}{\sqrt{e^{2\phi}+\vert p\vert^2}}-\frac{e^{2\bar\phi}\bar F}{\sqrt{e^{2\bar\phi}+\vert p\vert^2}}\Big){\rm d}p,
\end{array}\right.
	\end{eqnarray}
and
\begin{align}\label{id-s2}
	(f(0,x,p),\Phi(0,x),\partial_t\Phi(0,x))=(f_0(x,p),\Phi_0(x),\Phi_1(x)),\ x\in\mathbb R^3,\ p\in\mathbb R^3.
\end{align}
And then we can study  the large-time behavior of solutions  as described in \eqref{long-t}.

The global existence will be obtained by the  continuity argument that  combines the local existence and the \textit{a priori} energy estimates \eqref{eng-tt}. The local existence will be briefly sketched in Subsection \ref{sec-loc}. Here, we mainly focus on proving the \textit{a priori} energy estimate \eqref{eng-tt} based on the following \textit{a priori} assumption:
\begin{equation}
\sqrt{\mathcal{E}^{\delta_1}(t) + \mathcal{E}^{\delta_2}(t)} \leq 2\epsilon_0, \ t\in[0,+\infty). \label{aps}
\end{equation}
As discussed in Subsection \ref{sb-si}, by \eqref{lt-ho} and \eqref{aps}, one can find that there exists a finite time $T_c>0$ such that
\begin{align}\label{tc}
\bar\phi'(t)\leq\frac{\bar\phi'(\infty)}{2},
\ \partial_t\phi\leq\frac{\bar\phi'(\infty)}{4},\ \textrm{and}\ e^{2\phi}<\frac{1}{2},\ \textrm{for}\ t\geq T_c.
\end{align}
Hence, the proof of  Theorem \ref{thm1} is divided in the following two steps.

\subsection{Energy estimates in $[0,T_c]$}\label{sec-ft-eng}
In this subsection, we will show that the weighted energy estimate \eqref{eng-tt} holds in the time interval $[0,T_c]$ in the following proposition, where $T_c$ is givn in  \eqref{tc}.
\begin{proposition}\label{loc-eng-pro}
Assume $[f,\Phi]$ is a solution to \eqref{vnfp-s2} and \eqref{id-s2} that satisfies \eqref{aps}, let $T_c$ be given in  \eqref{tc}, then there exists a constant $C(T_c)>0$, dependent on $T_c$, such that
\begin{align}\label{loc-eng}
\sum_{\substack{m+n\leq4\\n<4}}&	(\mathcal E_{m,n}^{\delta_1}(t)+\mathcal E_{m,n}^{\delta_2}(t))+\sum_{\substack{m+n\leq4\\n<4}}\int_0^t	(\mathcal D_{m,n}^{\delta_1}(s)+\mathcal D_{m,n}^{\delta_2}(s)){\rm d}s\notag\\
&\leq C(T_c)\Big(\sum_{\substack{m+n\leq4\\n<4}}	\mathcal E_{m,n}^{\delta_2}(0)+\left\|\Phi_0\right\|_{L^2_x}^2\Big),
\end{align}
for $t\in[0,T_c]$.
\end{proposition}

\begin{proof}
The proof is divided into following three steps.

\noindent\underline{{\it Step 1. Zeroth order energy estimate.}}
 By multiplying \eqref{vnfp-s2}$_1$ by $(e^{2\phi}+\vert p\vert^2)^\gamma f$ and  integrating over $(x,p)\in\mathbb R^3\times\mathbb R^3$, it holds
 \begin{align}\label{z-ip}
 \frac{1}{2}\frac{d}{dt}&\int_{\mathbb R^6}(e^{2\phi}+\vert p\vert^2)^\gamma\vert f\vert^2{\rm d}x{\rm d}p\nonumber\\
 &=\gamma\int_{\mathbb R^6}\partial_t\phi e^{2\phi}(e^{2\phi}+\vert p\vert^2)^{\gamma-1}\vert f\vert^2{\rm d}x{\rm d}p+\int_{\mathbb R^6}\nabla_x\sqrt{e^{2\phi}+\vert p\vert^2}\cdot\nabla_p\bar F(e^{2\phi}+\vert p\vert^2)^\gamma f{\rm d}x{\rm d}p\nonumber\\
 &\quad+\int_{\mathbb R^6}e^{2\phi}\nabla_p\cdot(\Lambda_{\phi,p}\nabla_p f)(e^{2\phi}+\vert p\vert^2)^\gamma f{\rm d}x{\rm d}p\nonumber\\
 &\quad+\int_{\mathbb R^6}(e^{2\phi}\nabla_p\cdot(\Lambda_{\phi,p}\nabla_p\bar F)-e^{2\bar\phi}\nabla_p\cdot(\Lambda_{\bar\phi,p}\nabla_p\bar F))(e^{2\phi}+\vert p\vert^2)^\gamma f{\rm d}x{\rm d}p\nonumber\\
 &:=I_1+I_2+I_3+I_4.
 \end{align}
 We now estimate $I_i$ $(1\leq i\leq4)$ term by term.
For $I_1$,  by noting that $\bar\phi'(t)$ is bounded, cf. \cite[Theorem 2.1, pp.3703]{ACP-2014}, and by   Sobolev's embedding $H^2_x\hookrightarrow L^\infty_x$, we have
\begin{equation}
	|I_1|\leq C\int_{\mathbb R^6}(e^{2\phi}+\vert p\vert^2)^\gamma\vert f\vert^2{\rm d}x{\rm d}p.\notag
\end{equation}
For $I_2$, by Lemma \ref{lemma A3}, it follows that
\begin{align}
|I_2|\leq Ce^{\bar\phi}\int_{\mathbb R^6}(e^{2\phi}+\vert p\vert^2)^\gamma\vert f\vert^2{\rm d}x{\rm d}p+Ce^{\bar\phi}\left\|\nabla_x\Phi\right\|_{L^2_x}^2.\notag
\end{align}
After integrating by parts, $I_3$ can be estimated as follows
\begin{align}
	I_3=&-\int_{\mathbb R^6}e^{2\phi}(e^{2\phi}+\vert p\vert^2)^{\gamma-\frac{1}{2}}(e^{2\phi}\vert\nabla_p f\vert^2+\vert p\cdot\nabla_p f\vert^2){\rm d}x{\rm d}p\notag\\&-2\gamma\int_{\mathbb R^6}e^{2\phi}(e^{2\phi}+\vert p\vert^2)^{\gamma-\frac{1}{2}}p\cdot\nabla_pff{\rm d}x{\rm d}p\nonumber\\
	\leq&-\int_{\mathbb R^6}e^{2\phi}(e^{2\phi}+\vert p\vert^2)^{\gamma-\frac{1}{2}}(e^{2\phi}\vert\nabla_p f\vert^2+\vert p\cdot\nabla_p f\vert^2){\rm d}x{\rm d}p\nonumber\\
	&+C\int_{\mathbb R^6}e^\phi(e^{2\phi}+\vert p\vert^2)^{\frac{\gamma}{2}-\frac{1}{4}}\vert p\cdot\nabla_pf\vert e^{\frac{\phi}{2}}(e^{2\phi}+\vert p\vert^2)^{\frac{\gamma}{2}-\frac{1}{4}}\vert f\vert e^{\frac{\bar\phi}{2}}{\rm d}x{\rm d}p\nonumber\\
	\leq&-(1-\eta)\int_{\mathbb R^6}e^{2\phi}(e^{2\phi}+\vert p\vert^2)^{\gamma-\frac{1}{2}}(e^{2\phi}\vert\nabla_p f\vert^2+\vert p\cdot\nabla_p f\vert^2){\rm d}x{\rm d}p
\notag\\&+C_\eta e^{\bar\phi}\int_{\mathbb R^6}(e^{2\phi}+\vert p\vert^2)^\gamma\vert f\vert^2{\rm d}x{\rm d}p,\notag
\end{align}
where we have used the folllowing basic identity
\begin{align}\label{m-lb}
\nabla_p f\Lambda_{\phi,p}[\nabla_p f]^T=(e^{2\phi}+\vert p\vert^2)^{-\frac{1}{2}}(e^{2\phi}\vert\nabla_p f\vert^2+\vert p\cdot\nabla_p f\vert^2).
\end{align}
For the last term $I_4$, we use Lemma \ref{lemma A5} with $\vert\alpha\vert=\vert\beta\vert=0$  to get
\begin{align}
|I_4|\leq& \eta\int_{\mathbb R^6}e^{2\phi}(e^{2\phi}+\vert p\vert^2)^{\gamma-\frac{1}{2}}(e^{2\phi}\vert\nabla_p f\vert^2+\vert p\cdot\nabla_p f\vert^2){\rm d}x{\rm d}p+Ce^{\bar\phi}\int_{\mathbb R^6}(e^{2\phi}+\vert p\vert^2)^\gamma\vert f\vert^2{\rm d}x{\rm d}p
\notag\\&+C_\eta e^{\bar\phi}\left\|\Phi\right\|_{L^2_x}^2.\notag
\end{align}
Combining all these  estimates gives
\begin{align}\label{z-f}
\frac{d}{dt}\int_{\mathbb R^6}&(e^{2\phi}+\vert p\vert^2)^\gamma\vert f\vert^2{\rm d}x{\rm d}p+\int_{\mathbb R^6}e^{2\phi}(e^{2\phi}+\vert p\vert^2)^{\gamma-\frac{1}{2}}(e^{2\phi}\vert\nabla_p f\vert^2+\vert p\cdot\nabla_p f\vert^2){\rm d}x{\rm d}p\notag\\
&	\leq C\mathcal E_{0,0}^\gamma(t)+C_\eta e^{\bar\phi}\left\|\Phi\right\|_{L^2_x}^2.
\end{align}
Next we turn to estimate the  gravitational field $\Phi$. Multiplying \eqref{vnfp-s2}$_2$ by $\partial_t\Phi$, and integrating over $\mathbb R^3$ yield
\begin{align}\label{phi-ip}
\frac{1}{2}\frac{d}{dt}&\int_{\mathbb R^3}(\vert\partial_t\Phi\vert^2+\vert\nabla_x\Phi\vert^2){\rm d}x\nonumber\\
		=&-\int_{\mathbb R^3}e^{2\phi}\int_{\mathbb R^3}\frac{f}{\sqrt{e^{2\phi}+\vert p\vert^2}}{\rm d}p\partial_t\Phi{\rm d}x-\int_{\mathbb R^3}\int_{\mathbb R^3}\Big(e^{2\phi}\frac{\bar F}{\sqrt{e^{2\phi}+\vert p\vert^2}}-e^{2\bar\phi}\frac{\bar F}{\sqrt{e^{2\bar\phi}+\vert p\vert^2}}\Big){\rm d}p\partial_t\Phi{\rm d}x\nonumber\\	:=&J_1+J_2.
\end{align}
From \eqref{tc} and H\"{o}lder's inequality, it follows that
\begin{equation*}
	\int_{\mathbb R^3}\frac{f}{\sqrt{e^{2\phi}+\vert p\vert^2}}{\rm d}p=\int_{\mathbb R^3}({e^{2\phi}+\vert p\vert^2})^{-\frac{1}{2}-\frac{\delta_2}{2}}({e^{2\phi}+\vert p\vert^2})^{\frac{\delta_2}{2}}f{\rm d}p\leq C\left\|({e^{2\phi}+\vert p\vert^2})^{\frac{\delta_2}{2}}f\right\|_{L^2_p},
\end{equation*}
here we have chosen $\delta_2>\frac{1}{2}$. Hence, the first term on the right-hand side of \eqref{phi-ip} can be bounded as
\begin{equation*}
	|J_1|\leq Ce^{2\bar\phi}\left\|\partial_t\Phi\right\|_{L^2_x}\left\|\int_{\mathbb R^3}\frac{f}{\sqrt{e^{2\phi}+\vert p\vert^2}}{\rm d}p\right\|_{L^2_x}\\
	\leq Ce^{2\bar\phi}(\left\|\partial_t\Phi\right\|_{L^2_x}^2+\left\|(e^{2\phi}+\vert p\vert^2)^{\frac{\delta_2}{2}}f\right\|_{L^2_{x,p}}^2).
\end{equation*}
Similarily, for $J_2$, we have
\begin{equation*}
	\begin{split}
		|J_2|&\leq C\int_{\mathbb R^3}\int_{\mathbb R^3}\vert\Phi\vert\frac{e^{2\bar\phi}}{\sqrt{e^{2\bar\phi}+\vert p\vert^2}}\vert\bar F\vert{\rm d}p\vert\partial_t\Phi\vert{\rm d}x\\
		&\leq Ce^{2\bar\phi}\Big(\left\|\partial_t\Phi\right\|_{L^2_x}^2+\left\|\Phi\right\|_{L^2_x}^2\int_{\mathbb R^3}(e^{2\bar\phi}+\vert p\vert^2)^{\delta_2}\vert\bar F\vert^2{\rm d}p\Big)\leq Ce^{2\bar\phi}(\left\|\partial_t\Phi\right\|_{L^2_x}^2+\left\|\Phi\right\|_{L^2_x}^2).	
	\end{split}
\end{equation*}
Substituting the above estimates into \eqref{phi-ip}, we get
\begin{equation}
	\frac{d}{dt}\int_{\mathbb R^3}(\vert\partial_t\Phi\vert^2+\vert\nabla_x\Phi\vert^2){\rm d}x\leq Ce^{2\bar\phi}\mathcal E_{0,0}^{\delta_2}(t)+Ce^{2\bar\phi}\left\|\Phi\right\|_{L^2_x}^2,\notag
\end{equation}
which combines with \eqref{z-f} gives
\begin{equation}\label{z-sum1}
\frac{\rm d}{dt}\mathcal E_{0,0}^\gamma(t)+\mathcal D_{0,0}^\gamma(t)\leq C\mathcal E_{0,0}^\gamma(t)+C\mathcal E_{0,0}^{\delta_2}(t)+Ce^{\bar\phi}\left\|\Phi\right\|_{L^2_x}^2.
\end{equation}
Letting $\de_1\in(0,\frac{1}{2})$ and $\de_2\in(\frac{1}{2},+\infty)$, and further
taking $\gamma=\delta_1$ and $\delta_2$ in \eqref{z-sum1} respectively, and summing the resulting inequalities, we get
\begin{equation}
	\frac{\rm d}{dt}(\mathcal E_{0,0}^{\delta_1}(t)+\mathcal E_{0,0}^{\delta_2}(t))+\mathcal D_{0,0}^{\delta_1}(t)+\mathcal D_{0,0}^{\delta_2}(t)\leq C(\mathcal E_{0,0}^{\delta_1}(t)+\CE_{0,0}^{\delta_2}(t))+Ce^{\bar\phi}\left\|\Phi\right\|_{L^2_x}^2.\notag
\end{equation}
It remains now to control $\left\|\Phi\right\|_{L^2_x}$, by H\"{o}lder's inequality and Minkowski's inequality, one has
\begin{equation}\label{phi-l2}
	\left\|\Phi(t)\right\|_{L^2_x}^2=\left\|\Phi_0+\int_0^t\partial_s\Phi(s){\rm d}s\right\|_{L^2_x}^2\leq\left\|\Phi_0\right\|_{L^2_x}^2+t\int_0^t\left\|\partial_s\Phi\right\|_{L^2_x}^2{\rm d}s.
\end{equation}
Finally, by Gr\"{o}nwall's inequality, we get from \eqref{z-sum1} and \eqref{phi-l2} that for $t\in[0,T_c]$
\begin{equation}\label{z-sum3}
	\begin{split}
\mathcal E_{0,0}^{\delta_1}(t)+\mathcal E_{0,0}^{\delta_2}(t)+\int_0^t(\mathcal D_{0,0}^{\delta_1}(s)
+\mathcal D_{0,0}^{\delta_2}(s)){\rm d}s&\leq	C(T_c)(\mathcal E_{0,0}^{\delta_1}(0)+E_{0,0}^{\delta_2}(0))+C\left\|\Phi_0\right\|_{L^2_x}^2\\
&\leq C(T_c)(\mathcal E_{0,0}^{\delta_2}(0)+\left\|\Phi_0\right\|_{L^2_x}^2).
\end{split}
\end{equation}

\noindent\underline{{\it Step 2. First order energy estimates.}}
By applying $\partial_{p_i}$ to \eqref{vnfp-s2}$_1$, multiplying the equation by $e^{3\bar\phi}(e^{2\phi}+\vert p\vert^2)^\gamma\partial_{p_i}f$, and then integrating over $(x,p)\in\mathbb R^3\times\mathbb R^3$, it holds that
\begin{align}\label{2.12}
\frac{1}{2}\frac{d}{dt}&\int_{\mathbb R^6}e^{3\bar\phi}(e^{2\phi}+\vert p\vert^2)^\gamma\vert\partial_{p_i}f\vert^2{\rm d}x{\rm d}p\notag\\
=&\frac{1}{2}\int_{\mathbb R^6}\vert\partial_{p_i}f\vert^2\partial_t(e^{3\bar\phi}(e^{2\phi}+\vert p\vert^2)^\gamma){\rm d}x{\rm d}p\nonumber\\
&+\int_{\mathbb R^6}(\nabla_x\partial_{p_i}\sqrt{e^{2\phi}+\vert p\vert^2}\cdot\nabla_pf-\nabla_p\partial_{p_i}\sqrt{e^{2\phi}+\vert p\vert^2}\cdot\nabla_xf)e^{3\bar\phi}(e^{2\phi}+\vert p\vert^2)\partial_{p_i}f{\rm d}x{\rm d}p\nonumber\\
&+\int_{\mathbb R^6}\partial_{p_i}(\nabla_x\sqrt{e^{2\phi}+\vert p\vert^2}\cdot\nabla_p\bar F)e^{3\bar\phi}(e^{2\phi}+\vert p\vert^2)^\gamma \partial_{p_i}f{\rm d}x{\rm d}p\nonumber\\
	&+\int_{\mathbb R^6}\partial_{p_i}(e^{2\phi}\nabla_p\cdot(\Lambda_{\phi,p}\nabla_p f))e^{3\bar\phi}(e^{2\phi}+\vert p\vert^2)^\gamma \partial_{p_i}f{\rm d}x{\rm d}p\nonumber\\
	&+\int_{\mathbb R^6}\partial_{p_i}(e^{2\phi}\nabla_p\cdot(\Lambda_{\phi,p}\nabla_p\bar F)-e^{2\bar\phi}\nabla_p\cdot(\Lambda_{\bar\phi,p}\nabla_p\bar F))e^{3\bar\phi}(e^{2\phi}+\vert p\vert^2)^\gamma \partial_{p_i}f{\rm d}x{\rm d}p\nonumber\\
	:=&I_5+I_6+I_7+I_8+I_9.
	\end{align}
We now estimate $I_i$ $(5\leq i\leq9)$ as follows.
Similar to $I_1$, $ I_5$ can be bounded as
\begin{equation}\label{2.13}
	|I_5|\leq C\int_{\mathbb R^6}e^{3\bar\phi}(e^{2\phi}+\vert p\vert^2)^\gamma\vert\partial_{p_i}f\vert^2{\rm d}x{\rm d}p.\notag
\end{equation}
For $I_6$ and $I_7$, by using H\"{o}lder's inequality and Sobolev's imbedding $H_x^2\hookrightarrow L^\infty_x$, we have
\begin{align*}
|I_6|&\leq C\left\|\nabla_x\Phi\right\|_{L^\infty_x}\int_{\mathbb R^6}e^{3\bar\phi}(e^{2\phi}+\vert p\vert)^\gamma\vert\nabla_p f\vert\vert\partial_{p_i}f\vert{\rm d}x{\rm d}p\\
&\quad+C\int_{\mathbb R^6}e^{3\bar\phi}(e^{2\phi}+\vert p\vert^2)^{\gamma-\frac{1}{2}}\vert\nabla_xf\vert\vert\partial_{p_i}f\vert{\rm d}x{\rm d}p\\
&\leq C\int_{\mathbb R^6}e^{\bar\phi}(e^{2\phi}+\vert p\vert^2)^\gamma\vert\nabla_xf\vert^2{\rm d}x{\rm d}p+C\int_{\mathbb R^6}e^{3\bar\phi}(e^{2\phi}+\vert p\vert^2)^\gamma\vert\nabla_pf\vert^2{\rm d}x{\rm d}p,\notag
\end{align*}
and
\begin{equation*}
|I_7|
\leq Ce^{\bar\phi}\int_{\mathbb R^6}e^{3\bar\phi}(e^{2\phi}+\vert p\vert^2)^\gamma\vert\partial_{p_i}f\vert^2{\rm d}x{\rm d}p+Ce^{\bar\phi}\left\|\nabla_x\Phi\right\|_{L^2_x}^2,
\end{equation*}
where we have used Lemma \ref{lemma A3} with $\vert\alpha\vert=0$ and $\vert\beta\vert=1$.

For  $I_8$, by applying \eqref{m-lb} and Lemma \ref{lemma A4} with $\vert\alpha\vert=0$ and $\vert\beta\vert=1$, we get
\begin{equation}
	\begin{split}
I_8&\leq-(1-\eta)\int_{\mathbb R^6}e^{3\bar\phi+2\phi}(e^{2\phi}+\vert p\vert^2)^{\gamma-\frac{1}{2}}(e^{2\phi}\vert\nabla_p\partial_{p_i}f\vert^2+\vert p\cdot\nabla_p\partial_{p_i}f\vert^2){\rm d}x{\rm d}p\\
&\quad+C\mathcal D_{0,0}^\gamma(t)+C_\eta e^{\bar\phi}\sum_{m+n\leq 1}\mathcal E_{m,n}^\gamma(t).\notag
\end{split}
\end{equation}
For $I_9$, by Lemma \ref{lemma A5} with $\vert\alpha\vert=0$ and $\vert\beta\vert=1$, it follows directly that
\begin{align}
 |I_9|&\leq\eta\int_{\mathbb R^6}e^{3\bar\phi+2\phi}(e^{2\phi}+\vert p\vert^2)^{\gamma-\frac{1}{2}}(e^{2\phi}\vert\nabla_p\partial_{p_i}f\vert^2+\vert p\cdot\nabla_p\partial_{p_i}f\vert^2{\rm d}x{\rm d}p\nonumber\\
 &\quad+Ce^{\bar\phi}\int_{\mathbb R^6}e^{3\bar\phi}(e^{2\phi}+\vert p\vert^2)^\gamma\vert\partial_{p_i}f\vert^2{\rm d}x{\rm d}p+C_\eta e^{\bar\phi}\left\|\Phi\right\|_{L^2_x}^2.\notag
\end{align}

Now we turn to deduce the corresponding estimates involving $x$ derivatives.
For this, by applying $\partial_{x_i}$ to \eqref{vnfp-s2}$_1$, and multiplying the equation by $e^{\bar\phi}(e^{2\phi}+\vert p\vert^2)^\gamma\partial_{x_i}f$, and then integrating the resulting equality with respect to $(x,p)\in\mathbb R^3\times\mathbb R^3$, it  yields that
\begin{align}\label{xd-ip}
		\frac{1}{2}\frac{d}{dt}\int_{\mathbb R^6}&e^{\bar\phi}(e^{2\phi}+\vert p\vert^2)^\gamma\vert\partial_{x_i}f\vert^2{\rm d}x{\rm d}p\nonumber\\
		=&\frac{1}{2}\int_{\mathbb R^6}\vert\partial_{x_i}f\vert^2\partial_t(e^{\bar\phi}(e^{2\phi}+\vert p\vert^2)^\gamma){\rm d}x{\rm d}p\nonumber\\
		&+\int_{\mathbb R^6}\big(\nabla_x\partial_{x_i}\sqrt{e^{2\phi}+\vert p\vert^2}\cdot\nabla_pf-\nabla_p\partial_{x_i}\sqrt{e^{2\phi}+\vert p\vert^2}\cdot\nabla_xf\big)e^{\bar\phi}(e^{2\phi}+\vert p\vert^2)^\gamma\partial_{x_i}f{\rm d}x{\rm d}p\nonumber\\
		&+\int_{\mathbb R^6}\nabla_x\partial_{x_i}\sqrt{e^{2\phi}+\vert p\vert^2}\cdot\nabla_p\bar Fe^{\bar\phi}(e^{2\phi}+\vert p\vert^2)^\gamma \partial_{x_i}f{\rm d}x{\rm d}p\nonumber\\&+\int_{\mathbb R^6}\partial_{x_i}(e^{2\phi}\nabla_p\cdot(\Lambda_{\phi,p}\nabla_p f))e^{\bar\phi}(e^{2\phi}+\vert p\vert^2)^\gamma \partial_{x_i}f{\rm d}x{\rm d}p\nonumber\\
		&+\int_{\mathbb R^6}\partial_{x_i}(e^{2\phi}\nabla_p\cdot(\Lambda_{\phi,p}\nabla_p\bar F))e^{\bar\phi}(e^{2\phi}+\vert p\vert^2)^\gamma \partial_{x_i}f{\rm d}x{\rm d}p\nonumber\\
		:=&I_{10}+I_{11}+I_{12}+I_{13}+I_{14}.
	\end{align}
Now we estimate the right-hand side of \eqref{xd-ip} term by term. For $I_{10}$, we directly have
\begin{equation}\label{2.19}
	|I_{10}|\leq C\int_{\mathbb R^6}e^{\bar\phi}(e^{2\phi}+\vert p\vert^2)^\gamma\vert\partial_{x_i}f\vert^2{\rm d}x{\rm d}p.\notag
\end{equation}
To estimate $I_{11}$, we use H\"{o}lder's inequality to obtain
\begin{align}\label{2.20}
		|I_{11}|\leq& C\int_{\mathbb R^6}e^{\frac{3\bar\phi}{2}}(e^{2\phi}+\vert p\vert^2)^{\frac{\gamma}{2}}\vert\nabla_pf\vert e^{\frac{\bar\phi}{2}}(e^{2\phi}+\vert p\vert^2)^{\frac{\gamma}{2}}\vert\partial_{x_i}f\vert(\vert\nabla_x\Phi\vert^2+\vert\nabla_x^2\Phi\vert){\rm d}x{\rm d}p\nonumber\\
		&+C\int_{\mathbb R^6}e^{\bar\phi}(e^{2\phi}+\vert p\vert^2)^\gamma\vert\nabla_xf\vert^2\vert\nabla_x\Phi\vert{\rm d}x{\rm d}p\nonumber\\
		\leq& C\left\|(\nabla_x\Phi,\nabla_x^2\Phi)\right\|_{L^\infty_x}\int_{\mathbb R^6} e^{\bar\phi}(e^{2\phi}+\vert p\vert^2)^{\gamma}\vert\nabla_xf\vert^2{\rm d}x{\rm d}p\notag\\&+C\left\|(\nabla_x\Phi,\nabla_x^2\Phi)\right\|_{L^\infty_x}\int_{\mathbb R^6} e^{3\bar\phi}(e^{2\phi}+\vert p\vert^2)^\gamma\vert\nabla_pf\vert^2{\rm d}x{\rm d}p\nonumber\\
		\leq& C\int_{\mathbb R^6} e^{\bar\phi}(e^{2\phi}+\vert p\vert^2)^{\gamma}\vert\nabla_xf\vert^2{\rm d}x{\rm d}p+C\int_{\mathbb R^6} e^{3\bar\phi}(e^{2\phi}+\vert p\vert^2)^\gamma\vert\nabla_pf\vert^2{\rm d}x{\rm d}p.
	\end{align}
For $I_{12}$ and $I_{14}$, it follows from  Lemma \ref{lemma A3} and Lemma \ref{lemma A5} respectively  that
\begin{align*}
|I_{12}|+|I_{14}|
	\leq &\eta\int_{\mathbb R^6}e^{\bar\phi+2\phi}(e^{2\phi}+\vert p\vert^2)^{\gamma-\frac{1}{2}}(e^{2\phi}\vert\nabla_p\partial_{x_i}f\vert^2+\vert p\cdot\nabla_p\partial_{x_i}f\vert^2){\rm d}x{\rm d}p\\
\quad	&+Ce^{\bar\phi}\int_{\mathbb R^6} e^{\bar\phi}(e^{2\phi}+\vert p\vert^2)^{\gamma}\vert\partial_{x_i}f\vert^2{\rm d}x{\rm d}p+Ce^{\bar\phi}(\left\|\nabla_x\Phi\right\|_{L^2_x}^2+\left\|\nabla_x^2\Phi\right\|_{L^2_x}^2).
\end{align*}
By Lemma \ref{lemma A4} with $\vert\alpha\vert=1$, $\vert\beta\vert=0$, we have from \eqref{m-lb} that
\begin{align}
I_{13}
&\leq -(1-\eta)\int_{\mathbb R^6}e^{\bar\phi+2\phi}(e^{2\phi}+\vert p\vert^2)^{\gamma-\frac{1}{2}}(e^{2\phi}\vert\nabla_p\partial_{x_i}f\vert^2+\vert p\cdot\nabla_p\partial_{x_i}f\vert^2){\rm d}x{\rm d}p\nonumber\\
&\quad+C_\eta \int_{\mathbb R^6}e^{2\phi}(e^{2\phi}+\vert p\vert^2)^{\gamma-\frac{1}{2}}(e^{2\phi}\vert\nabla_p f\vert^2+\vert p\cdot\nabla_pf\vert^2){\rm d}x{\rm d}p+C_\eta e^{\bar\phi}\sum_{m+n\leq1}\mathcal E_{m,n}^\gamma(t).\notag
\end{align}
Combining the above estimates together gives
\begin{align}\label{1s-f}
\frac{d}{dt}\int_{\mathbb R^6}&\big(e^{3\bar\phi}(e^{2\phi}+\vert p\vert^2)^\gamma\vert\partial_{p_i}f\vert^2+e^{\bar\phi}(e^{2\phi}+\vert p\vert^2)^\gamma\vert\partial_{x_i}f\vert^2\big){\rm d}x{\rm d}p\nonumber\\
&+\int_{\mathbb R^6}e^{3\bar\phi+2\phi}(e^{2\phi}+\vert p\vert^2)^{\gamma-\frac{1}{2}}(e^{2\phi}\vert\nabla_p\partial_{p_i}f\vert^2+\vert p\cdot\nabla_p\partial_{p_i}f\vert^2){\rm d}x{\rm d}p\nonumber\\
&+\int_{\mathbb R^6}e^{\bar\phi+2\phi}(e^{2\phi}+\vert p\vert^2)^{\gamma-\frac{1}{2}}(e^{2\phi}\vert\nabla_p\partial_{x_i} f\vert^2+\vert p\cdot\nabla_p\partial_{x_i}f\vert^2){\rm d}x{\rm d}p\nonumber\\
\leq& C\sum_{m+n\leq1}\mathcal E_{m,n}^\gamma(t)+C\mathcal D_{0,0}^\gamma(t)+Ce^{\bar\phi}\left\|\Phi\right\|_{L^2_x}^2.
	\end{align}
Furthermore, by applying $\partial_{x_i}$ to \eqref{vnfp-s2}$_2$, multiplying the resultant identity by $\partial_t\partial_{x_i}\Phi$, and then integrating the resulting equality over $\mathbb R^3$,
it holds
\begin{align*}
\frac{1}{2}\frac{d}{dt}\int_{\mathbb R^3}&(\vert\partial_t\partial_{x_i}\Phi\vert^2+\vert\nabla_x\partial_{x_i}\Phi\vert^2){\rm d}x\\
=&-\int_{\mathbb R^3}\partial_{x_i}\Big(e^{2\phi}\int_{\mathbb R^3}\frac{f}{\sqrt{e^{2\phi}+\vert p\vert^2}}{\rm d}p\Big)\partial_t\partial_{x_i}\Phi{\rm d}x-\int_{\mathbb R^3}\int_{\mathbb R^3}\partial_{x_i}\Big(e^{2\phi}\frac{\bar F}{\sqrt{e^{2\phi}+\vert p\vert^2}}\Big){\rm d}p\partial_t\partial_{x_i}\Phi{\rm d}x\\
	:=&J_3+J_4.
	\end{align*}
Similar to the estimation on  $J_1$ and $J_2$, we have 
\begin{align*}
|J_3|&\leq Ce^{2\phi}\left\|\partial_t\partial_{x_i}\Phi\right\|_{L^2_x}\left\|\nabla_x\Phi\right\|_{L^\infty_x}\left\|\int_{\mathbb R^3}\frac{f}{\sqrt{e^{2\phi}+\vert p\vert^2}}{\rm d}p\right\|_{L^2_x}+ Ce^{2\phi}\left\|\partial_t\partial_{x_i}\Phi\right\|_{L^2_x}\left\|\int_{\mathbb R^3}\frac{\partial_{x_i}f}{\sqrt{e^{2\phi}+\vert p\vert^2}}{\rm d}p\right\|_{L^2_x}\\
&\leq Ce^{2\bar\phi}\Big(\left\|\partial_t\partial_{x_i}\Phi\right\|_{L^2_x}^2+\left\|(e^{2\phi}+\vert p\vert^2)^{\frac{\delta_2}{2}}f\right\|_{L_{x,p}^2}^2+\left\|(e^{2\phi}+\vert p\vert^2)^{\frac{\delta_2}{2}}\partial_{x_i}f\right\|_{L_{x,p}^2}^2\Big),
\end{align*}
and
\begin{equation*}
	\begin{split}
	|	J_4|&\leq Ce^{2\bar\phi}\Big(\left\|\partial_t\partial_{x_i}\Phi\right\|_{L^2_x}^2+\left\|(e^{2\bar\phi}+\vert p\vert^2)^{\frac{\delta_2}{2}}\bar F\right\|_{L_{p}^2}^2\left\|\partial_{x_i}\Phi\right\|_{L^2_x}^2\Big)\\
		&\leq Ce^{2\bar\phi}(\left\|\partial_t\partial_{x_i}\Phi\right\|_{L^2_x}^2+\left\|\partial_{x_i}\Phi\right\|_{L^2_x}^2).
	\end{split}
\end{equation*}
Thus,
\begin{equation}\label{1st-phi}
	\frac{d}{dt}\int_{\mathbb R^3}(\vert\partial_t\partial_{x_i}\Phi\vert^2+\vert\nabla_x\partial_{x_i}\Phi\vert^2){\rm d}x\leq Ce^{\bar\phi}\sum_{ m+n\leq1}\mathcal E_{m,n}^{\delta_2}(t).
\end{equation}
By taking $\gamma\in\{\delta_1,\delta_2\}$ in \eqref{1s-f}, respectively,
and combining the resulting inequalities with \eqref{1st-phi}, we get
\begin{align*}
\frac{d}{dt}\big(\mathcal E_{0,1}^{\delta_1}(t)&+\mathcal E_{0,1}^{\delta_2}(t)+\mathcal E_{1,0}^{\delta_1}(t)+\mathcal E_{1,0}^{\delta_2}(t)\big)+\mathcal D_{0,1}^{\delta_1}(t)+\mathcal D_{0,1}^{\delta_2}(t)+\mathcal D_{1,0}^{\delta_1}(t)+\mathcal D_{1,0}^{\delta_2}(t)\\
&\leq C\sum_{m+n\leq1}(\mathcal E_{m,n}^{\delta_1}(t)+\mathcal E_{m,n}^{\delta_2}(t))+C\mathcal D_{0,0}^{\delta_1}(t)+C\mathcal D_{0,0}^{\delta_2}(t)+Ce^{\bar\phi}\left\|\Phi\right\|_{L^2_x}^2.
\end{align*}
Consequently, we have 
\begin{align}\label{2.26}
\mathcal E_{0,1}^{\delta_1}(t)&+\mathcal E_{0,1}^{\delta_2}(t)+\mathcal E_{1,0}^{\delta_1}+\mathcal E_{1,0}^{\delta_2}(t)+\int_0^t\big(\mathcal D_{0,1}^{\delta_1}(s)+\mathcal D_{0,1}^{\delta_2}(s)
+\mathcal D_{1,0}^{\delta_1}(s)+\mathcal D_{1,0}^{\delta_2}(s)\big){\rm d}s\notag\\
&\leq C(T_c)\sum_{m+n\leq1}\mathcal E_{m,n}^{\delta_2}(0)+C(T_c)\left\|\Phi_0\right\|_{L^2_x}^2.	
\end{align}

\noindent\underline{{\it Step 3. High-order energy estimates.}}
 By taking $2\leq\vert\alpha\vert+\vert\beta\vert\leq4$ and $\vert\beta\vert<4$, applying $\partial_x^\alpha\partial^\beta_{p}$ to \eqref{vnfp-s2}$_1$, multiplying the equation by $e^{(\vert\alpha\vert+3\vert\beta\vert)\bar\phi}(e^{2\phi}+\vert p\vert^2)^\gamma\partial_x^\alpha\partial_p^\beta f$, and then integrating over $(x,p)\in\mathbb R^3\times\mathbb R^3$, it holds that
\begin{align}\label{2.27}
\frac{1}{2}\frac{d}{dt}&\int_{\mathbb R^6}e^{(\vert\alpha\vert+3\vert\beta\vert)\bar\phi}(e^{2\phi}+\vert p\vert^2)^\gamma\vert\partial_x^\alpha\partial_p^\beta f\vert^2{\rm d}x{\rm d}p\nonumber\\
		=&\frac{1}{2}\int_{\mathbb R^6}\vert\partial_x^\alpha\partial_p^\beta f\vert^2\partial_t(e^{(\vert\alpha\vert+3\vert\beta\vert)\bar\phi}(e^{2\phi}+\vert p\vert^2)^\gamma){\rm d}x{\rm d}p\nonumber\\
		&+\int_{\mathbb R^6}e^{(\vert\alpha\vert+3\vert\beta\vert)\bar\phi}\left[\nabla_p\sqrt{e^{2\phi}+\vert p\vert^2}\cdot\nabla_x-\nabla_x\sqrt{e^{2\phi}+\vert p\vert^2}\cdot\nabla_p,\partial_x^\alpha\partial_p^\beta\right]
f(e^{2\phi}+\vert p\vert^2)^\ga\partial_x^\alpha\partial_p^\beta f{\rm d}x{\rm d}p\nonumber\\
		&+\int_{\mathbb R^6}e^{(\vert\alpha\vert+3\vert\beta\vert)\bar\phi}\partial_x^\alpha\partial_p^\beta(\nabla_x\sqrt{e^{2\phi}+\vert p\vert^2}\cdot\nabla_p\bar F)(e^{2\phi}+\vert p\vert^2)^\gamma \partial_x^\alpha\partial_p^\beta f{\rm d}x{\rm d}p\nonumber\\
		&+\int_{\mathbb R^6}e^{(\vert\alpha\vert+3\vert\beta\vert)\bar\phi}\partial_x^\alpha\partial_p^\beta\big(e^{2\phi}\nabla_p\cdot(\Lambda_{\phi,p}\nabla_p f)\big)(e^{2\phi}+\vert p\vert^2)^\gamma \partial_x^\alpha\partial_p^\beta f{\rm d}x{\rm d}p\nonumber\\
		&+\int_{\mathbb R^6}e^{(\vert\alpha\vert+3\vert\beta\vert)\bar\phi}\partial_x^\alpha\partial_p^\beta\big(e^{2\phi}\nabla_p\cdot(\Lambda_{\phi,p}\nabla_p\bar F)-e^{2\bar\phi}\nabla_p\cdot(\Lambda_{\bar\phi,p}\nabla_p\bar F)\big)(e^{2\phi}+\vert p\vert^2)^\gamma \partial_x^\alpha\partial_p^\beta f{\rm d}x{\rm d}p\nonumber\\
		:=&I_{15}+I_{16}+I_{17}+I_{18}+I_{19},
\end{align}
where $[A,B]=AB-BA$ is the commutator operator.
We now estimate  $I_{i}$ $(15\leq i\leq19)$ as follows.
First of all, for all $(\alpha,\beta)$  with $2\leq\vert\alpha\vert+\vert\beta\vert\leq4$ and $\vert\beta\vert<4$,  it holds that 
\begin{equation}\label{I15}
	|I_{15}|\leq C\int_{\mathbb R^6}e^{(\vert\alpha\vert+3\vert\beta\vert)\bar\phi}(e^{2\phi}+\vert p\vert^2)^\gamma\vert\partial_x^\alpha\partial_p^\beta f\vert^2{\rm d}x{\rm d}p.
	\end{equation}
Lemma \ref{lemma A3} implies  that for all $(\alpha,\beta)$,
\begin{equation}\label{I17}
	\begin{split}
	|I_{17}|\leq Ce^{\bar\phi}\int_{\mathbb R^6}e^{(\vert\alpha\vert+3\vert\beta\vert)\bar\phi}(e^{2\phi}+\vert p\vert^2)^\gamma\vert\partial_x^\alpha\partial_p^\beta f\vert^2{\rm d}x{\rm d}p+Ce^{\bar\phi}\left\|\nabla_x\Phi\right\|_{H^{\vert\alpha\vert}_x}^2.
\end{split}
\end{equation}
And Lemma \ref{lemma A4} implies that for all $(\alpha,\beta)$,
\begin{align}\label{I18}
I_{18}
\leq& -(1-\eta)\int_{\mathbb R^6}e^{(\vert\alpha\vert+3\vert\beta\vert)\bar\phi+2\phi}(e^{2\phi}+\vert p\vert^2)^{\gamma-\frac{1}{2}}\big(e^{2\phi}\vert\nabla_p\partial_x^\alpha\partial_p^\beta f\vert^2+\vert p\cdot\na_p\partial_x^\alpha\partial_p^\beta f\vert^2\big){\rm d}x{\rm d}p\nonumber\\
&+Ce^{\bar\phi}\sum_{\substack{m+n\leq\vert\alpha\vert+\vert\beta\vert\\n<4}}\mathcal E^\gamma_{m,n}(t)+C\sum_{m+n<\vert\alpha\vert+\vert\beta\vert}\mathcal D_{m,n}^\gamma(t).
\end{align}
Similarly,  for all $(\alpha,\beta)$, Lemma \ref{lemma A5} gives
\begin{align}\label{I19}
|I_{19}|\leq& \eta\int_{\mathbb R^6}e^{(\vert\alpha\vert+3\vert\beta\vert)\bar\phi+2\phi}(e^{2\phi}+\vert p\vert^2)^{\gamma-\frac{1}{2}}(e^{2\phi}\vert\nabla_p\partial_x^\alpha\partial_p^\beta f\vert^2+\vert p\cdot\nabla_p\partial_x^\alpha\partial_p^\beta f\vert^2){\rm d}x{\rm d}p\nonumber\\
&+Ce^{\bar\phi}\int_{\mathbb R^6}e^{(\vert\alpha\vert+3\vert\beta\vert)\bar\phi}(e^{2\phi}+\vert p\vert^2)^\gamma\vert\partial_x^\alpha\partial_p^\beta f\vert^2{\rm d}x{\rm d}p+C_\eta e^{\bar\phi}\left\|\Phi\right\|_{H^{\vert\alpha\vert}_x}^2.
\end{align}

The estimation on  $I_{16}$ is more complicated that will be divided into following three cases.

\noindent\underline{{\it Case 1. $\vert\alpha\vert+\vert\beta\vert=2$.}}
We also divide this case into three subcases: (a)$|\al|=0$ and $|\beta|=2$, (b)$|\al|=|\beta|=1$,  (c)$|\al|=2$ and $|\beta|=0$.
For the first subcase when $|\al|=0$ and $|\beta|=2$,
by using Sobolev's imbedding inequality and H\"{o}lder's inequality, we have
\begin{align*}
|I_{16}\chi_{|\al|=0,|\beta|=2}|&\leq C\int_{\mathbb R^6}e^{3\bar\phi}(e^{2\phi}+\vert p\vert^2)^{\frac{\gamma}{2}}\vert\nabla_p^2f\vert\\
&\qquad\times\Big(e^{\frac{\bar\phi}{2}}(e^{2\phi}+\vert p\vert^2)^{\frac{\gamma}{2}}\vert\nabla_xf\vert e^{\frac{\bar\phi}{2}}+e^{2\bar\phi}(e^{2\phi}+\vert p\vert^2)^{\frac{\gamma}{2}}\vert\nabla_p\nabla_xf\vert\Big){\rm d}x{\rm d}p\nonumber\\
&\quad+C\int_{\mathbb R^6} e^{3\bar\phi}(e^{2\phi}+\vert p\vert^2)^{\frac{\gamma}{2}}\vert\nabla_p^2f\vert e^{\frac{3\bar\phi}{2}}(e^{2\phi}+\vert p\vert^2)^{\frac{\gamma}{2}}\vert\nabla_pf\vert{\rm d}x{\rm d}p\nonumber\\
&\quad+Ce^{\bar\phi}\int_{\mathbb R^6}e^{6\bar\phi}(e^{2\phi}+\vert p\vert^2)^{\gamma}\vert\nabla_p^2f\vert^2{\rm d}x{\rm d}p\nonumber\\
&\leq C\int_{\mathbb R^6}e^{6\bar\phi}(e^{2\phi}+\vert p\vert^2)^\gamma\vert\nabla_p^2f\vert^2{\rm d}x{\rm d}p+C\int_{\mathbb R^6}e^{4\bar\phi}(e^{2\phi}+\vert p\vert^2)^\gamma\vert\nabla_p\nabla_xf\vert^2{\rm d}x{\rm d}p\nonumber\\
&\quad+C\int_{\mathbb R^6}e^{3\bar\phi}(e^{2\phi}+\vert p\vert^2)^\gamma\vert\nabla_pf\vert^2{\rm d}x{\rm d}p+C\int_{\mathbb R^6}e^{\bar\phi}(e^{2\phi}+\vert p\vert^2)^\gamma\vert\nabla_xf\vert^2{\rm d}x{\rm d}p\nonumber\\
&\leq C\sum_{m+n\leq2}\mathcal E_{m,n}^\gamma(t).
\end{align*}
Next, if $\vert\alpha\vert=\vert\beta\vert=1$,
similarly, one has
\begin{align}
|I_{16}\chi_{\vert\alpha\vert=\vert\beta\vert=1}|&\leq C\int_{\mathbb R^6}e^{4\bar\phi}(e^{2\phi}+\vert p\vert^2)^\gamma\vert\nabla_x\nabla_p f\vert\Big((e^{2\phi}+\vert p\vert^2)^{-\frac{1}{2}}\vert\nabla_x^2 f\vert\notag\\&\qquad\qquad\qquad+\vert\nabla_x\Phi\vert\vert\nabla_p\nabla_x f\vert+(e^{2\phi}+\vert p\vert^2)^{-\frac{1}{2}}\vert\nabla_x\Phi\vert\vert\nabla_x f\vert\Big){\rm d}x{\rm d}p\nonumber\\
&\quad+C\int_{\mathbb R^6}e^{4\bar\phi+2\phi}(e^{2\phi}+\vert p\vert^2)^\gamma\vert\nabla_x\nabla_p f\vert\Big((e^{2\phi}+\vert p\vert^2)^{-\frac{1}{2}}(\vert\nabla_x\Phi\vert^2+\vert\nabla_x^2\Phi\vert)\vert\nabla_p^2f\vert\nonumber\\
&\qquad\qquad+(e^{2\phi}+\vert p\vert^2)^{-1}\vert\nabla_x\Phi\vert\vert\nabla_x\nabla_pf\vert+(e^{2\phi}+\vert p\vert^2)^{-1}(\vert\nabla_x\Phi\vert^2+\vert\nabla_x^2\Phi\vert)\vert\nabla_pf\vert\Big){\rm d}x{\rm d}p\nonumber\\
&\leq C\int_{\mathbb R^6}e^{2\bar\phi}(e^{2\phi}+\vert p\vert^2)^\gamma\vert\nabla_x^2f\vert^2{\rm d}x{\rm d}p+C\int_{\mathbb R^6}e^{4\bar\phi}(e^{2\phi}+\vert p\vert^2)^\gamma\vert\nabla_x\nabla_pf\vert^2{\rm d}x{\rm d}p\nonumber\\
&\quad+C\int_{\mathbb R^6}e^{6\bar\phi}(e^{2\phi}+\vert p\vert^2)^\gamma\vert\nabla_p^2f\vert^2{\rm d}x{\rm d}p+C\int_{\mathbb R^6}e^{\bar\phi}(e^{2\phi}+\vert p\vert^2)^\gamma\vert\nabla_xf\vert^2{\rm d}x{\rm d}p\nonumber\\
&\quad+C\int_{\mathbb R^6}e^{3\bar\phi}(e^{2\phi}+\vert p\vert^2)^\gamma\vert\nabla_pf\vert^2{\rm d}x{\rm d}p\leq C\sum_{m+n\leq2}\mathcal E_{m,n}^\gamma(t).\notag
\end{align}
And then for  $\vert\alpha\vert=2$ and $|\beta|=0$, one has
\begin{align}\label{I16-p1}
|I_{16}\chi_{\vert\alpha\vert=2,|\beta|=0}|&\leq C\left\|(\nabla_x\Phi,\nabla_x^2\Phi)\right\|_{L^\infty_x}\int_{\mathbb R^6}e^{\frac{\bar\phi}{2}}(e^{2\phi}+\vert p\vert^2)^{\frac{\gamma}{2}}\vert\nabla_xf\vert e^{\bar\phi}(e^{2\phi}+\vert p\vert^2)^{\frac{\gamma}{2}}\vert\nabla_x^2f\vert e^{\frac{\bar\phi}{2}}{\rm d}x{\rm d}p\nonumber\\
&\quad+C\left\|\nabla_x\Phi\right\|_{L^\infty_x}\int_{\mathbb R^6}e^{2\bar\phi}(e^{2\phi}+\vert p\vert^2)^{\gamma}\vert\nabla_x^2f\vert^2{\rm d}x{\rm d}p\nonumber\\
&\quad+C\int_{\mathbb R^6}e^{\bar\phi}(e^{2\phi}+\vert p\vert^2)^{\gamma}\vert\nabla_x^2f\vert e^{2\bar\phi}(\vert\nabla_x\Phi\vert+\vert\nabla_x^2\Phi\vert+\vert\nabla_x^3\Phi\vert)\vert\nabla_pf\vert{\rm d}x{\rm d}p\nonumber\\
&\quad+C\int_{\mathbb R^6} e^{\bar\phi}(e^{2\phi}+\vert p\vert^2)^{\gamma}\vert\nabla_x^2f\vert e^{2\bar\phi}(\vert\nabla_x\Phi\vert^2+\vert\nabla_x^2\Phi\vert)\vert\nabla_x\nabla_p f\vert{\rm d}x{\rm d}p\nonumber\\
&\leq C\sum_{m+m\leq2}\mathcal E_{m,n}^\gamma(t).
\end{align}
Substituting \eqref{I15}, \eqref{I17}, \eqref{I18}, \eqref{I19} and \eqref{I16-p1}  into \eqref{2.27} yields
\begin{align}
	\frac{d}{dt}\sum_{|\alpha|+\vert\beta\vert=2}&\int_{\mathbb R^6}e^{(|\alpha|+3|\beta|)\bar\phi}(e^{2\phi}+\vert p\vert^2)^\gamma\vert\partial_x^\alpha\partial_p^\beta f\vert^2{\rm d}x{\rm d}p\notag\\&+\sum_{|\alpha|+\vert\beta\vert=2}\int_{\mathbb R^6}e^{(|\alpha|+3|\beta|)\bar\phi+2\phi}(e^{2\phi}+\vert p\vert^2)^{\gamma-\frac{1}{2}}(e^{2\phi}\vert\nabla_p\partial_x^\alpha\partial_p^\beta f\vert^2+\vert p\cdot\nabla_p\partial_x^\alpha\partial_p^\beta f\vert^2){\rm d}x{\rm d}p
	\notag\\
	\leq&C\sum_{m+n\leq2}\mathcal E_{m,n}^{\gamma}(t)+C\sum_{m+n\leq1}\mathcal D_{m,n}^\gamma(t)+Ce^{\bar\phi}\left\|\Phi\right\|_{L^2_x}^2.\notag
\end{align}

\noindent\underline{{\it Case 2. $\vert\alpha\vert+\vert\beta\vert=3$.}} When $|\beta|=3$, by Sobolev's inequality and H\"{o}lder's inequality, we have
\begin{align*}
|I_{16}\chi_{|\beta|=3}|&\leq C\int_{\mathbb R^6}e^{9\bar\phi}(e^{2\phi}+\vert p\vert^2)^\gamma|\partial_p^\beta f|\Big((e^{2\phi}+\vert p\vert^2)^{-\frac{3}{2}}|\nabla_xf|+(e^{2\phi}+\vert p\vert^2)^{-1}|\nabla_x\nabla_pf|\\
&\qquad\qquad+(e^{2\phi}+\vert p\vert^2)^{-\frac{1}{2}}|\nabla_x\nabla_p^2f|\Big){\rm d}x{\rm d}p\\
&\quad+C\int_{\mathbb R^6}e^{9\bar\phi}(e^{2\phi}+\vert p\vert^2)^\gamma|\partial_p^\beta f| e^{\phi}|\nabla_x\Phi|\Big((e^{2\phi}+\vert p\vert^2)^{-\frac{3}{2}}|\nabla_pf|+(e^{2\phi}+\vert p\vert^2)^{-1}|\nabla_p^2f|\\
&\qquad\qquad+(e^{2\phi}+\vert p\vert^2)^{-\frac{1}{2}}|\nabla_p^3f|\Big){\rm d}x{\rm d}p\\
&\leq C\int_{\mathbb R^6}e^{\frac{9\bar\phi}{2}}(e^{2\phi}+\vert p\vert^2)^{\frac{\gamma}{2}}|\partial_p^\beta f|(e^{2\phi}+\vert p\vert^2)^{\frac{\gamma}{2}}\Big(e^{\frac{\bar\phi}{2}}|\nabla_xf| e^{\bar\phi}+e^{2\bar\phi}|\nabla_x\nabla_pf| e^{\frac{\bar\phi}{2}}\\
&\qquad\qquad+e^{\frac{7\bar\phi}{2}}|\nabla_x\nabla_p^2f|\Big){\rm d}x{\rm d}p\\
&\quad+C\int_{\mathbb R^6}e^{\frac{9\bar\phi}{2}}(e^{2\phi}+\vert p\vert^2)^{\frac{\gamma}{2}}|\partial_p^\beta f|(e^{2\phi}+\vert p\vert^2)^{\frac{\gamma}{2}}\Big(e^{\frac{3\bar\phi}{2}}|\nabla_pf|e^{\bar\phi}+e^{3\bar\phi}|\nabla_p^2f|e^{\frac{\bar\phi}{2}}
+e^{\frac{9\bar\phi}{2}}|\nabla_p^3f|\Big){\rm d}x{\rm d}p\\
&\leq C\int_{\mathbb R^6}(e^{2\phi}+|p|^2)^\gamma(e^{9\bar\phi}|\nabla_p^3f|^2+e^{7\bar\phi}|\nabla_x\nabla_p^2f|^2){\rm d}x{\rm d}p+Ce^{\bar\phi}\sum_{m+n\leq2}\mathcal E_{m,n}^\gamma(t)\\
&\leq C\sum_{m+n\leq3}\mathcal E_{m,n}^\gamma(t).
\end{align*}
Similarly, when $\vert\alpha\vert=1$ and $\vert\beta\vert=2$,
it holds that
\begin{align}
|I_{16}&\chi_{\vert\alpha\vert=1,\vert\beta\vert=2}|\notag\\
&\leq C\int_{\mathbb R^6}e^{7\bar\phi}(e^{2\phi}+\vert p\vert^2)^{\gamma}\vert\partial_x\partial_p^2f\vert\Big((e^{2\phi}+\vert p\vert^2)^{-1}\vert\nabla_xf\vert+(e^{2\phi}+\vert p\vert^2)^{-\frac{1}{2}}\vert\nabla_x\nabla_p f\vert+\vert\nabla_x\nabla_p^2f\vert\nonumber\\
&\qquad\qquad+(e^{2\phi}+\vert p\vert^2)^{-1}\vert\nabla_x^2f\vert+(e^{2\phi}+\vert p\vert^2)^{-\frac{1}{2}}\vert\nabla_p\nabla_x^2f\vert\Big){\rm d}x{\rm d}p\nonumber\\
&\quad+C\int_{\mathbb R^6}e^{7\bar\phi}(e^{2\phi}+\vert p\vert^2)^{\gamma}\vert\partial_x\partial_p^2f\vert\Big(e^{2\phi}(e^{2\phi}+\vert p\vert^2)^{-\frac{1}{2}}\vert\nabla_p\nabla_p^2f\vert+e^{\bar\phi}\vert\nabla_p^2f\vert\notag\\
&\qquad\qquad+e^{\bar\phi}(e^{2\phi}+\vert p\vert^2)^{-1}\vert\nabla_pf\vert\Big){\rm d}x{\rm d}p\nonumber\\
&\leq C\int_{\mathbb R^6}e^{\frac{7\bar\phi}{2}}(e^{2\phi}+\vert p\vert^2)^{\frac{\gamma}{2}}\Big(e^{\frac{\bar\phi}{2}}(e^{2\phi}+\vert p\vert^2)^{\frac{\gamma}{2}}\vert\nabla_xf\vert e^{\bar\phi}+e^{2\bar\phi}(e^{2\phi}+\vert p\vert^2)^{\frac{\gamma}{2}}\vert\nabla_x\nabla_pf\vert e^{\frac{\bar\phi}{2}}\nonumber\\
&\qquad\qquad+e^{\frac{7\bar\phi}{2}}(e^{2\phi}+\vert p\vert^2)^{\frac{\gamma}{2}}\vert\nabla_x\nabla_p^2f\vert+e^{\bar\phi}(e^{2\phi}+\vert p\vert^2)^{\frac{\gamma}{2}}\vert\nabla_x^2f\vert e^{\frac{\bar\phi}{2}}+e^{\frac{5\bar\phi}{2}}(e^{2\phi}+\vert p\vert^2)^{\frac{\gamma}{2}}\vert\nabla_x^2\nabla_pf\vert \Big){\rm d}x{\rm d}p\nonumber\\
&\quad+C\int_{\mathbb R^6}e^{\frac{7\bar\phi}{2}}(e^{2\phi}+\vert p\vert^2)^{\frac{\gamma}{2}}\vert\nabla_x\nabla_p^2f\vert\Big(e^{3\bar\phi+2\phi}(e^{2\phi}+\vert p\vert^2)^{\frac{\gamma}{2}-\frac{1}{4}}\vert\nabla_p\nabla_p^2f\vert\nonumber\\
&\qquad\qquad+e^{3\bar\phi}(e^{2\phi}+\vert p\vert^2)^{\frac{\gamma}{2}}\vert\nabla_p^2f\vert e^{\frac{\bar\phi}{2}}+e^{\frac{3\bar\phi}{2}}(e^{2\phi}+\vert p\vert^2)^{\frac{\gamma}{2}}\vert\nabla_pf\vert e^{\bar\phi}\Big){\rm d}x{\rm d}p\nonumber\\
&\leq C\int_{\mathbb R^6}(e^{2\phi}+\vert p\vert^2)^{\gamma}(e^{7\bar\phi}\vert\nabla_x\nabla_p^2f\vert^2+e^{5\bar\phi}\vert\nabla_x^2\nabla_pf\vert^2+e^{9\bar\phi}\vert\nabla_p^3f\vert^2){\rm d}x{\rm d}p+C\sum_{m+n\leq2}\mathcal E_{m,n}^\gamma(t)\notag\\
&\leq C\sum_{m+n\leq3}\mathcal E_{m,n}^\gamma(t).\notag
\end{align}
And when  $\vert\alpha\vert=2$ and $\vert\beta\vert=1$, we have 
\begin{align}
|I_{16}&\chi_{\vert\alpha\vert=2,|\beta|=1}\notag|\\
\leq& C\int_{\mathbb R^6}e^{5\bar\phi}(e^{2\phi}+\vert p\vert^2)^{\gamma}\vert\partial_x^2\partial_pf\vert\Big((e^{2\phi}+\vert p\vert^2)^{-\frac{1}{2}}\vert\nabla_xf\vert\notag\\&\qquad+(e^{2\phi}+\vert p\vert^2)^{-\frac{1}{2}}\vert\nabla_x^2f\vert+(e^{2\phi}+\vert p\vert^2)^{-\frac{1}{2}}\vert\nabla_x^3f\vert+\vert\partial_x\nabla_pf\vert\Big){\rm d}x{\rm d}p\nonumber\\
	&+C\int_{\mathbb R^6}e^{5\bar\phi}(e^{2\phi}+\vert p\vert^2)^{\gamma}\vert\partial_x^2\partial_pf\vert\Big((1+\vert\nabla_x^3\Phi\vert)\vert\nabla_pf\vert+e^\phi(1+\vert\nabla_x^3\Phi|)\vert\nabla_p^2f\vert+\vert\nabla_x\nabla_pf\vert\nonumber\\
	&\qquad+\vert\nabla_x^2\nabla_pf\vert+\vert\nabla_x\nabla_p^2f\vert\Big){\rm d}x{\rm d}p\nonumber\\
	\leq& C\sum_{m+n\leq3}\mathcal E_{m,n}^\gamma(t).\notag
\end{align}
Finally, when  $\vert\alpha\vert=3$ and $|\beta|=0$, it holds that 
\begin{align}
|I_{16}\chi_{\vert\alpha\vert=3}|\leq& C\int_{\mathbb R^6}e^{3\bar\phi}(e^{2\phi}+\vert p\vert^2)^{\gamma}\vert\partial_x^\alpha f\vert e^{2\phi}(e^{2\phi}+\vert p\vert^2)^{-1}\Big((1+\vert\nabla_x^3\Phi\vert)\vert\nabla_xf\vert+\vert\nabla_x^2f\vert+|\nabla_x^3f\vert\Big){\rm d}x{\rm d}p\nonumber\\
&+C\int_{\mathbb R^6}e^{3\bar\phi}(e^{2\phi}+\vert p\vert^2)^{\gamma}\vert\partial_x^\alpha f\vert e^\phi\Big((1+|\nabla_x^3\Phi|+\vert\nabla_x^4\Phi\vert)\vert\nabla_pf\vert\nonumber\\
&\qquad\qquad+(1+|\nabla_x^3\Phi\vert)\vert\nabla_x\nabla_pf\vert+\vert\nabla_x^2\nabla_pf\vert\Big){\rm d}x{\rm d}p\nonumber\\
\leq& C\int_{\mathbb R^6}(e^{2\phi}+\vert p\vert^2)^{\gamma}(e^{3\bar\phi}\vert\partial_x^\alpha f\vert^2+e^{5\bar\phi}\vert\nabla_x^2\nabla_pf\vert^2){\rm d}x{\rm d}p
+C\sum_{m+n\leq2}\mathcal E_{m,n}^\gamma(t)\notag\\
\leq&C\sum_{m+n\leq3}\mathcal E_{m,n}^\gamma(t).\notag
\end{align}
Combining  all these estimates with \eqref{I15}, \eqref{I17}, \eqref{I18} and \eqref{I19}  we obtain
\begin{align}
	\frac{d}{dt}\sum_{|\alpha|+\vert\beta\vert=3}&\int_{\mathbb R^6}e^{(|\alpha|+3|\beta|)\bar\phi}(e^{2\phi}+\vert p\vert^2)^\gamma\vert\partial_x^\alpha\partial_p^\beta f\vert^2{\rm d}x{\rm d}p\notag\\&+\sum_{|\alpha|+\vert\beta\vert=3}\int_{\mathbb R^6}e^{(|\alpha|+3|\beta|)\bar\phi+2\phi}(e^{2\phi}+\vert p\vert^2)^{\gamma-\frac{1}{2}}(e^{2\phi}\vert\nabla_p\partial_x^\alpha\partial_p^\beta f\vert^2+\vert p\cdot\nabla_p\partial_x^\alpha\partial_p^\beta f\vert^2){\rm d}x{\rm d}p
	\notag\\
	\leq&C\sum_{m+n\leq3}\mathcal E_{m,n}^{\gamma}(t)+C\sum_{m+n\leq2}\mathcal D_{m,n}^\gamma(t)+Ce^{\bar\phi}\left\|\Phi\right\|_{L^2_x}^2.\notag
\end{align}

\noindent\underline{{\it Case 3. $\vert\alpha\vert+\vert\beta\vert=4$ and $|\beta|<4$}.}
 We divide  this case into four subcases: (1)$|\al|=1$ and $|\beta|=3$, (2)$|\al|=|\beta|=2$,  (3)$|\al|=3$ and $|\beta|=1$, (4)$|\al|=4$ and $|\beta|=0$.  Denote the integral $I_{16}$  in the case of $(i)$ by $I_{16}^{(i)}$ for $1\leq i\leq4$. Similar to the arguments for   the cases when  $|\alpha|+|\beta|=2$ and $|\alpha|+|\beta|=3$, we have
\begin{align*}
|I_{16}^{(1)}|&\leq C\int_{\mathbb R^6}e^{10\bar\phi}(e^{2\phi}+|p|^2)^\gamma|\partial_x^\alpha\partial_p^\beta f|\Big((e^{2\phi}+|p|^2)^{-\frac{3}{2}}(|\nabla_xf|+|\nabla_x^2f|)\\
&\qquad\qquad+(e^{2\phi}+|p|^2)^{-1}(|\nabla_x\nabla_pf|+|\nabla_x^2\nabla_pf|)+(e^{2\phi}+|p|^2)^{-\frac{1}{2}}(|\nabla_x\nabla_p^2f|+|\nabla_x^2\nabla_p^2f|)\\
&\qquad\qquad+|\nabla_x\nabla_p^3f|\Big){\rm d}x{\rm d}p\\
&\quad+C\int_{\mathbb R^6}e^{10\bar\phi}(e^{2\phi}+|p|^2)^\gamma|\partial_x^\alpha\partial_p^\beta f|e^{\phi}\Big((e^{2\phi}+|p|^2)^{-\frac{3}{2}}(|\nabla_pf|+|\nabla_x\nabla_pf|)\\
&\qquad\qquad+(e^{2\phi}+|p|^2)^{-1}(|\nabla_p^2f|+|\nabla_x\nabla_p^2f|)+(e^{2\phi}+|p|^2)^{-\frac{1}{2}}(|\nabla_p^3f|+|\nabla_x\nabla_p^3f|)\\
&\qquad\qquad+e^{\phi}(e^{2\phi}+|p|^2)^{-\frac{1}{2}}|\nabla_p\nabla_p^3f|\Big){\rm d}x{\rm d}p\\
&\leq C\int_{\mathbb R^6}e^{5\bar\phi}(e^{2\phi}+|p|^2)^{\frac{\gamma}{2}}|\partial_x^\alpha\partial_p^\beta f|(e^{2\phi}+|p|^2)^{\frac{\gamma}{2}}\Big((e^{\frac{\bar\phi}{2}}|\nabla_xf|+e^{\bar\phi}|\nabla_x^2f|)e^{\bar\phi}\\
&\qquad\qquad+(e^{2\bar\phi}|\nabla_x\nabla_pf|+e^{\frac{5\bar\phi}{2}}|\nabla_x^2\nabla_pf|)e^{\frac{\bar\phi}{2}}+e^{\frac{7\bar\phi}{2}}|\nabla_x\nabla_p^2f|e^{\frac{\bar\phi}{2}}+e^{4\bar\phi}|\nabla_x^2\nabla_p^2f|\Big){\rm d}x{\rm d}p\\
&\quad+C\int_{\mathbb R^6}e^{5\bar\phi}(e^{2\phi}+|p|^2)^{\frac{\gamma}{2}}|\partial_x^\alpha\partial_p^\beta f|(e^{2\phi}+|p|^2)^{\frac{\gamma}{2}}\Big((e^{\frac{3\bar\phi}{2}}|\nabla_pf|+e^{2\bar\phi}|\nabla_x\nabla_pf|)e^{\bar\phi}\\
&\qquad\qquad+(e^{3\bar\phi}|\nabla_p^2f|+e^{\frac{7\bar\phi}{2}}|\nabla_x\nabla_p^2f|)e^{\frac{\bar\phi}{2}}+e^{\frac{9\bar\phi}{2}}|\nabla_p^3f|e^{\frac{\bar\phi}{2}}+e^{5\bar\phi}|\nabla_x\nabla_p^3f|\\
&\qquad\qquad+e^{\frac{9}{2}\bar\phi+2\phi}(e^{2\phi}+|p|^2)^{-\frac{1}{4}}|\nabla_p\nabla_p^3f|\Big){\rm d}x{\rm d}p\\
&\leq C\int_{\mathbb R^6}(e^{2\phi}+|p|^2)^{\gamma}(e^{10\bar\phi}|\nabla_x\nabla_p^3f|^2+e^{8\bar\phi}|\nabla_x^2\nabla_p^2f|^2){\rm d}x{\rm d}p+Ce^{\bar\phi}\sum_{m+n\leq3}\mathcal E_{m,n}^\gamma(t)+C\mathcal D_{0,3}^\gamma(t),
\end{align*}
\begin{align*}
|I_{16}^{(2)}|&\leq C\int_{\mathbb R^6}e^{8\bar\phi}(e^{2\phi}+|p|^2)^\gamma|\partial_x^\alpha\partial_p^\beta f|\Big((e^{2\phi}+|p|^2)^{-1}(|\nabla_xf|+|\nabla_x^2f|+|\nabla_x^3f|)\\
&\qquad\qquad+(e^{2\phi}+|p|^2)^{-\frac{1}{2}}(|\nabla_x\nabla_pf|+|\nabla_x^2\nabla_pf|+|\nabla_x^3\nabla_pf|)+|\nabla_x\nabla_p^2f|+|\nabla_x^2\nabla_p^2f|\Big){\rm d}x{\rm d}p\\
&\quad+C\int_{\mathbb R^6}e^{8\bar\phi}(e^{2\phi}+|p|^2)^\gamma|\partial_x^\alpha\partial_p^\beta f|\Big((e^{2\phi}+|p|^2)^{-\frac{1}{2}}(1+|\nabla_x^3\Phi|)|\nabla_pf|\\
&\qquad\qquad+((1+|\nabla_x^3\Phi|)|\nabla_p^2f|+e^{\phi}\big((1+|\nabla_x^3\Phi|)|\nabla_p^3f|+|\nabla_x\nabla_p^3f|\big)\Big){\rm d}x{\rm d}p\\
&\leq C\int_{\mathbb R^6}e^{4\bar\phi}(e^{2\phi}+|p|^2)^{\frac{\ga}{2}}|\partial_x^\alpha\partial_p^\beta f|(e^{2\phi}+|p|^2)^{\frac{\ga}{2}}\Big((e^{\frac{\bar\phi}{2}}|\nabla_xf|+e^{\bar\phi}|\nabla_x^2f|+e^{\frac{3\bar\phi}{2}}|\nabla_x^3f|)e^{\frac{\bar\phi}{2}}\\
&\qquad\qquad+(e^{2\bar\phi}|\nabla_x\nabla_pf|+e^{\frac{5\bar\phi}{2}}|\nabla_x^2\nabla_pf|)e^{\frac{\bar\phi}{2}}+e^{4\bar\phi}|\nabla_x^3\nabla_pf|+e^{\frac{7\bar\phi}{2}}|\nabla_x\nabla_p^2f|e^{\frac{\bar\phi}{2}}+e^{4\bar\phi}|\nabla_x^2\nabla_p^2f|\Big){\rm }x{\rm d}p\\
&\quad+C\int_{\mathbb R^6}e^{4\bar\phi}(e^{2\phi}+|p|^2)^{\frac{\ga}{2}}|\partial_x^\alpha\partial_p^\beta f|(e^{2\phi}+|p|^2)^{\frac{\ga}{2}}\\
&\qquad\qquad\times\Big(e^{\frac{3\bar\phi}{2}}|\nabla_pf|e^{\frac{3\bar\phi}{2}}+e^{3\bar\phi}|\nabla_p^2f|e^{\bar\phi}+e^{\frac{9\bar\phi}{2}}|\nabla_p^3f|e^{\frac{\bar\phi}{2}}+e^{5\bar\phi}|\nabla_x\nabla_p^3f|\Big){\rm d}x{\rm d}p\\
&\quad+C\int_{\mathbb R^6}e^{4\bar\phi}(e^{2\phi}+|p|^2)^{\frac{\ga}{2}}|\partial_x^\alpha\partial_p^\beta f||\nabla_x^3\Phi|(e^{2\phi}+|p|^2)^{\frac{\ga}{2}}\\
&\qquad\qquad\times\Big(e^{2\bar\phi}|\nabla_pf|e^{\bar\phi}+e^{\frac{7\bar\phi}{2}}|\nabla_p^2f|e^{\frac{\bar\phi}{2}}+e^{5\bar\phi}|\nabla_p^3f|\Big){\rm d}x{\rm d}p\\
&\leq C\int_{\mathbb R^6}(e^{2\phi}+|p|^2)^\ga(e^{8\bar\phi}|\nabla_x^2\nabla_p^2f|+e^{6\bar\phi}|\nabla_x^3\nabla_pf|+e^{10\bar\phi}|\nabla_x\nabla_p^3f|){\rm d}x{\rm d}p+Ce^{\bar\phi}\sum_{m+n\leq3}\mathcal E_{m,n}^\gamma(t),
\end{align*}
\begin{align*}
|I_{16}^{(3)}|&\leq C\int_{\mathbb R^6}e^{6\bar\phi}(e^{2\phi}+|p|^2)^\gamma|\partial_x^\alpha\partial_p^\beta f|\Big((e^{2\phi}+|p|^2)^{-\frac{1}{2}}\big((1+|\nabla_x^3\Phi|)|\nabla_xf|+|\nabla_x^2f|+|\nabla_x^3f|+|\nabla_x^4f|\big)\\
&\qquad\qquad+(1+|\nabla_x^3\Phi|)|\nabla_x\nabla_pf|+|\nabla_x^2\nabla_pf|+|\nabla_x^3\nabla_pf|\Big){\rm d}x{\rm d}p\\
&\quad+C\int_{\mathbb R^6}e^{6\bar\phi}(e^{2\phi}+|p|^2)^\gamma|\partial_x^\alpha\partial_p^\beta f|\Big((1+|\nabla_x^3\Phi|+|\nabla_x^4\Phi|)|\nabla_pf|\\
&\qquad\qquad+e^{\bar\phi}\big((1+|\nabla_x^3\Phi|+|\nabla_x^4\Phi|)|\nabla_p^2f|+(1+|\nabla_x^3\Phi|)|\nabla_x\nabla_p^2f|+|\nabla_x^2\nabla_p^2f|\big)\Big){\rm d}x{\rm d}p\\
&\leq C\int_{\mathbb R^6}e^{3\bar\phi}(e^{2\phi}+|p|^2)^{\frac{\gamma}{2}}|\partial_x^\alpha\partial_p^\beta f|(e^{2\phi}+|p|^2)^{\frac{\ga}{2}}\Big((e^{\frac{\bar\phi}{2}}|\nabla_xf|+e^{\bar\phi}|\nabla_x^2f|\\
&\qquad\qquad+e^{\frac{3\bar\phi}{2}}|\nabla_x^3f|)e^{\frac{\bar\phi}{2}}+e^{2\bar\phi}|\nabla_x^4f|+(e^{2\bar\phi}|\nabla_x\nabla_pf|+e^{\frac{5\bar\phi}{2}}|\nabla_x^2\nabla_pf|)e^{\frac{\bar\phi}{2}}+e^{3\bar\phi}|\nabla_x^3\nabla_pf|\\
&\qquad\qquad+e^{\frac{3\bar\phi}{2}}|\nabla_pf|e^{\frac{\bar\phi}{2}}+(e^{3\bar\phi}|\nabla_p^2f|+e^{\frac{7\bar\phi}{2}}|\nabla_x\nabla_p^2f|)e^{\frac{\bar\phi}{2}}+e^{4\bar\phi}|\nabla_x^2\nabla_p^2f|\Big){\rm d}x{\rm d}p\\
&\quad+C\int_{\mathbb R^6}e^{3\bar\phi}(e^{2\phi}+|p|^2)^{\frac{\gamma}{2}}|\partial_x^\alpha\partial_p^\beta f||\nabla_x^3\Phi|(e^{2\phi}+|p|^2)^{\frac{\ga}{2}}\Big(e^{\bar\phi}|\nabla_xf|e^{\bar\phi}+e^{2\bar\phi}|\nabla_pf|e^{\bar\phi}\\
&\qquad\qquad+e^{\frac{5\bar\phi}{2}}|\nabla_x\nabla_pf|e^{\frac{\bar\phi}{2}}+e^{\frac{7\bar\phi}{2}}|\nabla_p^2f|e^{\frac{\bar\phi}{2}}+e^{4\bar\phi}|\nabla_x\nabla_p^2f|\Big){\rm d}x{\rm d}p\\
&\quad+C\int_{\mathbb R^6}e^{3\bar\phi}(e^{2\phi}+|p|^2)^{\frac{\gamma}{2}}|\partial_x^\alpha\partial_p^\beta f||\nabla_x^4\Phi|(e^{2\phi}+|p|^2)^{\frac{\ga}{2}} e^{\frac{5\bar\phi}{2}}|\nabla_pf|e^{\frac{\bar\phi}{2}}{\rm d}x{\rm d}p\\
&\leq C\int_{\mathbb R^6}(e^{2\phi}+|p|^2)^{\gamma}(e^{6\bar\phi}|\nabla_x^3\nabla_pf|^2+e^{4\bar\phi}|\nabla_x^4f|^2+e^{8\bar\phi}|\nabla_x^2\nabla_p^2f|^2){\rm d}x{\rm d}p+Ce^{\bar\phi}\sum_{m+n\leq3}\mathcal E_{m,n}^\gamma(t),
\end{align*}
and
\begin{align*}
		|I_{16}^{(4)}|&\leq	C\int_{\mathbb R^6}e^{4\bar\phi}(e^{2\phi}+|p|^2)^{\gamma}|\partial_x^\alpha f|\Big((1+|\nabla_x^3\Phi|+|\nabla_x^4\Phi|)|\nabla_xf|+(1+|\nabla_x^3\Phi|)|\nabla_x^2f|\\
		&\qquad\qquad+|\nabla_x^3f|+|\nabla_x^4f|\Big){\rm d}x{\rm d}p\\
		&\quad+C\int_{\mathbb R^6}e^{4\bar\phi}(e^{2\phi}+|p|^2)^{\gamma}|\partial_x^\alpha f|e^\phi\Big((1+|\nabla_x^3\Phi|+|\nabla_x^4\Phi|+|\nabla_x^5\Phi|)|\nabla_pf|\\
		&\qquad\qquad+(1+|\nabla_x^3\Phi|+|\nabla_x^4\Phi|)|\nabla_x\nabla_pf|+(1+|\nabla_x^3\Phi|)|\nabla_x^2\nabla_pf|+|\nabla_x^3\nabla_pf|\Big){\rm d}x{\rm d}p\\
		&\leq C\int_{\mathbb R^6}e^{2\bar\phi}(e^{2\phi}+|p|^2)^{\frac{\gamma}{2}}|\partial_x^\alpha f|\Big((e^{\frac{\bar\phi}{2}}|\nabla_xf|+e^{\bar\phi}|\nabla_x^2f|+e^{\frac{3\bar\phi}{2}}|\nabla_x^3f|)e^{\frac{\bar\phi}{2}}+e^{2\bar\phi}|\nabla_x^4f|\\
		&\qquad\qquad+(e^{\frac{3\bar\phi}{2}}|\nabla_pf|+e^{2\bar\phi}|\nabla_x\nabla_pf|)e^{\frac{\bar\phi}{2}}+e^{\frac{5\bar\phi}{2}}|\nabla_x^2\nabla_pf|e^{\frac{\bar\phi}{2}}+e^{3\bar\phi}|\nabla_x^3\nabla_pf|\Big){\rm d}x{\rm d}p\\
		&\quad+ C\int_{\mathbb R^6}e^{2\bar\phi}(e^{2\phi}+|p|^2)^{\frac{\gamma}{2}}|\partial_x^\alpha f||\nabla_x^3\Phi|\Big(e^{\bar\phi}|\nabla_xf|e^{\bar\phi}+e^{\frac{3\bar\phi}{2}}|\nabla_x^2f|e^{\frac{\bar\phi}{2}}\\
		&\qquad\qquad+(e^{2\bar\phi}|\nabla_pf|+e^{\frac{5\bar\phi}{2}}|\nabla_x\nabla_pf|)e^{\frac{\bar\phi}{2}}+e^{3\bar\phi}|\nabla_x^2\nabla_pf|\Big){\rm d}x{\rm d}p\\
		&\quad+ C\int_{\mathbb R^6}e^{2\bar\phi}(e^{2\phi}+|p|^2)^{\frac{\gamma}{2}}|\partial_x^\alpha f||\nabla_x^4\Phi|\Big(e^{\frac{3\bar\phi}{2}}|\nabla_xf|e^{\frac{\bar\phi}{2}}+e^{\frac{5\bar\phi}{2}}|\nabla_pf|e^{\frac{\bar\phi}{2}}\Big){\rm d}x{\rm d}p\\
		&\quad+ C\int_{\mathbb R^6}e^{2\bar\phi}(e^{2\phi}+|p|^2)^{\frac{\gamma}{2}}|\partial_x^\alpha f||\nabla_x^5\Phi|e^{3\bar\phi}|\nabla_pf|{\rm d}x{\rm d}p\\
		&\leq C\int_{\mathbb R^6}(e^{2\phi}+|p|^2)^{\gamma}(e^{4\bar\phi}|\nabla_x^4 f|^2+e^{5\bar\phi}|\nabla_x^2\nabla_pf|^2+e^{6\bar\phi}|\nabla_x^3\nabla_pf|^2){\rm d}x{\rm d}p+Ce^{\bar\phi}\sum_{m+n\leq3}\mathcal E_{m,n}^\gamma(t).
\end{align*}

Finally, for $2\leq\vert\alpha\vert\leq3$, applying $\partial_{x}^\alpha$ to \eqref{vnfp-s2}$_2$, multiplying the resulting equality by $\partial_t\partial_{x}^\alpha\Phi$, and integrating the resulting identity over $\mathbb R^3$ yield
\begin{align}\label{2.54}
		&\quad\frac{1}{2}\frac{d}{dt}\int_{\mathbb R^3}(\vert\partial_t\partial_{x}^\alpha\Phi\vert^2+\vert\nabla_x\partial_{x}^\alpha\Phi\vert^2){\rm d}x\nonumber\\
		&=-\int_{\mathbb R^3}\partial_{x}^\alpha\Big(e^{2\phi}\int_{\mathbb R^3}\frac{f}{\sqrt{e^{2\phi}+\vert p\vert^2}}{\rm d}p\Big)\partial_t\partial_{x}^\alpha\Phi{\rm d}x-\int_{\mathbb R^3}\int_{\mathbb R^3}\partial_{x}^\alpha\Big(e^{2\phi}\frac{\bar F}{\sqrt{e^{2\phi}+\vert p\vert^2}}\Big){\rm d}p\partial_t\partial_{x}^\alpha\Phi{\rm d}x\nonumber\\
		&	:=J_5+J_6.
\end{align}
By H\"{o}lder's inequality, we have
\begin{align}
		|J_5|&\leq Ce^{\frac{\bar\phi}{2}}\Big(\left\|\partial_t\partial_{x}^\alpha\Phi\right\|_{L^2_x}^2+\sum_{|\alpha'|\leq|\alpha|}\left\|e^{\frac{|\alpha'|\bar\phi}{2}}(e^{2\phi}+\vert p\vert^2)^{\frac{\delta_2}{2}}\nabla_x^{|\alpha'|}f\right\|_{L_{x,p}^2}^2\Big)\nonumber\\
	&\leq Ce^{\frac{\bar\phi}{2}}\sum_{ m+n\leq|\alpha|}\mathcal E_{m,n}^{\gamma}(t).\notag
\end{align}
Similarly, it holds that
\begin{equation}
|J_6|\leq  e^{2\bar\phi}(\left\|\partial_t\partial_x^\alpha\Phi\right\|_{L^2_x}^2+\left\|\nabla_x\Phi\right\|_{H^2_x}^2).\notag
\end{equation}
If $|\alpha|=4$, we directly have
\begin{align}\label{n=4,fin}
\frac{1}{2}&\frac{d}{dt}\int_{\mathbb R^3}e^{\bar\phi}(\vert\partial_t\partial_{x}^\alpha\Phi\vert^2+\vert\nabla_x\partial_{x}^\alpha\Phi\vert^2){\rm d}x\notag\\
&=\frac{1}{2}\int_{\mathbb R^3}\bar\phi'(t)e^{\bar\phi}\vert\partial_t\partial_{x}^\alpha\Phi\vert^2{\rm d}x-\int_{\mathbb R^3}\partial_{x}^\alpha\Big(e^{2\phi}\int_{\mathbb R^3}\frac{f}{\sqrt{e^{2\phi}+\vert p\vert^2}}{\rm d}p\Big)e^{\bar\phi}\partial_t\partial_{x}^\alpha\Phi{\rm d}x\nonumber\\
&\quad-\int_{\mathbb R^3}\int_{\mathbb R^3}\partial_{x}^\alpha\Big(e^{2\phi}\frac{\bar F}{\sqrt{e^{2\phi}+\vert p\vert^2}}\Big){\rm d}pe^{\bar\phi}\partial_t\partial_{x}^\alpha\Phi{\rm d}x\nonumber\\
&\leq C\int_{\mathbb R^3}e^{\bar\phi}\vert\partial_t\partial_{x}^\alpha\Phi\vert^2{\rm d}x+Ce^{\frac{\bar\phi}{2}}\sum_{\substack{ m+n\leq4\\ n<4}}\mathcal E_{m,n}^{\gamma}(t)+Ce^{\bar\phi}\left\|\nabla_x\Phi\right\|_{H^3_x}^2.
\end{align}
To summarize, we deduce when $2\leq\vert\alpha\vert+\vert\beta\vert\leq4$ and $\vert\beta\vert<4$ that
\begin{align}
\frac{d}{dt}\sum_{\substack{2\leq m+n\leq4\\n<4}}&\mathcal E_{m,n}^\gamma(t)+\sum_{\substack{2\leq m+n\leq4\\n<4}}\mathcal D_{m,n}^\gamma(t)
\leq C\sum_{\substack{m+n\leq4\\n<4}}(\mathcal E_{m,n}^{\gamma}(t)+\mathcal E_{m,n}^{\delta_2}(t))+Ce^{\bar\phi}\left\|\Phi\right\|_{L^2_x}^2.\label{th-eng-sum}
\end{align}
Putting \eqref{z-sum3}, \eqref{2.26} and \eqref{th-eng-sum} together,  \eqref{loc-eng} holds
for $t\in[0,T_c]$. This completes the proof of the proposition.

\end{proof}

\subsection{Energy estimates in $(T_c,\infty)$.}\label{sec-ift-eng}
In this subsection, we will extend the energy estimates obtained in \eqref{loc-eng} to all time.
For this, we prove  the following proposition.

\begin{proposition}\label{lg-eng-pro}
Assume $[f,\Phi]$ is a solution to \eqref{vnfp-s2} and \eqref{id-s2} that satisfies \eqref{aps}, then for any $t>T_c$, it holds that
\begin{align}
\sum_{\substack{m+n\leq4\\n<4}}&(\mathcal E_{m,n}^{\delta_1}(t)+\mathcal E_{m,n}^{\delta_2}(t))+\int_0^t(\mathcal D_{0,0}^{\delta_1}(s)+\mathcal D_{0,0}^{\delta_2}(s)){\rm d}s+\sum_{\substack{ 1\leq m+n\leq4\\n<4}}\int_0^t	(\bar{\mathcal D}_{m,n}^{\delta_1}(s)+\bar{\mathcal D}_{m,n}^{\delta_2}(s)){\rm d}s\notag\\
&\leq C(T_c)\Big(\sum_{\substack{m+n\leq4\\n<4}}	
\mathcal E_{m,n}^{\delta_2}(0)+\left\|\Phi_0\right\|_{L^2_x}^2\Big).\notag
\end{align}
\end{proposition}
\begin{proof}
The proof is based on the following three steps.

\noindent\underline{{\it Step 1. Zeroth order energy estimate.}}
From \eqref{vnfp-s2}$_1$, it follows that
\begin{align}\label{3.1}
\frac{1}{2}\frac{d}{dt}\int_{\mathbb R^6}&(e^{2\phi}+\vert p\vert^2)^\gamma\vert f\vert^2{\rm d}x{\rm d}p\nonumber\\
		=&\gamma\int_{\mathbb R^6}\partial_t\phi e^{2\phi}(e^{2\phi}+\vert p\vert^2)^{\gamma-1}\vert f\vert^2{\rm d}x{\rm d}p+\int_{\mathbb R^6}\nabla_x\sqrt{e^{2\phi}+\vert p\vert^2}\cdot\nabla_p\bar F(e^{2\phi}+\vert p\vert^2)^\gamma f{\rm d}x{\rm d}p\nonumber\\
		&+\int_{\mathbb R^6}e^{2\phi}\nabla_p\cdot(\Lambda_{\phi,p}\nabla_p f)(e^{2\phi}+\vert p\vert^2)^\gamma f{\rm d}x{\rm d}p\nonumber\\
		&+\int_{\mathbb R^6}(e^{2\phi}\nabla_p\cdot(\Lambda_{\phi,p}\nabla_p\bar F)-e^{2\bar\phi}\nabla_p\cdot(\Lambda_{\bar\phi,p}\nabla_p\bar F))(e^{2\phi}+\vert p\vert^2)^\gamma f{\rm d}x{\rm d}p\nonumber\\
		:=&\CI_1+\CI_2+\CI_3+\CI_4.
\end{align}
In view of \eqref{tc}, it is clear that $\bar\phi'(t)<\frac{1}{2}\bar\phi'(\infty)<0$ for $t>T_c$. The {\it a priori assumption}
\eqref{aps} implies $\left\|\pa_t\Phi\right\|_{L^\infty}$ is sufficiently small. Consequently, we have
\begin{equation}\label{3.2}
	\CI_1\leq \frac{\gamma}{4}\bar\phi'(\infty)\int_{\mathbb R^6}e^{2\phi}(e^{2\phi}+\vert p\vert^2)^{\gamma-1}\vert f\vert^2{\rm d}x{\rm d}p<0.
\end{equation}
Similar to \eqref{z-ip} in Subsection \ref{sec-ft-eng}, it follows that
\begin{align}\label{3.3}
|\CI_2|+\CI_3+|\CI_4|
\leq&-\frac{1}{2}\int_{\mathbb R^6}e^{2\phi}(e^{2\phi}+\vert p\vert^2)^{\gamma-\frac{1}{2}}(e^{2\phi}\vert\nabla_p f\vert^2+\vert p\cdot\nabla_p f\vert^2){\rm d}x{\rm d}p\nonumber\\
&+C e^{\bar\phi}\int_{\mathbb R^6}(e^{2\phi}+\vert p\vert^2)^\gamma\vert f\vert^2{\rm d}x{\rm d}p+Ce^{\bar\phi}\left\|\nabla_x\Phi\right\|_{L^2_x}^2+Ce^{\bar\phi}\left\|\Phi\right\|^2_{L^2_x}.
\end{align}
Next, by uing the argument for \eqref{el-sg} to estimate\eqref{phi-ip}, we obtain
\begin{equation}\label{3.4}
		\frac{d}{dt}\int_{\mathbb R^3}(\vert\partial_t\Phi\vert^2+\vert\nabla_x\Phi\vert^2){\rm d}x\leq Ce^{2\bar\phi}(\mathcal E_{0,0}^{\delta_1}(t)+\mathcal E_{0,0}^{\delta_2}(t))+Ce^{2\bar\phi}\left\|\Phi\right\|_{L^2_x}^2.
\end{equation}
Substituting \eqref{3.2}, \eqref{3.3} and \eqref{3.4} into \eqref{3.1}, we have
\begin{equation}\label{3.5}
	\frac{\rm d}{dt}\mathcal E_{0,0}^\gamma(t)+\mathcal D_{0,0}^\gamma(t)\leq Ce^{\bar\phi}(\mathcal E_{0,0}^\gamma(t)+\mathcal E_{0,0}^{\delta_1}(t)+\mathcal E_{0,0}^{\delta_2}(t))+Ce^{\bar\phi}\left\|\Phi\right\|_{L^2_x}^2.
\end{equation}
Taking the summation of \eqref{3.5} with $\gamma=\delta_1$ and $\gamma=\delta_2$, then one has
\begin{equation}
	\frac{\rm d}{dt}(\mathcal E_{0,0}^{\delta_1}(t)+\mathcal E_{0,0}^{\delta_2}(t))+(\mathcal D_{0,0}^{\delta_1}(t)+\mathcal D_{0,0}^{\delta_2}(t))\leq Ce^{\bar\phi}(\mathcal E_{0,0}^{\delta_1}(t)+\mathcal E_{0,0}^{\delta_2}(t))+Ce^{\bar\phi}\left\|\Phi\right\|_{L^2_x}^2.\notag
\end{equation}
As a consequence, Gr\"{o}nwall's inequality over $(T_c,t)$ gives
\begin{align}\label{3.7}
\mathcal	E_{0,0}^{\delta_1}(t)+\mathcal E_{0,0}^{\delta_2}(t)+\int_{T_c}^t(\mathcal D_{0,0}^{\delta_1}(s)+\mathcal D_{0,0}^{\delta_2}(s)){\rm d}s&\leq	C(T_c)(\mathcal E_{0,0}^{\delta_1}(T_c)+\mathcal E_{0,0}^{\delta_2}(T_c))+C\left\|\Phi(T_c)\right\|^2_{L^2_x},
\end{align}
for $t>T_c$. This together with \eqref{z-sum3}  implies
\begin{equation}\label{z-lt}
	\begin{split}
\mathcal	E_{0,0}^{\delta_1}(t)+\mathcal E_{0,0}^{\delta_2}(t)+\int_{0}^t(\mathcal D_{0,0}^{\delta_1}(s)+\mathcal D_{0,0}^{\delta_2}(s)){\rm d}s
	\leq C(T_c)(\mathcal E_{0,0}^{\delta_2}(0)+\left\|\Phi_0\right\|_{L^2_x}^2).
	\end{split}
\end{equation}

\noindent\underline{{\it Step 2. First-order energy estimates.}}
In this step, we will estimate the corresponding terms on the right-hand side of \eqref{2.12} and \eqref{xd-ip} over the interval $(T_c,\infty)$ denoted by:
\begin{align}
\CI_5=&\frac{1}{2}\int_{\mathbb R^6}\vert\partial_{p_i}f\vert^2\partial_t(e^{3\bar\phi}(e^{2\phi}+\vert p\vert^2)^\gamma){\rm d}x{\rm d}p,\nonumber\\
\CI_6=&\int_{\mathbb R^6}(\nabla_x\partial_{p_i}\sqrt{e^{2\phi}+\vert p\vert^2}\cdot\nabla_p f-\nabla_p\partial_{p_i}\sqrt{e^{2\phi}+\vert p\vert^2}\cdot\nabla_x f)e^{3\bar\phi}(e^{2\phi}+\vert p\vert^2)\partial_{p_i}f{\rm d}x{\rm d}p,\nonumber\\
\CI_7=&\int_{\mathbb R^6}\partial_{p_i}(\nabla_x\sqrt{e^{2\phi}+\vert p\vert^2}\cdot\nabla_p\bar F)e^{3\bar\phi}(e^{2\phi}+\vert p\vert^2)^\gamma \partial_{p_i}f{\rm d}x{\rm d}p,\nonumber\\
\CI_8=&\int_{\mathbb R^6}\partial_{p_i}(e^{2\phi}\nabla_p\cdot(\Lambda_{\phi,p}\nabla_p f))e^{3\bar\phi}(e^{2\phi}+\vert p\vert^2)^\gamma \partial_{p_i}f{\rm d}x{\rm d}p,\nonumber\\
\CI_9=&\int_{\mathbb R^6}\partial_{p_i}(e^{2\phi}\nabla_p\cdot(\Lambda_{\phi,p}\nabla_p\bar F)-e^{2\bar\phi}\nabla_p\cdot(\Lambda_{\bar\phi,p}\nabla_p\bar F))e^{3\bar\phi}(e^{2\phi}+\vert p\vert^2)^\gamma \partial_{p_i}f{\rm d}x{\rm d}p,\nonumber\\
\CI_{10}=&\frac{1}{2}\int_{\mathbb R^6}\vert\partial_{x_i}f\vert^2\partial_t(e^{\bar\phi}(e^{2\phi}+\vert p\vert^2)^\gamma){\rm d}x{\rm d}p,\nonumber\\
\CI_{11}=&\int_{\mathbb R^6}\big(\nabla_x\partial_{x_i}\sqrt{e^{2\phi}+\vert p\vert^2}\cdot\nabla_pf-\nabla_p\partial_{x_i}\sqrt{e^{2\phi}+\vert p\vert^2}\cdot\nabla_xf\big)e^{\bar\phi}(e^{2\phi}+\vert p\vert^2)^\gamma\partial_{x_i}f{\rm d}x{\rm d}p,\nonumber\\
\CI_{12}=&\int_{\mathbb R^6}\nabla_x\partial_{x_i}\sqrt{e^{2\phi}+\vert p\vert^2}\cdot\nabla_p\bar Fe^{\bar\phi}(e^{2\phi}+\vert p\vert^2)^\gamma \partial_{x_i}f{\rm d}x{\rm d}p,\nonumber\\
\CI_{13}=&\int_{\mathbb R^6}\partial_{x_i}(e^{2\phi}\nabla_p\cdot(\Lambda_{\phi,p}\nabla_p f))e^{\bar\phi}(e^{2\phi}+\vert p\vert^2)^\gamma \partial_{x_i}f{\rm d}x{\rm d}p,\nonumber\\
\CI_{14}=&\int_{\mathbb R^6}\partial_{x_i}(e^{2\phi}\nabla_p\cdot(\Lambda_{\phi,p}\nabla_p\bar F))e^{\bar\phi}(e^{2\phi}+\vert p\vert^2)^\gamma \partial_{x_i}f{\rm d}x{\rm d}p.\nonumber
	\end{align}
For $\CI_5$, since $\bar\phi'(t)\leq\frac{\bar\phi'(\infty)}{2}<0$ and $\left\|\partial_t\phi\right\|_{L^\infty_x}$ is sufficiently small,   there exists constant $C>0$ such that
\begin{equation}
	\CI_5\leq C\bar\phi'(\infty)\int_{\mathbb R^6}e^{3\bar\phi}(e^{2\phi}+\vert p\vert^2)^\gamma\vert\partial_{p_i}f\vert^2{\rm d}x{\rm d}p.\notag
\end{equation}
For $\CI_6$, we directly have by the {\it a priori} assumption \eqref{aps} that
\begin{align}
		|\CI_6|\leq& C\int_{\mathbb R^6}e^{2\phi}(e^{2\phi}+\vert p\vert^2)^{\frac{\gamma}{2}-\frac{1}{4}}\vert\nabla_pf\vert e^{\frac{\bar\phi}{2}}(e^{2\phi}+\vert p\vert^2)^{\frac{\gamma}{2}}\vert\nabla_xf\vert{\rm d}x{\rm d}p\nonumber\\
		&+C\left\|\nabla_x\Phi\right\|_{L^\infty_x}\int_{\mathbb R^6}e^{2\phi}(e^{2\phi}+\vert p\vert^2)^{\gamma-\frac{1}{2}}\vert\nabla_p f\vert^2{\rm d}x{\rm d}p\nonumber\\
		\leq& \eta\int_{\mathbb R^6}e^{\bar\phi}(e^{2\phi}+\vert p\vert^2)^\gamma\vert\nabla_xf\vert^2{\rm d}x{\rm d}p+C_\eta\int_{\mathbb R^6}e^{4\phi}(e^{2\phi}+\vert p\vert^2)^{\gamma-\frac{1}{2}}\vert\nabla_pf\vert^2{\rm d}x{\rm d}p.\notag
\end{align}
By Lemmas \ref{lemma A3}, \ref{lemma A4} and \ref{lemma A5}, $\CI_7$, $\CI_8$ and $\CI_9$ can be bounded as
\begin{align}\label{3.10}
	|\CI_7|+\CI_8+|\CI_9|&\leq-\frac{1}{2}\int_{\mathbb R^6}e^{3\bar\phi+2\phi}(e^{2\phi}+\vert p\vert^2)^{\gamma-\frac{1}{2}}(e^{2\phi}\vert\nabla_p\partial_{p_i}f\vert^2+\vert p\cdot\nabla_p\partial_{p_i}f\vert^2){\rm d}x{\rm d}p\nonumber\\
	&\quad+Ce^{\bar\phi}\sum_{m+n\leq1}\mathcal E_{m,n}^\gamma(t)+C\mathcal D_{0,0}^\gamma(t)
+C e^{\bar\phi}\left\|\Phi\right\|_{L^2_x}^2.\notag
	\end{align}
Thus
\begin{align}
\frac{d}{dt}&\int_{\mathbb R^6}e^{3\bar\phi}(e^{2\phi}+\vert p\vert^2)^\gamma\vert\partial_{p_i}f\vert^2{\rm d}x{\rm d}p-C\bar\phi'(\infty)\int_{\mathbb R^6}e^{3\bar\phi}(e^{2\phi}+\vert p\vert^2)^\gamma\vert\partial_{p_i}f\vert^2{\rm d}x{\rm d}p\notag\\
&\quad+\int_{\mathbb R^6}e^{3\bar\phi+2\phi}(e^{2\phi}+\vert p\vert^2)^{\gamma-\frac{1}{2}}(e^{2\phi}\vert\nabla_p\partial_{p_i}f\vert^2+\vert p\cdot\nabla_p\partial_{p_i}f\vert^2){\rm d}x{\rm d}p\notag\\
		\leq& \eta\int_{\mathbb R^6}e^{\bar\phi}(e^{2\phi}+\vert p\vert^2)^{\gamma}\vert\nabla_xf\vert^2{\rm d}x{\rm d}p+Ce^{\bar\phi}\sum_{m+n\leq1}\mathcal E_{m,n}^\gamma(t)+C\mathcal D_{0,0}^\gamma(t)
+C e^{\bar\phi}\left\|\Phi\right\|_{L^2_x}^2.\notag
\end{align}
Next, for $\CI_{10}$, it is straightforward to show
\begin{equation*}
	\CI_{10}\leq C\bar\phi'(\infty)\int_{\mathbb R^6}e^{\bar\phi}(e^{2\phi}+\vert p\vert^2)^\gamma\vert\partial_{x_i}f\vert^2{\rm d}x{\rm d}p.
\end{equation*}
For $\CI_{11}$, unlike the case in finite time interval as shown in  \eqref{2.20}, we need to use the smallness of $\left\|\nabla_x\Phi\right\|_{L^\infty_x}$ from the {\it a priori} assumption \eqref{aps} to deduce
\begin{align*}
		|\CI_{11}|&\leq C\int_{\mathbb R^6}e^{2\phi}(e^{2\phi}+\vert p\vert^2)^{\frac{\gamma}{2}-\frac{1}{4}}\vert\nabla_pf\vert e^{\frac{\bar\phi}{2}}(e^{2\phi}+\vert p\vert^2)^{\frac{\gamma}{2}}\vert\partial_{x_i}f\vert(\vert\nabla_x\Phi\vert+\vert\nabla_x^2\Phi\vert){\rm d}x{\rm d}p\nonumber\\
		&\quad+C\int_{\mathbb R^6}e^{\bar\phi}(e^{2\phi}+\vert p\vert^2)^\gamma\vert\nabla_xf\vert^2\vert\nabla_x\Phi\vert{\rm d}x{\rm d}p\nonumber\\
		&\leq C\left\|(\nabla_x\Phi,\nabla_x^2\Phi)\right\|_{L^\infty_x}\int_{\mathbb R^6} e^{\bar\phi}(e^{2\phi}+\vert p\vert^2)^{\gamma}\vert\nabla_xf\vert^2{\rm d}x{\rm d}p+C\int_{\mathbb R^6} e^{2\phi}(e^{2\phi}+\vert p\vert^2)^{\gamma-\frac{1}{2}}\vert\nabla_pf\vert^2{\rm d}x{\rm d}p\nonumber\\
		&\leq C\eps_0\int_{\mathbb R^6} e^{\bar\phi}(e^{2\phi}+\vert p\vert^2)^{\gamma}\vert\nabla_xf\vert^2{\rm d}x{\rm d}p+C\int_{\mathbb R^6} e^{4\phi}(e^{2\phi}+\vert p\vert^2)^{\gamma-\frac{1}{2}}\vert\nabla_pf\vert^2{\rm d}x{\rm d}p.
\end{align*}
Next, Lemmas \ref{lemma A3}, \ref{lemma A4} and \ref{lemma A5} give
\begin{align*}
|\CI_{12}|+\CI_{13}+|\CI_{14}|\leq&-\frac{1}{2}\int_{\mathbb R^6}e^{\bar\phi+2\phi}(e^{2\phi}+\vert p\vert^2)^{\gamma-\frac{1}{2}}(e^{2\phi}\vert\nabla_p\partial_{x_i}f\vert^2+\vert p\cdot\nabla_p\partial_{x_i}f\vert^2){\rm d}x{\rm d}p\nonumber\\
	&+Ce^{\bar\phi}\sum_{m+n\leq1}\mathcal E_{m,n}^\gamma(t)+C\mathcal D_{0,0}^\gamma(t)
+C e^{\bar\phi}\left\|\Phi\right\|_{L^2_x}^2.
\end{align*}
Putting the above estimates together, we have
\begin{align}\label{3.12}
\frac{d}{dt}\int_{\mathbb R^6}&(e^{3\bar\phi}(e^{2\phi}+\vert p\vert^2)^\gamma\vert\partial_{p_i}f\vert^2+e^{\bar\phi}(e^{2\phi}+\vert p\vert^2)^\gamma\vert\partial_{x_i}f\vert^2){\rm d}x{\rm d}p\nonumber\\
		&-C\bar\phi'(\infty)\int_{\mathbb R^6}\big(e^{3\bar\phi}(e^{2\phi}+\vert p\vert^2)^\gamma\vert\partial_{p_i}f\vert^2+e^{\bar\phi}(e^{2\phi}+\vert p\vert^2)^\gamma\vert\partial_{x_i}f\vert^2\big){\rm d}x{\rm d}p\nonumber\\
		&+\int_{\mathbb R^6}e^{3\bar\phi+2\phi}(e^{2\phi}+\vert p\vert^2)^{\gamma-\frac{1}{2}}(e^{2\phi}\vert\nabla_p\partial_{p_i}f\vert^2+\vert p\cdot\nabla_p\partial_{p_i}f\vert^2){\rm d}x{\rm d}p\nonumber\\
		&+\int_{\mathbb R^6}e^{\bar\phi+2\phi}(e^{2\phi}+\vert p\vert^2)^{\gamma-\frac{1}{2}}(e^{2\phi}\vert\nabla_p\partial_{x_i} f\vert^2+\vert p\cdot\nabla_p\partial_{x_i}f\vert^2){\rm d}x{\rm d}p\nonumber\\
		\leq& Ce^{\bar\phi}\sum_{m+n\leq1}\mathcal E_{m,n}^\gamma(t)+C\mathcal D_{0,0}^\gamma(t)+C e^{\bar\phi}\left\|\Phi\right\|_{L^2_x}^2.
\end{align}
Similar to the derivation of \eqref{3.4}, we obtain the energy estimate for the gravitational field as follows
\begin{equation}\label{3.13}
		\frac{d}{dt}\int_{\mathbb R^3}(\vert\partial_t\partial_{x_i}\Phi\vert^2+\vert\nabla_x\partial_{x_i}\Phi\vert^2){\rm d}x\leq Ce^{\frac{\bar\phi}{2}}\sum_{ m+n\leq1}(\mathcal E_{m,n}^{\delta_1}(t)+\mathcal E_{m,n}^{\delta_2}(t)).
\end{equation}
From \eqref{3.12} and \eqref{3.13}, it follows directly that
\begin{align}
\frac{d}{dt}\sum_{m+n=1}&(\mathcal E_{m,n}^{\delta_1}(t)+\mathcal E_{m,n}^{\delta_2}(t))	-C\bar\phi'(\infty)\sum_{|\alpha|+|\beta|=1}\int_{\mathbb R^6}e^{(|\alpha|+3|\beta|)\bar\phi}(e^{2\phi}+\vert p\vert^2)^\gamma\vert\partial_x^\alpha\partial_p^\beta f\vert^2{\rm d}x{\rm d}p\notag\\
&\quad+\sum_{m+n=1}(\mathcal D_{m,n}^{\delta_1}(t)+\mathcal D_{m,n}^{\delta_2}(t))\notag\\
		\leq& Ce^{\frac{\bar\phi}{2}}\sum_{ m+n\leq1}(\mathcal E_{m,n}^{\delta_1}(t)+\mathcal E_{m,n}^{\delta_2}(t))+C(\mathcal D_{0,0}^{\delta_1}(t)+\mathcal D_{0,0}^{\delta_2}(t))+Ce^{\bar\phi}\left\|\Phi\right\|_{L^2_x}^2.\notag
\end{align}
Applying Gr\"{o}nwall's inequality to the above inequality and using \eqref{z-lt} yield
\begin{align}\label{3.15}
\sum_{m+n=1}&(\mathcal E_{m,n}^{\delta_1}(t)+\mathcal E_{m,n}^{\delta_2}(t))+\sum_{m+n=1}\int_0^t\big(\bar{\mathcal D}_{m,n}^{\delta_1}(s)+\bar{\mathcal D}_{m,n}^{\delta_2}(s)\big){\rm d}s\notag\\
		\leq& C(T_c)\sum_{m+n\leq1}\mathcal E_{m,n}^{\delta_2}(0)+C(T_c)\left\|\Phi_0\right\|_{L^2_x}^2.	
\end{align}
\underline{{\it Step 3. High-order energy estimates.}}
In this step, similar to \eqref{2.27}, we will esitmate the terms on  the right-hand side in the following equality:
\begin{align}\label{3.16}
		\frac{1}{2}\frac{d}{dt}\int_{\mathbb R^6}&e^{(\vert\alpha\vert+3\vert\beta\vert)\bar\phi}(e^{2\phi}+\vert p\vert^2)^\gamma\vert\partial_x^\alpha\partial_p^\beta f\vert^2{\rm d}x{\rm d}p\nonumber\\
		=&\frac{1}{2}\int_{\mathbb R^6}\vert\partial_x^\alpha\partial_p\beta f\vert^2\partial_t(e^{(\vert\alpha\vert+3\vert\beta\vert)\bar\phi}(e^{2\phi}+\vert p\vert^2)^\gamma){\rm d}x{\rm d}p\nonumber\\
		&+\int_{\mathbb R^6}e^{(\vert\alpha\vert+3\vert\beta\vert)\bar\phi}[\nabla_p\sqrt{e^{2\phi}+\vert p\vert^2}\cdot\nabla_x-\nabla_x\sqrt{e^{2\phi}+\vert p\vert^2}\cdot\nabla_p,\partial_x^\alpha\partial_p^\beta]f(e^{2\phi}+\vert p\vert^2)\partial_x^\alpha\partial_p^\beta f{\rm d}x{\rm d}p\nonumber\\
		&+\int_{\mathbb R^6}e^{(\vert\alpha\vert+3\vert\beta\vert)\bar\phi}\partial_x^\alpha\partial_p^\beta(\nabla_x\sqrt{e^{2\phi}+\vert p\vert^2}\cdot\nabla_p\bar F)(e^{2\phi}+\vert p\vert^2)^\gamma \partial_x^\alpha\partial_p^\beta f{\rm d}x{\rm d}p\nonumber\\
		&+\int_{\mathbb R^6}e^{(\vert\alpha\vert+3\vert\beta\vert)\bar\phi}\partial_x^\alpha\partial_p^\beta(e^{2\phi}\nabla_p\cdot(\Lambda_\phi\nabla_p f))(e^{2\phi}+\vert p\vert^2)^\gamma \partial_x^\alpha\partial_p^\beta f{\rm d}x{\rm d}p\nonumber\\
		&+\int_{\mathbb R^6}e^{(\vert\alpha\vert+3\vert\beta\vert)\bar\phi}\partial_x^\alpha\partial_p^\beta(e^{2\phi}\nabla_p\cdot(\Lambda_{\phi}\nabla_p\bar F)-e^{2\bar\phi}\nabla_p\cdot(\Lambda_{\bar\phi}\nabla_p\bar F))(e^{2\phi}+\vert p\vert^2)^\gamma \partial_x^\alpha\partial_p^\beta f{\rm d}x{\rm d}p\nonumber\\
		:=&\CI_{15}+\CI_{16}+\CI_{17}+\CI_{18}+\CI_{19}.\notag
\end{align}
First of all, it holds directly that
\begin{equation}\label{I15-2}
	\CI_{15}\leq C\bar\phi'(\infty)\int_{\mathbb R^6}e^{(\vert\alpha\vert+3\vert\beta\vert)\bar\phi}(e^{2\phi}+\vert p\vert^2)^\gamma\vert\partial_x^\alpha\partial_p^\beta f\vert^2{\rm d}x{\rm d}p.
\end{equation}
For $\CI_{17}$, $\CI_{18}$ and $\CI_{19}$, similar to \eqref{I17}, \eqref{I18} and \eqref{I19}, by  Lemmas \ref{lemma A3}, \ref{lemma A4} and
\ref{lemma A5},
we have
\begin{align}
|\CI_{17}|+\CI_{18}+|\CI_{19}|\leq& -\frac{1}{2}\int_{\mathbb R^6}e^{(\vert\alpha\vert+3\vert\beta\vert)\bar\phi+2\phi}(e^{2\phi}+\vert p\vert^2)^{\gamma-\frac{1}{2}}\big(e^{2\phi}\vert\nabla_p\partial_x^\alpha\partial_p^\beta f\vert^2+\vert p\cdot\na_p\partial_x^\alpha\partial_p^\beta f\vert^2\big){\rm d}x{\rm d}p\nonumber\\
&+Ce^{\bar\phi}\sum_{\substack{m+n\leq\vert\alpha\vert+\vert\beta\vert\\n<4}}\mathcal E^\gamma_{m,n}(t)+C\sum_{m+n<\vert\alpha\vert+\vert\beta\vert}\mathcal D_{m,n}^\gamma(t)+Ce^{\bar\phi}\left\|\Phi\right\|_{L^2_x}^2.\label{I17-19-2}
\end{align}
The estimation on  $\CI_{16}$ over the time interval $(T_c,+\infty)$ is different  from those in finite time interval given in  \eqref{2.27}. For this, the estimation is  divided into following three cases.

\noindent\underline{{\it Case 1. $\vert\alpha\vert+\vert\beta\vert=2$.}}
In this case, if $|\al|=0$ and $|\beta|=2$,
by using H\"{o}lder's inequality and Cauchy-Schwarz's inequality with $\eta>0$, we get
\begin{align}\label{I16-c21}
		|\CI_{16}&\chi_{|\al|=0,|\beta|=2}|\notag\\
		\leq& C\int_{\mathbb R^6}e^{3\bar\phi}(e^{2\phi}+\vert p\vert^2)^{\frac{\gamma}{2}}\vert\nabla_p^2f\vert\Big(e^{\frac{\bar\phi}{2}}(e^{2\phi}+\vert p\vert^2)^{\frac{\gamma}{2}}\vert\nabla_xf\vert e^{\frac{\bar\phi}{2}}+e^{\frac{\bar\phi}{2}+2\phi}(e^{2\phi}+\vert p\vert^2)^{\frac{\gamma}{2}-\frac{1}{4}}\vert\nabla_p\nabla_xf\vert\Big){\rm d}x{\rm d}p\nonumber\\
		&+C\int_{\mathbb R^6} e^{3\bar\phi}(e^{2\phi}+\vert p\vert^2)^{\frac{\gamma}{2}}\vert\nabla_p^2f\vert e^{2\phi}(e^{2\phi}+\vert p\vert^2)^{\frac{\gamma}{2}-\frac{1}{4}}\vert\nabla_pf\vert{\rm d}x{\rm d}p\notag\\&+C\int_{\mathbb R^6}e^{3\bar\phi+4\phi}(e^{2\phi}+\vert p\vert^2)^{\gamma-\frac{1}{2}}\vert\nabla_p^2f\vert^2{\rm d}x{\rm d}p\nonumber\\
		\leq&\eta\int_{\mathbb R^6}e^{6\bar\phi}(e^{2\phi}+\vert p\vert^2)^\gamma\vert\nabla_p^2f\vert^2{\rm d}x{\rm d}p+C_\eta\int_{\mathbb R^6}e^{\bar\phi+4\phi}(e^{2\phi}+\vert p\vert^2)^{\gamma-\frac{1}{2}}\vert\nabla_p\nabla_xf\vert^2{\rm d}x{\rm d}p\nonumber\\
		&+C_\eta\int_{\mathbb R^6}e^{4\phi}(e^{2\phi}+\vert p\vert^2)^{\gamma-\frac{1}{2}}\vert\nabla_pf\vert^2{\rm d}x{\rm d}p+C_\eta e^{\bar\phi}\int_{\mathbb R^6}e^{\bar\phi}(e^{2\phi}+\vert p\vert^2)^\gamma\vert\nabla_xf\vert^2{\rm d}x{\rm d}p\nonumber\\
		&+C\int_{\mathbb R^6}e^{3\bar\phi+4\phi}(e^{2\phi}+\vert p\vert^2)^{\gamma-\frac{1}{2}}\vert\nabla_p^2f\vert^2{\rm d}x{\rm d}p\nonumber\\
		\leq&\eta\int_{\mathbb R^6}e^{6\bar\phi}(e^{2\phi}+\vert p\vert^2)^\gamma\vert\nabla_p^2f\vert^2{\rm d}x{\rm d}p+C_\eta\sum_{m+n\leq1}(e^{\bar\phi}\mathcal E_{m,n}^\gamma(t)+\mathcal D_{m,n}^\gamma(t)).	
	\end{align}
If $\vert\alpha\vert=\vert\beta\vert=1$, by using H\"{o}lder's inequality and Cauchy-Schwarz's inequality with $\eta>0$ again,
we obtain
\begin{align}\label{I16-c22}
		|\CI_{16}&\chi_{\vert\alpha\vert=\vert\beta\vert=1}|\notag\\
		\leq& C\int_{\mathbb R^6}e^{4\bar\phi}(e^{2\phi}+\vert p\vert^2)^\gamma\vert\nabla_x\nabla_p f\vert\Big((e^{2\phi}+\vert p\vert^2)^{-\frac{1}{2}}\vert\nabla_x^2 f\vert+e^{2\phi}(e^{2\phi}+\vert p\vert^2)^{-1}\vert\nabla_x\Phi\vert\vert\nabla_p\nabla_x f\vert\nonumber\\
		&\qquad\qquad+e^{2\phi}(e^{2\phi}+\vert p\vert^2)^{-\frac{3}{2}}\vert\nabla_x\Phi\vert\vert\nabla_x f\vert\Big){\rm d}x{\rm d}p\nonumber\\
		&+C\int_{\mathbb R^6}e^{4\bar\phi+2\phi}(e^{2\phi}+\vert p\vert^2)^\gamma\vert\nabla_x\nabla_p f\vert\Big((e^{2\phi}+\vert p\vert^2)^{-\frac{1}{2}}(\vert\nabla_x\Phi\vert^2+\vert\nabla_x^2\Phi)\vert\nabla_p^2f\vert\nonumber\\
		&\qquad\qquad+(e^{2\phi}+\vert p\vert^2)^{-1}\vert\nabla_x\Phi\vert\vert\nabla_x\nabla_pf\vert+(e^{2\phi}+\vert p\vert^2)^{-1}(\vert\nabla_x\Phi\vert^2+\vert\nabla_x^2\Phi\vert)\vert\nabla_pf\vert\Big){\rm d}x{\rm d}p\nonumber\\
		\leq& \eta\int_{\mathbb R^6}e^{2\bar\phi}(e^{2\phi}+\vert p\vert^2)^\gamma\vert\nabla_x^2f\vert^2{\rm d}x{\rm d}p+C_\eta\int_{\mathbb R^6}e^{\bar\phi+4\phi}(e^{2\phi}+\vert p\vert^2)^{\gamma-\frac{1}{2}}\vert\nabla_x\nabla_pf\vert^2{\rm d}x{\rm d}p\nonumber\\
		&+C\int_{\mathbb R^6}e^{3\bar\phi+4\phi}(e^{2\phi}+\vert p\vert^2)^{\gamma-\frac{1}{2}}\vert\nabla_p^2f\vert^2{\rm d}x{\rm d}p+Ce^{\bar\phi}\int_{\mathbb R^6}e^{\bar\phi}(e^{2\phi}+\vert p\vert^2)^\gamma\vert\nabla_xf\vert^2{\rm d}x{\rm d}p\nonumber\\
		&+C\int_{\mathbb R^6}e^{4\phi}(e^{2\phi}+\vert p\vert^2)^{\gamma-\frac{1}{2}}\vert\nabla_pf\vert^2{\rm d}x{\rm d}p\nonumber\\
		\leq& \eta\int_{\mathbb R^6}e^{2\bar\phi}(e^{2\phi}+\vert p\vert^2)^\gamma\vert\nabla_x^2f\vert^2{\rm d}x{\rm d}p+C\int_{\mathbb R^6}e^{6\bar\phi}(e^{2\phi}+\vert p\vert^2)^\gamma\vert\nabla_p^2f\vert^2{\rm d}x{\rm d}p
\nonumber\\
		&+C_\eta\sum_{m+n\leq1}\mathcal D_{m,n}^\gamma(t)
+Ce^{\bar\phi}\int_{\mathbb R^6}e^{\bar\phi}(e^{2\phi}+\vert p\vert^2)^\gamma\vert\nabla_xf\vert^2{\rm d}x{\rm d}p.
\end{align}
If $\vert\alpha\vert=2,$
by using H\"{o}lder's inequality and \eqref{aps}, it follows that
\begin{align}\label{I16-c23}
		|\CI_{16}\chi_{\vert\alpha\vert=2}|&\leq C\left\|(\nabla_x\Phi,\nabla_x^2\Phi\right\|_{L^\infty_x}\int_{\mathbb R^6}e^{\frac{\bar\phi}{2}}(e^{2\phi}+\vert p\vert^2)^{\frac{\gamma}{2}}\vert\nabla_xf\vert e^{\bar\phi}(e^{2\phi}+\vert p\vert^2)^{\frac{\gamma}{2}}\vert\nabla_x^2f\vert{\rm d}x{\rm d}p\nonumber\\
		&\quad+C\left\|\nabla_x\Phi\right\|_{L^\infty_x}\int_{\mathbb R^6}e^{2\bar\phi}(e^{2\phi}+\vert p\vert^2)^{\gamma}\vert\nabla_x^2f\vert^2{\rm d}x{\rm d}p\nonumber\\
		&\quad+C\int_{\mathbb R^6}e^{\bar\phi}(e^{2\phi}+\vert p\vert^2)^{\gamma}\vert\nabla_x^2f\vert e^{2\bar\phi}(\vert\nabla_x\Phi\vert^2+\vert\nabla_x\Phi\vert\nabla_x^2\Phi\vert+\vert\nabla_x^2\Phi\vert)\vert\nabla_pf\vert{\rm d}x{\rm d}p\nonumber\\
		&\quad+C\int_{\mathbb R^6} e^{\bar\phi}(e^{2\phi}+\vert p\vert^2)^{\gamma}\vert\nabla_x^2f\vert e^{2\bar\phi}(\vert\nabla_x\Phi\vert^2+\vert\nabla_x^2f\vert)\vert\nabla_x\nabla_p f\vert{\rm d}x{\rm d}p\nonumber\\
		&\leq Ce^{\bar\phi}\int_{\mathbb R^6}\Big(e^{\bar\phi}(e^{2\phi}+\vert p\vert^2)^{\gamma}\vert\nabla_xf\vert+e^{2\bar\phi}(e^{2\phi}+\vert p\vert^2)^{\gamma}\vert\nabla_x^2f\vert+e^{3\bar\phi}(e^{2\phi}+\vert p\vert^2)^{\gamma}\vert\nabla_pf\vert^2\Big){\rm d}x{\rm d}p\nonumber\\
		&\quad+C\left\|\nabla_x\Phi\right\|_{H^2_x}^2\int_{\mathbb R^6}\Big(e^{2\bar\phi}(e^{2\phi}+\vert p\vert^2)^{\gamma}
	\vert\nabla_x^2f\vert^2+e^{4\bar\phi}(e^{2\phi}+\vert p\vert^2)^{\gamma}\vert\nabla_x\nabla_pf\vert^2\Big){\rm d}x{\rm d}p\nonumber\\
	&\leq C\eps_0\int_{\mathbb R^6}\Big(e^{2\bar\phi}(e^{2\phi}+\vert p\vert^2)^{\gamma}
	\vert\nabla_x^2f\vert^2+e^{4\bar\phi}(e^{2\phi}+\vert p\vert^2)^{\gamma}\vert\nabla_x\nabla_pf\vert^2\Big){\rm d}x{\rm d}p\nonumber\\
	&\quad+Ce^{\bar\phi}\sum_{m+n\leq2}\mathcal E_{m,n}^\gamma(f,\Phi)(t).
\end{align}
By combining the above estimates \eqref{I15-2}, \eqref{I17-19-2}, \eqref{I16-c21}, \eqref{I16-c22} and \eqref{I16-c23},
and applying \eqref{el-sg} in \eqref{2.54} with $|\al|=2$, we conclude
\begin{align}
\frac{d}{dt}\sum_{\substack{ m+n=2}}&(\mathcal E_{m,n}^{\delta_1}(t)+\mathcal E_{m,n}^{\delta_2}(t))+\sum_{\substack{ m+n=2}}(\mathcal D_{m,n}^{\delta_1}(t)+\mathcal D_{m,n}^{\delta_2}(t))\notag\\
&\quad-C\bar\phi'(\infty)\sum_{|\al|+|\bet|=2}\int_{\mathbb R^6} e^{(|\alpha|+3|\beta|)\bar\phi}\Big((e^{2\phi}+|p|^2)^{\delta_1}+(e^{2\phi}+|p|^2)^{\delta_2}\Big)|\partial_x^\alpha\partial_p^\beta f|^2{\rm d}x{\rm d}p\notag\\
		\leq& Ce^{\frac{\bar\phi}{2}}\sum_{m+n\leq2}(\mathcal E_{m,n}^{\delta_1}(t)+\mathcal E_{m,n}^{\delta_2}(t))+C\sum_{ m+n\leq1}(\mathcal D_{m,n}^{\delta_1}(t)+\mathcal D_{m,n}^{\delta_2}(t))+Ce^{\bar\phi}\left\|\Phi\right\|_{L^2_x}^2,\notag
\end{align}
which, together with \eqref{loc-eng}, \eqref{3.7} and \eqref{3.15}, further implies for  $t\in(T_c,+\infty)$,
\begin{align}\label{3.26}
\sum_{\substack{m+n=2}}&(\mathcal E_{m,n}^{\delta_1}(t)+\mathcal E_{m,n}^{\delta_2}(t))
+\sum_{\substack{m+n=2}}\int_0^t(\bar{\mathcal D}_{m,n}^{\delta_1}(s)+\bar{\mathcal D}_{m,n}^{\delta_2}(s)){\rm d}s\notag\\
		\leq& C(T_c)\Big(\sum_{m+n\leq2}\mathcal E_{m,n}^{\delta_2}(0)+\left\|\Phi_0\right\|_{L^2_x}^2\Big).
\end{align}
\underline{\noindent{\it Case 2. $\vert\alpha\vert+\vert\beta\vert=3$.}} First, we consider the subcase when $|\alpha|=0$ and $|\beta|=3$. For this, we have
\begin{align}\label{i16-c30}
		| \CI_{16}&\chi_{|\alpha|=0,|\beta|=3}|\notag\\
		&\leq C\int_{\mathbb R^6}e^{9\bar\phi}(e^{2\phi}+\vert p\vert^2)^\gamma|\partial_p^\beta f|\Big((e^{2\phi}+\vert p\vert^2)^{-\frac{3}{2}}|\nabla_xf|+(e^{2\phi}+\vert p\vert^2)^{-1}|\nabla_x\nabla_pf|\notag\\
		&\qquad\qquad+(e^{2\phi}+\vert p\vert^2)^{-\frac{1}{2}}|\nabla_x\nabla_p^2f|\Big){\rm d}x{\rm d}p\notag\\
		&\quad+C\int_{\mathbb R^6}e^{9\bar\phi}(e^{2\phi}+\vert p\vert^2)^\gamma|\partial_p^\beta f| e^{\phi}|\nabla_x\Phi|\Big((e^{2\phi}+\vert p\vert^2)^{-\frac{3}{2}}|\nabla_pf|+(e^{2\phi}+\vert p\vert^2)^{-1}|\nabla_p^2f|\notag\\
		&\qquad\qquad+(e^{2\phi}+\vert p\vert^2)^{-\frac{1}{2}}|\nabla_p^3f|\Big)\notag\\
		&\leq C\int_{\mathbb R^6}e^{3\bar\phi+2\phi}(e^{2\phi}+\vert p\vert^2)^{\frac{\gamma}{2}-\frac{1}{4}}|\partial_p^\beta f|(e^{2\phi}+\vert p\vert^2)^{\frac{\gamma}{2}}\Big(e^{\frac{\bar\phi}{2}}|\nabla_xf| e^{\bar\phi}+e^{\frac{9\bar\phi}{2}}|\nabla_x\nabla_pf| e^{\frac{\bar\phi}{2}}\notag\\
		&\qquad\qquad+e^{2\bar\phi+2\phi}(e^{2\phi}+|p|^2)^{-\frac{1}{4}}|\nabla_x\nabla_p^2f|\Big){\rm d}x{\rm d}p\notag\\
		&\quad+C\int_{\mathbb R^6}e^{3\bar\phi+2\phi}(e^{2\phi}+\vert p\vert^2)^{\frac{\gamma}{2}-\frac{1}{4}}|\partial_p^\beta f|(e^{2\phi}+\vert p\vert^2)^{\frac{\gamma}{2}}\Big(e^{\frac{3\bar\phi}{2}}|\nabla_pf|e^{\bar\phi}+e^{3\bar\phi}|\nabla_p^2f|e^{\frac{\bar\phi}{2}}\notag\\
		&\qquad\qquad+e^{3\bar\phi+2\phi}(e^{2\phi}+|p|^2)^{-\frac{1}{4}}|\nabla_p^3f|\Big){\rm d}x{\rm d}p\notag\\
		&\leq Ce^{\bar\phi}\sum_{m+n\leq2}\mathcal E_{m,n}^\gamma(t)+C\sum_{m+n=2}\mathcal D_{m,n}^\gamma(t).
\end{align}
On the other hand, if $\vert\alpha\vert=1$ and $\vert\beta\vert=2$, by Sobolev's inequality and H\"{o}lder's inequality, we deduce that
\begin{align}\label{i16-c31}
		|\CI_{16}&\chi_{\vert\alpha\vert=1,\vert\beta\vert=2}|\notag\\
		\leq& C\int_{\mathbb R^6}e^{7\bar\phi}(e^{2\phi}+\vert p\vert^2)^{\gamma}\vert\partial_x\partial_p^2f\vert\Big((e^{2\phi}+\vert p\vert^2)^{-1}\vert\nabla_xf\vert+(e^{2\phi}+\vert p\vert^2)^{-\frac{1}{2}}\vert\nabla_x\nabla_p f\vert\nonumber\\
		&\qquad+\vert\nabla_x\nabla_p^2f\vert\nabla_x\Phi\vert+(e^{2\phi}+\vert p\vert^2)^{-1}\vert\nabla_x^2f\vert+(e^{2\phi}+\vert p\vert^2)^{-\frac{1}{2}}\vert\nabla_x^2\nabla_pf\vert\Big){\rm d}x{\rm d}p\nonumber\\
		&+C\int_{\mathbb R^6}e^{7\bar\phi}(e^{2\phi}+\vert p\vert^2)^{\gamma}\vert\partial_x\partial_p^2f\vert e^{2\phi}\Big((e^{2\phi}+\vert p\vert^2)^{-\frac{1}{2}}\vert\nabla_p\nabla_p^2f\vert\notag\\&\qquad\qquad+(e^{2\phi}+\vert p\vert^2)^{-1}\vert\nabla_p^2f\vert+(e^{2\phi}+\vert p\vert^2)^{-\frac{3}{2}}\vert\nabla_pf\vert\Big){\rm d}x{\rm d}p\nonumber\\
		\leq& C\int_{\mathbb R^6}e^{\frac{7\bar\phi}{2}}(e^{2\phi}+\vert p\vert^2)^{\frac{\gamma}{2}}\vert\partial_x\partial_p^2f\vert\Big(e^{\frac{\bar\phi}{2}}(e^{2\phi}+\vert p\vert^2)^{\frac{\gamma}{2}}\vert\nabla_xf\vert e^{\bar\phi}+e^{2\bar\phi}(e^{2\phi}+\vert p\vert^2)^{\frac{\gamma}{2}}\vert\nabla_x\nabla_pf\vert e^{\frac{\bar\phi}{2}}\nonumber\\
		&\qquad+\vert\nabla_x\Phi\vert e^{\frac{7\bar\phi}{2}}(e^{2\phi}+\vert p\vert^2)^{\frac{\gamma}{2}}\vert\nabla_x\nabla_p^2f\vert+e^{\bar\phi}(e^{2\phi}+\vert p\vert^2)^{\frac{\gamma}{2}}\vert\nabla_x^2f\vert e^{\frac{\bar\phi}{2}}\nonumber\\
		&\qquad+e^{\bar\phi+2\phi}(e^{2\phi}+\vert p\vert^2)^{\frac{\gamma}{2}-\frac{1}{4}}\vert\nabla_x^2\nabla_pf\vert\Big){\rm d}x{\rm d}p\nonumber\\
		&+C\int_{\mathbb R^6}e^{\frac{7\bar\phi}{2}}(e^{2\phi}+\vert p\vert^2)^{\frac{\gamma}{2}}\vert\nabla_x\nabla_p^2f\vert\Big(e^{3\bar\phi+2\phi}(e^{2\phi}+\vert p\vert^2)^{\frac{\gamma}{2}-\frac{1}{4}}\vert\nabla_p\nabla_p^2f\vert\nonumber\\
		&\qquad\qquad+e^{3\bar\phi}(e^{2\phi}+\vert p\vert^2)^{\frac{\gamma}{2}}\vert\nabla_p^2f\vert e^{\frac{\bar\phi}{2}}+e^{\frac{3\bar\phi}{2}}(e^{2\phi}+\vert p\vert^2)^{\frac{\gamma}{2}}\vert\nabla_pf\vert e^{\bar\phi}\Big){\rm d}x{\rm d}p\nonumber\\
	\leq& C\eps_0\int_{\mathbb R^6}e^{7\bar\phi}(e^{2\phi}+\vert p\vert^2)^{\gamma}\vert\nabla_x\nabla_p^2f\vert^2{\rm d}x{\rm d}p\notag\\&+Ce^{\bar\phi}\Big(\int_{\mathbb R^6}e^{7\bar\phi}(e^{2\phi}+\vert p\vert^2)^{\gamma}\vert\nabla_x\nabla_p^2f\vert^2{\rm d}x{\rm d}p+\sum_{m+n\leq2}\mathcal E^\gamma_{m,n}(t)\Big)\nonumber\\
		&+C(\mathcal D_{2,0}^\gamma(t)+\mathcal D_{0,2}^\gamma(t)).
\end{align}
Moreover, if $\vert\alpha\vert=2$ and $\vert\beta\vert=1$,
by Sobolev's inequality and Cauchy-Schwarz's inequality with $\eta>0$, we get
\begin{align}\label{i16-c32}
		|\CI_{16}&\chi_{\vert\alpha\vert=2,\vert\beta\vert=1}|\notag\\
		&\leq C\int_{\mathbb R^6}e^{5\bar\phi}(e^{2\phi}+\vert p\vert^2)^{\gamma}\vert\nabla_x^2\nabla_pf\vert\nonumber\\
		&\qquad\qquad\times\Big((e^{2\phi}+\vert p\vert^2)^{-\frac{1}{2}}\vert\nabla_xf\vert+(e^{2\phi}+\vert p\vert^2)^{-\frac{1}{2}}\vert\nabla_x^2f\vert+(e^{2\phi}+\vert p\vert^2)^{-\frac{1}{2}}\vert\nabla_x^3f\vert+\vert\nabla_x\nabla_pf\vert\Big){\rm d}x{\rm d}p\nonumber\\
		&\quad+C\int_{\mathbb R^6}e^{5\bar\phi}(e^{2\phi}+\vert p\vert^2)^{\gamma}\vert\nabla_x^2\nabla_pf\vert\Big((1+\vert\nabla_x^3\Phi\vert)|\nabla_pf\vert+e^\phi(1+|\nabla_x^3\Phi\vert)\vert\nabla_p^2f\vert\notag\\
		&\qquad\qquad+\vert\nabla_x\nabla_pf\vert+\vert\nabla_x\Phi\vert\vert\nabla_x^2\nabla_pf\vert\Big){\rm d}x{\rm d}p\nonumber\\
		&\leq C(\left\|\nabla_x\Phi\right\|_{H^3_x}+e^{\bar\phi})\int_{\mathbb R^6}e^{5\bar\phi}(e^{2\phi}+\vert p\vert^2)^{\gamma}\vert\nabla_x^2\nabla_pf\vert^2{\rm d}x{\rm d}p\nonumber\\
		&\quad+C\left\|\nabla_x^3\Phi\right\|_{H_x^1}\int_{\mathbb R^6}e^{7\bar\phi}(e^{2\phi}+\vert p\vert^2)^{\gamma}\vert\nabla_x\nabla_p^2f\vert^2{\rm d}x{\rm d}p\nonumber\\
		&\quad+\eta\int_{\mathbb R^6}e^{3\bar\phi}(e^{2\phi}+\vert p\vert^2)^{\gamma}\vert\nabla_x^3f\vert^2{\rm d}x{\rm d}p+C_\eta\int_{\mathbb R^6}e^{2\bar\phi+4\phi}(e^{2\phi}+\vert p\vert^2)^{\gamma-\frac{1}{2}}\vert\nabla_p\nabla_x^2f\vert^2{\rm d}x{\rm d}p\nonumber\\
		&\quad+Ce^{\bar\phi}\sum_{m+n\leq 2}\mathcal E_{m,n}^\gamma(t).
	\end{align}
Finally, if $\vert\alpha\vert=3$, by applying Sobolev's inequality and H\"{o}lder's inequality, we
have
\begin{align}\label{i16-c33}
		|\CI_{16}\chi_{|\al|=3}|&\leq C\int_{\mathbb R^6}e^{3\bar\phi}(e^{2\phi}+\vert p\vert^2)^{\gamma}\vert\partial_x^\alpha f\vert e^{2\phi}(e^{2\phi}+\vert p\vert^2)^{-1}((1+\vert\nabla_x^3\Phi\vert)\nabla_xf\vert\notag\\
		&\qquad\qquad+\vert\nabla_x^2f\vert+\vert\nabla_x\Phi\vert\vert\nabla_x^3f\vert){\rm d}x{\rm d}p\nonumber\\
		&\quad+C\int_{\mathbb R^6}e^{3\bar\phi}(e^{2\phi}+\vert p\vert^2)^{\gamma}\vert\partial_x^\alpha f\vert e^\phi((1+|\nabla_x^3\Phi|+\vert\nabla_x^4\Phi\vert)\vert\nabla_pf\vert+(1+\vert\nabla_x^3\Phi\vert)\vert\nabla_x\nabla_pf\vert\nonumber\\
		&\qquad\qquad+\left\|\nabla_x^2\Phi\right\|_{L^\infty_x}\vert\nabla_x^2\nabla_pf\vert){\rm d}x{\rm d}p\nonumber\\
		&\leq C(e^{\bar\phi}+\left\|\nabla_x^2\Phi\right\|_{L^\infty_x})\int_{\mathbb R^6}e^{3\bar\phi}(e^{2\phi}+\vert p\vert^2)^{\gamma}\vert\partial_x^\alpha f\vert^2{\rm d}x{\rm d}p\nonumber\\
		&\quad+C\left\|\nabla_x\Phi\right\|_{H_x^3}\int_{\mathbb R^6}e^{5\bar\phi}(e^{2\phi}+\vert p\vert^2)^{\gamma}\vert\nabla_x^2\nabla_pf\vert^2{\rm d}x{\rm d}p+Ce^{\bar\phi}\sum_{m+n\leq2}\mathcal E_{m,n}^\gamma(t).
\end{align}
Now, combining \eqref{I15-2}, \eqref{I17-19-2}, \eqref{i16-c30} \eqref{i16-c31}, \eqref{i16-c32}, and \eqref{i16-c33} with  $\vert\alpha\vert+\vert\beta\vert=3$, and using \eqref{el-sg} in \eqref{2.54} with $|\al|=3$, we obtain
\begin{align}\label{m+n=3}
	\frac{d}{dt}&\sum_{m+n=3}\mathcal E_{m,n}^\gamma(t)+\sum_{m+n=3}\mathcal D_{m,n}^\gamma(t)\notag\\
	&\quad-C\bar\phi'(\infty)\sum_{|\al|+|\beta|=3}\int_{\mathbb R^6} e^{(|\alpha|+3|\beta|)\bar\phi}\Big((e^{2\phi}+|p|^2)^{\delta_1}+(e^{2\phi}+|p|^2)^{\delta_2}\Big)|\partial_x^\alpha\partial_p^\beta f|^2{\rm d}x{\rm d}p\notag\\	
		&\leq Ce^{\frac{\bar\phi}{2}}\sum_{\ m+n=3}(\mathcal E_{m,n}^{\delta_1}(t)+\mathcal E_{m,n}^{\delta_2}(t))+C\sum_{m+n\leq2}\mathcal D_{m,n}^\gamma(t).
\end{align}
\noindent\underline{{\it Case 3. $\vert\alpha\vert+\vert\beta\vert=4$ and $|\beta|<4$}.}
Similar to the estimation on  $I_{16}^{(i)}$ given in Case 3 in Subsection \ref{sec-ft-eng},
we have
\begin{align*}
	|\CI_{16}^{(1)}|&\leq C\int_{\mathbb R^6}e^{10\bar\phi}(e^{2\phi}+|p|^2)^\gamma|\partial_x^\alpha\partial_p^\beta f|\Big((e^{2\phi}+|p|^2)^{-\frac{3}{2}}(|\nabla_xf|+|\nabla_x^2f|)\\
	&\qquad\qquad+(e^{2\phi}+|p|^2)^{-1}(|\nabla_x\nabla_pf|+|\nabla_x^2\nabla_pf|)+(e^{2\phi}+|p|^2)^{-\frac{1}{2}}(|\nabla_x\nabla_p^2f|+|\nabla_x^2\nabla_p^2f|)\\
	&\qquad\qquad+|\nabla_x\Phi||\nabla_x\nabla_p^3f|\Big){\rm d}x{\rm d}p\\
	&\quad+C\int_{\mathbb R^6}e^{10\bar\phi}(e^{2\phi}+|p|^2)^\gamma|\partial_x^\alpha\partial_p^\beta f|e^{\phi}\Big((e^{2\phi}+|p|^2)^{-\frac{3}{2}}(|\nabla_pf|+|\nabla_x\nabla_pf|)\\
	&\qquad\qquad+(e^{2\phi}+|p|^2)^{-1}(|\nabla_p^2f|+|\nabla_x\nabla_p^2f|)+(e^{2\phi}+|p|^2)^{-\frac{1}{2}}(|\nabla_p^3f|+|\nabla_x\Phi||\nabla_x\nabla_p^3f|)\\
	&\qquad\qquad+e^{\phi}(e^{2\phi}+|p|^2)^{-\frac{1}{2}}|\nabla_p\nabla_p^3f|\Big){\rm d}x{\rm d}p\\
	&\leq C\int_{\mathbb R^6}e^{\frac{7\bar\phi}{2}+2\phi}(e^{2\phi}+|p|^2)^{\frac{\gamma}{2}-\frac{1}{4}}|\partial_x^\alpha\partial_p^\beta f|(e^{2\phi}+|p|^2)^{\frac{\gamma}{2}}\Big((e^{\frac{\bar\phi}{2}}|\nabla_xf|+e^{\bar\phi}|\nabla_x^2f|)e^{\bar\phi}\\
	&\qquad\qquad+(e^{2\bar\phi}|\nabla_x\nabla_pf|+e^{\frac{5\bar\phi}{2}}|\nabla_x^2\nabla_pf|)e^{\frac{\bar\phi}{2}}+e^{\frac{7\bar\phi}{2}}|\nabla_x\nabla_p^2f|e^{\frac{\bar\phi}{2}}\\
	&\qquad\qquad+e^{\frac{5\bar\phi}{2}+2\phi}(e^{2\phi}+|p|^2)^{-\frac{1}{4}}|\nabla_x^2\nabla_p^2f|\Big){\rm d}x{\rm d}p\\
	&\quad+ C\left\|\nabla_x\Phi\right\|_{L^\infty_x}\int_{\mathbb R^6}e^{10\bar\phi}(e^{2\phi}+|p|^2)^{\gamma}|\nabla_x\nabla_p^3f|^2{\rm d}x{\rm d}p\\
	&\quad+C\int_{\mathbb R^6}e^{\frac{7\bar\phi}{2}+2\phi}(e^{2\phi}+|p|^2)^{\frac{\gamma}{2}-\frac{1}{4}}|\partial_x^\alpha\partial_p^\beta f|(e^{2\phi}+|p|^2)^{\frac{\gamma}{2}}\Big((e^{\frac{3\bar\phi}{2}}|\nabla_pf|+e^{2\bar\phi}|\nabla_x\nabla_pf|)e^{\bar\phi}\\
	&\qquad\qquad+(e^{3\bar\phi}|\nabla_p^2f|+e^{\frac{7\bar\phi}{2}}|\nabla_x\nabla_p^2f|)e^{\frac{\bar\phi}{2}}+e^{\frac{9\bar\phi}{2}}|\nabla_p^3f|e^{\frac{\bar\phi}{2}}+e^{\frac{9}{2}\bar\phi+2\phi}(e^{2\phi}+|p|^2)^{-\frac{1}{4}}|\nabla_p\nabla_p^3f|\Big){\rm d}x{\rm d}p\\
	&\leq C\left\|\nabla_x\Phi\right\|_{L^\infty_x}\int_{\mathbb R^6}e^{10\bar\phi}(e^{2\phi}+|p|^2)^{\gamma}|\nabla_x\nabla_p^3f|^2{\rm d}x{\rm d}p+C \sum_{m+n\leq3}(e^{\bar\phi}\mathcal E_{m,n}^\gamma(t)+\mathcal D_{m,n}^\gamma(t)),
\end{align*}
and
\begin{align*}
	|\CI_{16}^{(2)}|&\leq C\int_{\mathbb R^6}e^{8\bar\phi}(e^{2\phi}+|p|^2)^\gamma|\partial_x^\alpha\partial_p^\beta f|\Big((e^{2\phi}+|p|^2)^{-1}(|\nabla_xf|+|\nabla_x^2f|+|\nabla_x^3f|)\\
	&\qquad\qquad+(e^{2\phi}+|p|^2)^{-\frac{1}{2}}(|\nabla_x\nabla_pf|+|\nabla_x^2\nabla_pf|+|\nabla_x^3\nabla_pf|)
	+\eps_0|\nabla_x\nabla_p^2f|+\eps_0|\nabla_x^2\nabla_p^2f|\Big){\rm d}x{\rm d}p\\
	&\quad+C\int_{\mathbb R^6}e^{8\bar\phi}(e^{2\phi}+|p|^2)^\gamma|\partial_x^\alpha\partial_p^\beta f|\Big((e^{2\phi}+|p|^2)^{-\frac{1}{2}}(\eps_0+|\nabla_x^3\Phi|)|\nabla_pf|\\
	&\qquad\qquad+(\eps_0+|\nabla_x^3\Phi|)|\nabla_p^2f|+e^{\phi}\big((\eps_0+|\nabla_x^3\Phi|)|\nabla_p^3f|+\eps_0|\nabla_x\nabla_p^3f|\big)\Big){\rm d}x{\rm d}p\\
	&\leq C\int_{\mathbb R^6}e^{\frac{5\bar\phi}{2}+2\phi}(e^{2\phi}+|p|^2)^{\frac{\ga}{2}-\frac{1}{4}}|\partial_x^\alpha\partial_p^\beta f|(e^{2\phi}+|p|^2)^{\frac{\ga}{2}}\Big((e^{\frac{\bar\phi}{2}}|\nabla_xf|+e^{\bar\phi}|\nabla_x^2f|+e^{\frac{3\bar\phi}{2}}|\nabla_x^3f|)e^{\frac{\bar\phi}{2}}\\
	&\qquad\qquad+(e^{2\bar\phi}|\nabla_x\nabla_pf|+e^{\frac{5\bar\phi}{2}}|\nabla_x^2\nabla_pf|)e^{\frac{\bar\phi}{2}}+e^{\frac{3\bar\phi}{2}}(e^{2\phi}+|p|^2)^{-\frac{1}{4}}|\nabla_x^3\nabla_pf|\Big){\rm }x{\rm d}p\\
	&\quad+C\eps_0\int_{\mathbb R^6}e^{4\bar\phi}(e^{2\phi}+|p|^2)^{\frac{\ga}{2}}|\partial_x^\alpha\partial_p^\beta f|(e^{2\phi}+|p|^2)^{\frac{\ga}{2}}\Big(e^{\frac{7\bar\phi}{2}}|\nabla_x\nabla_p^2f|e^{\frac{\bar\phi}{2}}+e^{4\bar\phi}|\nabla_x^2\nabla_p^2f|\\
	&\qquad\qquad+e^{\frac{3\bar\phi}{2}}|\nabla_pf|e^{\bar\phi}+e^{3\bar\phi}|\nabla_p^2f|e^{\bar\phi}+e^{\frac{9\bar\phi}{2}}|\nabla_p^3f|e^{\frac{\bar\phi}{2}}+e^{5\bar\phi}|\nabla_x\nabla_p^3f|\Big){\rm d}x{\rm d}p\\
	&\quad+C\int_{\mathbb R^6}e^{4\bar\phi}(e^{2\phi}+|p|^2)^{\frac{\ga}{2}}|\partial_x^\alpha\partial_p^\beta f||\nabla_x^3\Phi|(e^{2\phi}+|p|^2)^{\frac{\ga}{2}}\\
	&\qquad\qquad\times\Big(e^{2\bar\phi}|\nabla_pf|e^{\bar\phi}+e^{\frac{7\bar\phi}{2}}|\nabla_p^2f|e^{\bar\phi}+e^{5\bar\phi}|\nabla_p^3f|\Big){\rm d}x{\rm d}p\\
	&\leq C\eps_0\int_{\mathbb R^6}(e^{2\phi}+|p|^2)^\ga(e^{8\bar\phi}|\nabla_x^2\nabla_p^2f|^2+e^{10\bar\phi}|\nabla_x\nabla_p^3f|^2){\rm d}x{\rm d}p\notag\\&\quad+C\sum_{m+n\leq3}\Big(e^{\bar\phi}\mathcal E_{m,n}^\gamma(t)+\mathcal D_{m,n}^\ga(t)\Big),
\end{align*}
and
\begin{align*}
	|\CI_{16}^{(3)}|&\leq C\int_{\mathbb R^6}e^{6\bar\phi}(e^{2\phi}+|p|^2)^\gamma|\partial_x^\alpha\partial_p^\beta f|\Big((e^{2\phi}+|p|^2)^{-\frac{1}{2}}\big((\eps_0+|\nabla_x^3\Phi|)|\nabla_xf|+|\nabla_x^2f|+|\nabla_x^3f|+|\nabla_x^4f|\big)\\
	&\qquad\qquad+(\eps+|\nabla_x^3\Phi|)|\nabla_x\nabla_pf|+\eps_0|\nabla_x^2\nabla_pf|+\eps_0|\nabla_x^3\nabla_pf|\Big){\rm d}x{\rm d}p\\
	&\quad+C\int_{\mathbb R^6}e^{6\bar\phi}(e^{2\phi}+|p|^2)^\gamma|\partial_x^\alpha\partial_p^\beta f|\Big((\eps_0+|\nabla_x^3\Phi|+|\nabla_x^4\Phi|)|\nabla_pf|+(\eps_0+|\nabla_x^3\Phi|)|\nabla_x\nabla_pf|\\
	&\qquad\qquad+\eps_0|\nabla_x^2\nabla_pf|+\eps_0|\nabla_x^3\nabla_pf|+e^{\bar\phi}\big((\eps_0+|\nabla_x^3\Phi|+|\nabla_x^4\Phi|)|\nabla_p^2f|
\\&\qquad\qquad+(\eps_0+|\nabla_x^3\Phi|)|\nabla_x\nabla_p^2f|+\eps_0|\nabla_x^2\nabla_p^2f|\big)\Big){\rm d}x{\rm d}p\\
	&\leq C\int_{\mathbb R^6}e^{\frac{3\bar\phi}{2}+2\phi}(e^{2\phi}+|p|^2)^{\frac{\gamma}{2}-\frac{1}{4}}|\partial_x^\alpha\partial_p^\beta f|(e^{2\phi}+|p|^2)^{\frac{\gamma}{2}}\\
	&\qquad\qquad\times\Big((e^{\frac{\bar\phi}{2}}|\nabla_xf|+e^{\bar\phi}|\nabla_x^2f|
+e^{\frac{3\bar\phi}{2}}|\nabla_x^3f|)e^{\frac{\bar\phi}{2}}e^{2\bar\phi}|\nabla_x^4f|\Big){\rm d}x{\rm d}p\\
		&\quad+C\eps_0\int_{\mathbb R^6}e^{3\bar\phi}(e^{2\phi}+|p|^2)^{\frac{\gamma}{2}}|\partial_x^\alpha\partial_p^\beta f|(e^{2\phi}+|p|^2)^{\frac{\gamma}{2}}\Big(e^{\frac{\bar\phi}{2}}|\nabla_xf|e^{\frac{3\bar\phi}{2}}+e^{2\bar\phi}|\nabla_x\nabla_pf|e^{\bar\phi}\\
		&\qquad\qquad+e^{\frac{5\bar\phi}{2}}|\nabla_x^2\nabla_pf|e^{\frac{\bar\phi}{2}}+e^{3\bar\phi}|\nabla_x^3\nabla_pf|+(e^{\frac{3\bar\phi}{2}}|\nabla_pf|+e^{2\bar\phi}|\nabla_x\nabla_pf|)e^{\frac{\bar\phi}{2}}\\
	&\qquad\qquad+(e^{3\bar\phi}|\nabla_p^2f|+e^{\frac{7\bar\phi}{2}}|\nabla_x\nabla_p^2f|)e^{\frac{\bar\phi}{2}}+e^{4\bar\phi}|\nabla_x^2\nabla_p^2f|\Big){\rm d}x{\rm d}p\\
	&\quad+C\int_{\mathbb R^6}e^{3\bar\phi}(e^{2\phi}+|p|^2)^{\frac{\gamma}{2}}|\partial_x^\alpha\partial_p^\beta f|(e^{2\phi}+|p|^2)^{\frac{\gamma}{2}}|\nabla_x^3\Phi|\Big(e^{\bar\phi}|\nabla_xf|e^{\bar\phi}+e^{\frac{5\bar\phi}{2}}|\nabla_x\nabla_pf|e^{\frac{\bar\phi}{2}}\\
	&\qquad\qquad+e^{2\bar\phi}|\nabla_pf|e^{\bar\phi}+e^{\frac{7\bar\phi}{2}}|\nabla_p^2f|e^{\frac{\bar\phi}{2}}+e^{4\bar\phi}|\nabla_x\nabla_p^2f|\Big){\rm d}x{\rm d}p\\
&\quad+C\int_{\mathbb R^6}e^{3\bar\phi}(e^{2\phi}+|p|^2)^{\frac{\gamma}{2}}|\partial_x^\alpha\partial_p^\beta f|(e^{2\phi}+|p|^2)^{\frac{\gamma}{2}}|\nabla_x^4\Phi|\Big(e^{\frac{5\bar\phi}{2}}|\nabla_pf|e^{\frac{\bar\phi}{2}}+e^{4\bar\phi}|\nabla_p^2f|\Big){\rm d}x{\rm d}p\\
	&\leq C\eps_0\int_{\mathbb R^6}(e^{2\phi}+|p|^2)^{\gamma}(e^{6\bar\phi}|\nabla_x^3\nabla_pf|^2+e^{7\bar\phi}|\nabla_x\nabla_p^2f|^2+e^{8\bar\phi}|\nabla_x^2\nabla_p^2f|^2){\rm d}x{\rm d}p\\
	&\quad+\eta\int_{\mathbb R^6}e^{4\bar\phi}(e^{2\phi}+|p|^2)^{\gamma}|\nabla_x^4f|^2{\rm d}x{\rm d}p+C_\eta \sum_{m+n\leq3}\Big(e^{\bar\phi}\mathcal E_{m,n}^\gamma(t)+\mathcal D_{m,n}^\gamma(t)\Big),
\end{align*}
and
\begin{align*}
	|\CI_{16}^{(4)}|&\leq	C\int_{\mathbb R^6}e^{4\bar\phi}(e^{2\phi}+|p|^2)^{\gamma}|\partial_x^\alpha f|\Big((\eps_0+|\nabla_x^3\Phi|+|\nabla_x^4\Phi|)|\nabla_xf|+(\eps_0+|\nabla_x^3\Phi|)|\nabla_x^2f|\\
	&\qquad\qquad+\eps_0|\nabla_x^3f|+\eps_0|\nabla_x^4f|\Big){\rm d}x{\rm d}p\\
	&\quad+C\int_{\mathbb R^6}e^{4\bar\phi}(e^{2\phi}+|p|^2)^{\gamma}|\partial_x^\alpha f|e^\phi\Big((\eps_0+|\nabla_x^3\Phi|+|\nabla_x^4\Phi|+|\nabla_x^5\Phi|)|\nabla_pf|\\
	&\qquad\qquad+(\eps_0+|\nabla_x^3\Phi|+|\nabla_x^4\Phi|)|\nabla_x\nabla_pf|+(\eps_0+|\nabla_x^3\Phi|)|\nabla_x^2\nabla_pf|+\eps_0|\nabla_x^3\nabla_pf|\Big){\rm d}x{\rm d}p\\
	&\leq C\eps_0\int_{\mathbb R^6}e^{2\bar\phi}(e^{2\phi}+|p|^2)^{\frac{\gamma}{2}}|\partial_x^\alpha f|\Big((e^{\frac{\bar\phi}{2}}|\nabla_xf|+e^{\bar\phi}|\nabla_x^2f|+e^{\frac{3\bar\phi}{2}}|\nabla_x^3f|)e^{\frac{\bar\phi}{2}}+e^{2\bar\phi}|\nabla_x^4f|\\
	&\qquad\qquad+(e^{\frac{3\bar\phi}{2}}|\nabla_pf|+e^{2\bar\phi}|\nabla_x\nabla_pf|)e^{\frac{\bar\phi}{2}}+e^{\frac{5\bar\phi}{2}}|\nabla_x^2\nabla_pf|e^{\frac{\bar\phi}{2}}+e^{3\bar\phi}|\nabla_x^3\nabla_pf|\Big){\rm d}x{\rm d}p\\
		&\quad+ C\int_{\mathbb R^6}e^{2\bar\phi}(e^{2\phi}+|p|^2)^{\frac{\gamma}{2}}|\partial_x^\alpha f||\nabla_x^3\Phi|\Big(e^{\bar\phi}|\nabla_xf|e^{\bar\phi}+e^{\frac{3\bar\phi}{2}}|\nabla_x^2f|e^{\frac{\bar\phi}{2}}\\
	&\qquad\qquad+(e^{2\bar\phi}|\nabla_pf|+e^{\frac{5\bar\phi}{2}}|\nabla_x\nabla_pf|)e^{\frac{\bar\phi}{2}}+e^{3\bar\phi}|\nabla_x^2\nabla_pf|\Big){\rm d}x{\rm d}p\\
&\quad+ C\int_{\mathbb R^6}e^{2\bar\phi}(e^{2\phi}+|p|^2)^{\frac{\gamma}{2}}|\partial_x^\alpha f||\nabla_x^4\Phi|\Big(e^{\frac{3\bar\phi}{2}}|\nabla_xf|e^{\frac{\bar\phi}{2}}+e^{\frac{5\bar\phi}{2}}|\nabla_pf|e^{\frac{\bar\phi}{2}}\Big){\rm d}x{\rm d}p\\
&\quad+ C\int_{\mathbb R^6}e^{2\bar\phi}(e^{2\phi}+|p|^2)^{\frac{\gamma}{2}}|\partial_x^\alpha f||\nabla_x^5\Phi|e^{3\bar\phi}|\nabla_pf|{\rm d}x{\rm d}p\\
	&\leq C\eps_0\int_{\mathbb R^6}(e^{2\phi}+|p|^2)^{\gamma}(e^{4\bar\phi}|\nabla_x^4 f|^2+e^{5\bar\phi}|\nabla_x^2\nabla_pf|^2+e^{6\bar\phi}|\nabla_x^3\nabla_pf|^2){\rm d}x{\rm d}p+Ce^{\bar\phi}\sum_{m+n\leq3}\mathcal E_{m,n}^\gamma(t).
\end{align*}
Finally, for $|\alpha|=4$, we bound \eqref{n=4,fin} by using \eqref{el-sg} to have
\begin{align}
		\frac{1}{2}&\frac{d}{dt}\int_{\mathbb R^3}e^{\bar\phi}(\vert\partial_t\partial_{x}^\alpha\Phi\vert^2+\vert\nabla_x\partial_{x}^\alpha\Phi\vert^2){\rm d}x\notag\\
		&\leq\frac{1}{2}\int_{\mathbb R^3}\bar\phi'(t)e^{\bar\phi}\vert\partial_t\partial_{x}^\alpha\Phi\vert^2{\rm d}x\nonumber\\
		&\quad-\int_{\mathbb R^3}\partial_{x}^\alpha\Big(e^{2\phi}\int_{\mathbb R^3}\frac{f}{\sqrt{e^{2\phi}+\vert p\vert^2}}{\rm d}p\Big)\partial_t\partial_{x}^\alpha\Phi{\rm d}x-\int_{\mathbb R^3}\int_{\mathbb R^3}\partial_{x}^\alpha\Big(e^{2\phi}\frac{\bar F}{\sqrt{e^{2\phi}+\vert p\vert^2}}\Big){\rm d}p\partial_t\partial_{x}^\alpha\Phi{\rm d}x\nonumber\\
		&\leq Ce^{\frac{\bar\phi}{2}}\sum_{\substack{ m+n\leq4\\ n<4}}(\mathcal E_{m,n}^{\delta_1}(t)+\mathcal E_{m,n}^{\delta_2}(t))+Ce^{\bar\phi}\left\|\nabla_x\Phi\right\|_{H^3_x}^2.\notag
\end{align}
Therefore,
\begin{align}\label{m+n=4}
	\frac{d}{dt}&\sum_{\substack{m+n=4\\ n<4}}\mathcal E_{m,n}^\gamma(t)+\sum_{{\substack{m+n\leq4\\ n<4}}}\mathcal D_{m,n}^\gamma(t)\notag\\
	&\quad-C\bar\phi'(\infty)\sum_{\substack{|\al|+|\beta|=4\\|\beta|<4}}\int_{\mathbb R^6} e^{(|\alpha|+3|\beta|)\bar\phi}\Big((e^{2\phi}+|p|^2)^{\delta_1}+(e^{2\phi}+|p|^2)^{\delta_2}\Big)|\partial_x^\alpha\partial_p^\beta f|^2{\rm d}x{\rm d}p\notag\\	
	&\leq Ce^{\frac{\bar\phi}{2}}\sum_{{\substack{m+n\leq4\\ n<4}}}(\mathcal E_{m,n}^{\delta_1}(t)+\mathcal E_{m,n}^{\delta_2}(t))+C\sum_{m+n\leq3}\mathcal D_{m,n}^\gamma(t)+Ce^{\bar\phi}\left\|\Phi\right\|_{L^2_x}^2.
\end{align}
By applying Gr\"{o}nwall's inequality to \eqref{m+n=3}, \eqref{m+n=4} and using \eqref{3.7}, \eqref{3.15} and \eqref{3.26}, one has
\begin{equation}
	\begin{split}
\sum_{\substack{m+n\leq4\\n<4}}&	(\mathcal E_{m,n}^{\delta_1}(t)+\mathcal E_{m,n}^{\delta_2}(t))+\int_0^t(\mathcal D_{0,0}^{\delta_1}(s)+\mathcal D_{0,0}^{\delta_2}(s)){\rm d}s+\sum_{\substack{ 1\leq m+n\leq4\\n<4}}\int_0^t	(\bar{\mathcal D}_{m,n}^{\delta_1}(s)+\bar{\mathcal D}_{m,n}^{\delta_2}(s)){\rm d}s\\
	&	\leq C(T_c)\Big(\sum_{\substack{m+n\leq4\\n<4}}	\mathcal E_{m,n}^{\delta_2}(0)+\left\|\Phi_0\right\|_{L^2_x}^2\Big).\notag
	\end{split}
\end{equation}
This completes the proof of Proposition \ref{lg-eng-pro}.
\end{proof}

\subsection{Long time behavior}
In this subsection, we will focus on proving the long time behavior as shown in \eqref{long-t} based on the {\it a priori} energy estimate \eqref{eng-tt}
proved in Subsections \ref{sec-ft-eng} and \ref{sec-ift-eng}.
From \eqref{eng-tt}, it follows
\begin{equation*}
\sum_{\substack{\vert\alpha\vert+\vert\beta\vert\leq4\\\vert\beta\vert<4}}\int_0^t\int_{\mathbb R^6}e^{(\vert\alpha\vert+3\vert\beta\vert)\bar\phi+4\bar\phi}\vert\nabla_p\partial_x^\alpha\partial_p^\beta f\vert^2{\rm d}x{\rm d}p{\rm d}s\leq C(\mathcal E^{\delta_2}(0)+\left\|\Phi_0\right\|_{L^2_x}^2),
\end{equation*}
which implies
\begin{equation}\label{4.1}
		\sum_{\substack{|\alpha|\leq2,\vert\beta\vert\leq1}}\int_0^t\int_{\mathbb R^6}e^{9\bar\phi}\vert\nabla_p\partial_x^\alpha\partial_p^\beta f\vert^2{\rm d}x{\rm d}p{\rm d}s\leq C(\mathcal E^{\delta_2}(0)+\left\|\Phi_0\right\|_{L^2_x}^2).
\end{equation}
Since
\begin{align}\label{4.2}
	\sum_{\substack{|\alpha|\leq2,\vert\beta\vert\leq1}}&\int_0^t\int_{\mathbb R^6}e^{10\bar\phi}\vert\nabla_p\partial_x^\alpha\partial_p^\beta \partial_tf\vert^2{\rm d}x{\rm d}p{\rm d}s\nonumber\\
	&\leq C\Big(\mathcal E^{\delta_2}(0)+\left\|(1+|p|^2)^{\frac{\delta_2}{2}}\nabla_p^4f_0\right\|_{H^1_{x,p}}^2+\left\|\Phi_0\right\|_{L^2_x}^2\Big),
\end{align}
then \eqref{4.1} and \eqref{4.2} give
\begin{align*}
\int_{\mathbb R^6}\Big\vert\frac{d}{dt}\left\| e^{5\bar\phi}(\nabla_p f,\nabla_p^2f)\right\|_{L^2_pH^2_x}^2\Big\vert{\rm d}s\leq& C\int_0^t\left\| e^{5\bar\phi}(\nabla_p f,\nabla_p^2f,\nabla_p\partial_t f,\nabla_p^2\partial_t f)\right\|_{L^2_pH^2_x}^2{\rm d}s\\
\leq& C\Big(\mathcal E^{\delta_2}(0)+\left\|(1+|p|^2)^{\frac{\delta_2}{2}}\nabla_p^4f_0\right\|_{H^1_{x,p}}^2+\left\|\Phi_0\right\|_{L^2_x}^2\Big).
\end{align*}
Thus, $\left\|e^{5\bar\phi}\nabla_pf\right\|_{H_p^1H_x^2}\in W^{1,1}(\R^+)$  implies that
\begin{align}\label{H2-lim}
\left\|e^{5\bar\phi}f\right\|_{L^\infty_{x,p}}\rightarrow0,\ \textrm{as}\ t\rightarrow\infty.
\end{align}
Now we turn to prove \eqref{4.2}.
For this, set $\tilde f=\partial_t f$, apply $\partial_t$ to \eqref{vnfp-s2}$_1$, to obtain
\begin{equation*}
		\partial_t \tilde f+\nabla_p\sqrt{e^{2\phi}+\vert p\vert^2}\cdot\nabla_x \tilde f-\nabla_x\sqrt{e^{2\phi}+\vert p\vert^2}\cdot\nabla_p\tilde f-e^{2\phi}\nabla_p\cdot(\Lambda_{\phi,p}\nabla_p\tilde f)=\mathcal R
\end{equation*}
with
\begin{equation*}
	\begin{split}
\mathcal R&=\nabla_x\partial_t\sqrt{e^{2\phi}+\vert p\vert^2}\cdot\nabla_p f-\nabla_p\partial_t\sqrt{e^{2\phi}+\vert p\vert^2}\cdot\nabla_x f+\partial_t\big(e^{2\phi}\nabla_p\cdot(\Lambda_\phi\nabla_pf)\big)-e^{2\phi}\nabla_p\cdot(\Lambda_\phi\nabla_p\tilde f)\\
&\quad+\partial_t\Big(\nabla_x\sqrt{e^{2\phi}+\vert p\vert^2}\cdot\nabla_p\bar F+(e^{2\phi}\nabla_p\cdot(\Lambda_{\phi,p}\nabla_p\bar F)-e^{2\bar\phi}\nabla_p\cdot(\Lambda_{\bar\phi,p}\nabla_p\bar F))\Big).
	\end{split}
\end{equation*}
We first consider the case when $t\in[0,T_c]$. Similar to the discussion in Subsection \ref{sec-ft-eng}, we can obtain
\begin{align}\label{4.4}
\frac{d}{dt}&\sum_{\vert\alpha\vert+ \vert\beta\vert\leq3}\int_{\mathbb R^6}e^{(1+\vert\alpha\vert+3\vert\beta\vert)\bar\phi}(e^{2\phi}+\vert p\vert^2)^{\frac{1}{2}}\vert\partial_x^\alpha\partial_p^\beta \tilde f\vert^2{\rm d}x{\rm d}p\nonumber\\
&+\sum_{\vert\alpha\vert+ \vert\beta\vert\leq3}\int_{\mathbb R^6}e^{(\vert\alpha\vert+3\vert\beta\vert+3)\bar\phi+2\Phi}(e^{2\phi}\vert\nabla_p\partial_x^\alpha\partial_p^\beta \tilde f\vert^2+\vert p\cdot\nabla_p\partial_x^\alpha\partial_p^\beta \tilde f\vert^2){\rm d}x{\rm d}p\nonumber\\
\leq& C\sum_{\vert\alpha\vert+ \vert\beta\vert\leq3}\int_{\mathbb R^6}e^{(1+\vert\alpha\vert+3\vert\beta\vert)\bar\phi}(e^{2\phi}+\vert p\vert^2)^{\frac{1}{2}}\vert\partial_x^\alpha\partial_p^\beta \tilde f\vert^2{\rm d}x{\rm d}p\nonumber\\
&+\underbrace{\sum_{\vert\alpha\vert+\vert\beta\vert\leq3}\int_{\mathbb R^6}e^{(1+\vert\alpha\vert+3\vert\beta\vert)\bar\phi}(e^{2\phi}+\vert p\vert^2)^{\frac{1}{2}}\partial_x^\alpha\partial_p^\beta\mathcal R\partial_x^\alpha\partial_p^\beta \tilde f{\rm d}x{\rm d}p}_{\mathcal R_0}.
\end{align}
Next, by using Sobolev's inequality and H\"{o}lder's inequality, $\mathcal R_0$ can be bounded by
\begin{align}\label{r0}
C \Big(&\sum_{\vert\alpha\vert+\vert\beta\vert\leq3}\left\|e^{\frac{1}{2}(1+\vert\alpha\vert+3\vert\beta\vert)\bar\phi}(e^{2\phi}+\vert p\vert^2)^{\frac{1}{4}}\partial_x^\alpha\partial_p^\beta\tilde f\right\|_{L^2_{x,p}}\Big)\nonumber\\
	&\times\Big(\sum_{\vert\alpha\vert+\vert\beta\vert\leq3}\left\|e^{\frac{1}{2}(1+\vert\alpha\vert+3\vert\beta\vert)\bar\phi}(e^{2\phi}+\vert p\vert^2)^{\frac{1}{4}}\nabla_x\partial_x^\alpha\partial_p^\beta f\right\|_{L^2_{x,p}}\notag\\&\qquad\qquad\qquad\qquad+\sum_{\substack{|\alpha|+|\beta|\leq3\\|\beta<3}}\left\|e^{\frac{1}{2}(3+\vert\alpha\vert+3\vert\beta\vert)\bar\phi}(e^{2\phi}+\vert p\vert^2)^{\frac{1}{4}}\nabla_p\partial_x^\alpha\partial_p^\beta f\right\|_{L^2_{x,p}}\Big)\nonumber\\
	&+\eta\sum_{\vert\alpha\vert+ \vert\beta\vert\leq3}\int_{\mathbb R^6}e^{(\vert\alpha\vert+3\vert\beta\vert+3)\bar\phi+2\Phi}(e^{2\phi}\vert\nabla_p\partial_x^\alpha\partial_p^\beta \tilde f\vert^2+\vert p\cdot\nabla_p\partial_x^\alpha\partial_p^\beta \tilde f\vert^2){\rm d}x{\rm d}p\nonumber\\
	&+Ce^{\bar\phi}\sum_{\vert\alpha\vert+\vert\beta\vert\leq3}\left\|e^{\frac{1}{2}(1+\vert\alpha\vert+3\vert\beta\vert)\bar\phi}(e^{2\phi}+\vert p\vert^2)^{\frac{1}{4}}\partial_x^\alpha\partial_p^\beta\tilde f\right\|_{L^2_{x,p}}^2+Ce^{\bar\phi}(\left\|(\partial_t\Phi,\nabla_x\Phi)\right\|_{H_x^3}^2+\left\|\Phi\right\|_{L^2_x}^2)\nonumber\\
	&+C_\eta\sum_{\vert\alpha\vert+ \vert\beta\vert\leq3}\int_{\mathbb R^6}e^{(\vert\alpha\vert+3\vert\beta\vert+2)\bar\phi+2\Phi}(e^{2\phi}\vert\nabla_p\partial_x^\alpha\partial_p^\beta f\vert^2+\vert p\cdot\nabla_p\partial_x^\alpha\partial_p^\beta  f\vert^2){\rm d}x{\rm d}p,
\end{align}
where we have used the fact that for $|\beta|=3$
\begin{align*}
\int_{\mathbb R^6}&e^{10\bar\phi}(e^{2\phi}+|p|^2)^{\frac{1}{2}}\nabla_x\partial_t\sqrt{e^{2\phi}+|p|^2}\cdot\nabla_p\partial_p^\beta f\partial_p^\beta\tilde f{\rm d}x{\rm d}p\\
&\leq C\int_{\mathbb R^6}e^{\frac{9\bar\phi}{2}+2\phi}|\nabla_p\partial_p^\beta f|e^{\frac{7\bar\phi}{2}+2\phi}|\partial_p^\beta\tilde f|{\rm d}x{\rm d}x\\
&\leq \eta\int_{\mathbb R^3}e^{7\bar\phi+4\phi}|\nabla_p^3\tilde f|^2{\rm d}x{\rm d}p+C_\eta\int_{\mathbb R^6}e^{9\bar\phi+4\phi}|\nabla_p^4f|^2{\rm d}x{\rm d}p.
\end{align*}
Substituting \eqref{r0} into \eqref{4.4}, we then have
\begin{align*}
\frac{d}{dt}\sum_{\vert\alpha\vert+ \vert\beta\vert\leq3}&\int_{\mathbb R^6}e^{(1+\vert\alpha\vert+3\vert\beta\vert)\bar\phi}(e^{2\phi}+\vert p\vert^2)^{\frac{1}{2}}\vert\partial_x^\alpha\partial_p^\beta \tilde f\vert^2{\rm d}x{\rm d}p\nonumber\\
&+\sum_{\vert\alpha\vert+ \vert\beta\vert\leq3}\int_{\mathbb R^6}e^{(\vert\alpha\vert+3\vert\beta\vert+3)\bar\phi+2\Phi}(e^{2\phi}\vert\nabla_p\partial_x^\alpha\partial_p^\beta \tilde f\vert^2+\vert p\cdot\nabla_p\partial_x^\alpha\partial_p^\beta \tilde f\vert^2){\rm d}x{\rm d}p\nonumber\\
\leq& C\sum_{\vert\alpha\vert+ \vert\beta\vert\leq3}\int_{\mathbb R^6}e^{(1+\vert\alpha\vert+3\vert\beta\vert)\bar\phi}(e^{2\phi}+\vert p\vert^2)^{\frac{1}{2}}\vert\partial_x^\alpha\partial_p^\beta \tilde f\vert^2{\rm d}x{\rm d}p\\
&+C\sum_{\vert\alpha\vert+\vert\beta\vert\leq3}\left\|e^{\frac{1}{2}(1+\vert\alpha\vert+3\vert\beta\vert)\bar\phi}(e^{2\phi}+\vert p\vert^2)^{\frac{1}{4}}\nabla_x\partial_x^\alpha\partial_p^\beta f\right\|^2_{L^2_{x,p}}\\
&+C\sum_{\substack{\vert\alpha\vert+\vert\beta\vert\leq3\\|\beta|<3}}\left\|e^{\frac{1}{2}(3+\vert\alpha\vert+3\vert\beta\vert)\bar\phi}(e^{2\phi}+\vert p\vert^2)^{\frac{1}{4}}\nabla_p\partial_x^\alpha\partial_p^\beta f\right\|^2_{L^2_{x,p}}\\
	&+Ce^{\bar\phi}\sum_{\vert\alpha\vert+\vert\beta\vert\leq3}
\left\|e^{\frac{1}{2}(1+\vert\alpha\vert+3\vert\beta\vert)\bar\phi}(e^{2\phi}+\vert p\vert^2)^{\frac{1}{4}}\partial_x^\alpha\partial_p^\beta\tilde f\right\|_{L^2_{x,p}}^2\notag\\&
+Ce^{\bar\phi}\left(\left\|(\partial_t\Phi,\nabla_x\Phi)\right\|_{H_x^3}^2+\left\|\Phi\right\|_{L^2_x}^2\right)\nonumber\\
&+C_\eta\sum_{\vert\alpha\vert+ \vert\beta\vert\leq3}\int_{\mathbb R^6}e^{(\vert\alpha\vert+3\vert\beta\vert+2)\bar\phi+2\Phi}(e^{2\phi}\vert\nabla_p\partial_x^\alpha\partial_p^\beta f\vert^2+\vert p\cdot\nabla_p\partial_x^\alpha\partial_p^\beta  f\vert^2){\rm d}x{\rm d}p.
\end{align*}
Next, applying Gronwall's inequality to the above inequality and using \eqref{eng-tt} yields that
	\begin{align*}
		\sum_{\vert\alpha\vert+ \vert\beta\vert\leq3}&\int_{\mathbb R^6}e^{(1+\vert\alpha\vert+3\vert\beta\vert)\bar\phi}(e^{2\phi}+\vert p\vert^2)^{\frac{1}{2}}\vert\partial_x^\alpha\partial_p^\beta \tilde f\vert^2{\rm d}x{\rm d}p\nonumber\\
		&+\sum_{\vert\alpha\vert+ \vert\beta\vert\leq3}\int_0^t\int_{\mathbb R^6}e^{(\vert\alpha\vert+3\vert\beta\vert+3)\bar\phi+2\Phi}(e^{2\phi}\vert\nabla_p\partial_x^\alpha\partial_p^\beta \tilde f\vert^2+\vert p\cdot\nabla_p\partial_x^\alpha\partial_p^\beta \tilde f\vert^2){\rm d}x{\rm d}p{\rm d}s\nonumber\\
		\leq& C(T_c)\Big(\sum_{\vert\alpha\vert+ \vert\beta\vert\leq3}\int_{\mathbb R^6}(1+\vert p\vert^2)^{\frac{1}{2}}\vert\partial_x^\alpha\partial_p^\beta \tilde f(0)\vert^2{\rm d}x{\rm d}p+\mathcal E^{\delta_2}(0)+\left\|\Phi_0\right\|_{L^2_x}^2\Big),
	\end{align*}
	for $t\in[0,T_c]$. Here, $\tilde f(0)$ is given by
	\begin{equation*}
		\begin{split}
			\tilde f(0)&=\nabla_x\sqrt{e^{2\phi_0}+\vert p\vert^2}\cdot\nabla_p f_0-\nabla_p\sqrt{e^{2\phi_0}+\vert p\vert^2}\cdot\nabla_x f_0+\nabla_x\sqrt{e^{2\phi_0}+\vert p\vert^2}\cdot\nabla_p\bar F_0\\
			&\quad+e^{2\phi_0}\nabla_p\cdot(\Lambda_{\phi_0,p}\nabla_pf_0)+(e^{2\phi_0}\nabla_p\cdot(\Lambda_{\phi_0,p}\nabla_p\bar F_0)-e^{2\bar\phi_0}\nabla_p\cdot(\Lambda_{\bar\phi_0,p}\nabla_p\bar F_0)).
		\end{split}
	\end{equation*}
	It is direct to check that
	\begin{align*}
		\sum_{\vert\alpha\vert+ \vert\beta\vert\leq3}&\int_{\mathbb R^6}(1+\vert p\vert^2)^{\frac{1}{2}}\vert\partial_x^\alpha\partial_p^\beta \tilde f(0)\vert^2{\rm d}x{\rm d}p\notag\\
		\leq& C(T_c)\Big(\left\|\nabla_x\Phi_0\right\|_{H^3_x}^2\left\|(1+|p|^2)^{\frac{1}{4}}\nabla_p f_0\right\|_{H_{x,p}^3}^2+(1+\left\|\nabla_x\Phi_0\right\|_{H^2_x}^2)\left\|(1+|p|^2)^{\frac{1}{4}}\nabla_x f_0\right\|_{H_{x,p}^3}^2\notag\\
		&+\left\|\nabla_x\Phi_0\right\|_{H^3_x}^2\left\|(1+|p|^2)^{\frac{1}{4}}\nabla_p \bar F_0\right\|_{H_{p}^3}^2+(1+\left\|\nabla_x\Phi_0\right\|_{H^2_x}^2)\left\|(1+|p|^2)^{\frac{1}{4}}\nabla_p f_0\right\|_{H_{x,p}^3}^2\notag\\&+(1+\left\|\nabla_x\Phi_0\right\|_{H^2_x}^2)\left\|(1+|p|^2)^{\frac{3}{4}}\nabla_p^2 f_0\right\|_{H_{x,p}^3}^2
	+\left\|\Phi_0\right\|_{H^3_x}^2\left\|(1+|p|^2)^{\frac{1}{4}}\nabla_p \bar F_0\right\|_{H_{p}^3}^2\notag\\&+\left\|\Phi_0\right\|_{H^3_x}^2\left\|(1+|p|^2)^{\frac{3}{4}}\nabla_p^2 \bar F_0\right\|_{H_{p}^3}^2
		+\left\|\Phi_0\right\|_{H^3_x}^2\left\|(1+|p|^2)^{\frac{1}{4}}\nabla_p \bar F_0\right\|_{H_{p}^4}^2\Big)\notag\\
		\leq& C(T_c)\Big(\left\|(1+|p|^2)^{\frac{3}{4}} f_0\right\|_{H_{x,p}^5}^2+\left\|\Phi_0\right\|_{H^4_x}^2\Big)
		\leq C(T_c)\Big(\left\|(1+|p|^2)^{\frac{\delta_2}{2}} f_0\right\|_{H_{x,p}^5}^2+\left\|\Phi_0\right\|_{H^4_x}^2\Big)\notag\\
		\leq& C(T_c)\Big(\mathcal E^{\delta_2}(0)+\left\|(1+|p|^2)^{\frac{\delta_2}{2}}\nabla_p^4f_0\right\|_{H_{x,p}^1}^2+\left\|\Phi_0\right\|_{L^2_x}^2\Big),
	\end{align*}
where we have used the assumption that $\delta_2\geq\frac{3}{2}$.
	Then we conclude that
	\begin{align}\label{4.7}
		\sum_{\vert\alpha\vert+ \vert\beta\vert\leq3}&\int_{\mathbb R^6}e^{(1+\vert\alpha\vert+3\vert\beta\vert)\bar\phi}(e^{2\phi}+\vert p\vert^2)^{\frac{1}{2}}\vert\partial_x^\alpha\partial_p^\beta \tilde f\vert^2{\rm d}x{\rm d}p\nonumber\\
		&+\sum_{\vert\alpha\vert+ \vert\beta\vert\leq3}\int_0^t\int_{\mathbb R^6}e^{(\vert\alpha\vert+3\vert\beta\vert+3)\bar\phi+2\Phi}(e^{2\phi}\vert\nabla_p\partial_x^\alpha\partial_p^\beta \tilde f\vert^2+\vert p\cdot\nabla_p\partial_x^\alpha\partial_p^\beta \tilde f\vert^2){\rm d}x{\rm d}p{\rm d}s\nonumber\\
		\leq& C(T_c)\Big(\mathcal E^{\delta_2}(0)+\left\|(1+|p|^2)^{\frac{\delta_2}{2}}\nabla_p^4f_0\right\|_{H_{x,p}^1}^2+\left\|\Phi_0\right\|_{L^2_x}^2\Big),
	\end{align}
	for $t\in[0,T_c]$.

Next we prove that \eqref{4.2} holds for $t\in(T_c,\infty)$. By a similar argument as in Subsection \ref{sec-ift-eng}, one has
\begin{align}\label{4.8}
\frac{d}{dt}\sum_{\vert\alpha\vert+ \vert\beta\vert\leq3}&\int_{\mathbb R^6}e^{(1+\vert\alpha\vert+3\vert\beta\vert)\bar\phi}(e^{2\phi}+\vert p\vert^2)^{\frac{1}{2}}\vert\partial_x^\alpha\partial_p^\beta \tilde f\vert^2{\rm d}x{\rm d}p\nonumber\\
&-C\bar\phi'(\infty)\sum_{\vert\alpha\vert+ \vert\beta\vert\leq3}\int_{\mathbb R^6}e^{(1+\vert\alpha\vert+3\vert\beta\vert)\bar\phi}(e^{2\phi}+\vert p\vert^2)^{\frac{1}{2}}\vert\partial_x^\alpha\partial_p^\beta \tilde f\vert^2{\rm d}x{\rm d}p\nonumber\\
&+\sum_{\vert\alpha\vert+ \vert\beta\vert\leq3}\int_{\mathbb R^6}e^{(\vert\alpha\vert+3\vert\beta\vert+3)\bar\phi+2\Phi}(e^{2\phi}\vert\nabla_p\partial_x^\alpha\partial_p^\beta \tilde f\vert^2+\vert p\cdot\nabla_p\partial_x^\alpha\partial_p^\beta \tilde f\vert^2){\rm d}x{\rm d}p\nonumber\\
\leq& Ce^{\bar\phi}\sum_{\vert\alpha\vert+ \vert\beta\vert\leq3}\int_{\mathbb R^6}e^{(1+\vert\alpha\vert+3\vert\beta\vert)\bar\phi}(e^{2\phi}+\vert p\vert^2)^{\frac{1}{2}}\vert\partial_x^\alpha\partial_p^\beta \tilde f\vert^2{\rm d}x{\rm d}p+\mathcal R_0.
\end{align}
Unlike \eqref{r0}, we bound the first integral in $\CR_0$ by using Cauchy-Schwarz's inequality with a small parameter $\eta>0$ as follows
\begin{align}\label{4.9} \eta\sum_{\vert\alpha\vert+\vert\beta\vert\leq3}&\left\|e^{\frac{1}{2}(1+\vert\alpha\vert+3\vert\beta\vert)\bar\phi}(e^{2\phi}+\vert p\vert^2)^{\frac{1}{4}}\partial_x^\alpha\partial_p^\beta\tilde f\right\|_{L^2_{x,p}}^2\notag\\
		&+C_\eta\sum_{\vert\alpha\vert+\vert\beta\vert\leq3}
\left\|e^{\frac{1}{2}(1+\vert\alpha\vert+3\vert\beta\vert)\bar\phi}(e^{2\phi}+\vert p\vert^2)^{\frac{1}{4}}\nabla_x\partial_x^\alpha\partial_p^\beta f\right\|_{L^2_{x,p}}^2\notag\\
&+C_\eta\sum_{\substack{\vert\alpha\vert+\vert\beta\vert\leq3\\|\beta|<3}}
\left\|e^{\frac{1}{2}(3+\vert\alpha\vert+3\vert\beta\vert)\bar\phi}(e^{2\phi}+\vert p\vert^2)^{\frac{1}{4}}\nabla_p\partial_x^\alpha\partial_p^\beta f\right\|^2_{L^2_{x,p}}\nonumber\\
&+\eta\sum_{\vert\beta\vert=2}\int_{\mathbb R^6}e^{7\bar\phi+4\phi}\vert\nabla_p\partial_p^\beta \tilde f\vert^2{\rm d}x{\rm d}p+C_\eta\sum_{\vert\beta\vert=3}\int_{\mathbb R^6}e^{9\bar\phi+4\phi}\vert\nabla_p\partial_p^\beta  f\vert^2{\rm d}x{\rm d}p.
\end{align}
Consequently, from \eqref{r0} with \eqref{4.9} and \eqref{4.8}, we have
\begin{align*}
\frac{d}{dt}\sum_{\vert\alpha\vert+ \vert\beta\vert\leq3}&\int_{\mathbb R^6}e^{(1+\vert\alpha\vert+3\vert\beta\vert)\bar\phi}(e^{2\phi}+\vert p\vert^2)^{\frac{1}{2}}\vert\partial_x^\alpha\partial_p^\beta \tilde f\vert^2{\rm d}x{\rm d}p\nonumber\\
	&-C\bar\phi'(\infty)\sum_{\vert\alpha\vert+ \vert\beta\vert\leq3}\int_{\mathbb R^6}e^{(1+\vert\alpha\vert+3\vert\beta\vert)\bar\phi}(e^{2\phi}+\vert p\vert^2)^{\frac{1}{2}}\vert\partial_x^\alpha\partial_p^\beta \tilde f\vert^2{\rm d}x{\rm d}p\nonumber\\
	&+\sum_{\vert\alpha\vert+ \vert\beta\vert\leq3}\int_{\mathbb R^6}e^{(\vert\alpha\vert+3\vert\beta\vert+3)\bar\phi+2\Phi}(e^{2\phi}\vert\nabla_p\partial_x^\alpha\partial_p^\beta \tilde f\vert^2+\vert p\cdot\nabla_p\partial_x^\alpha\partial_p^\beta \tilde f\vert^2){\rm d}x{\rm d}p\nonumber\\
	\leq& Ce^{\bar\phi}\sum_{\vert\alpha\vert+ \vert\beta\vert\leq3}\int_{\mathbb R^6}e^{(1+\vert\alpha\vert+3\vert\beta\vert)\bar\phi}(e^{2\phi}+\vert p\vert^2)^{\frac{1}{2}}\vert\partial_x^\alpha\partial_p^\beta \tilde f\vert^2{\rm d}x{\rm d}p\nonumber\\
	&+C_\eta\sum_{\vert\alpha\vert+\vert\beta\vert\leq3}
\left\|e^{\frac{1}{2}(1+\vert\alpha\vert+3\vert\beta\vert)\bar\phi}(e^{2\phi}+\vert p\vert^2)^{\frac{1}{4}}\nabla_x\partial_x^\alpha\partial_p^\beta f\right\|^2_{L^2_{x,p}}\\
	&+C_\eta\sum_{\substack{\vert\alpha\vert+\vert\beta\vert\leq3\\|\beta|<3}}
\left\|e^{\frac{1}{2}(3+\vert\alpha\vert+3\vert\beta\vert)\bar\phi}(e^{2\phi}+\vert p\vert^2)^{\frac{1}{4}}\nabla_p\partial_x^\alpha\partial_p^\beta f\right\|^2_{L^2_{x,p}}\\
&+Ce^{\bar\phi}(\left\|(\partial_t\Phi,\nabla_x\Phi)\right\|_{H_x^3}^2+\left\|\Phi\right\|_{L^2_x}^2)\nonumber\\
	&+C_\eta\sum_{\vert\alpha\vert+ \vert\beta\vert\leq3}\int_{\mathbb R^6}e^{(\vert\alpha\vert+3\vert\beta\vert+2)\bar\phi+2\Phi}(e^{2\phi}\vert\nabla_p\partial_x^\alpha\partial_p^\beta f\vert^2+\vert p\cdot\nabla_p\partial_x^\alpha\partial_p^\beta  f\vert^2){\rm d}x{\rm d}p.
\end{align*}
Next, applying Gr\"{o}nwall's inequality to the above inequality and using \eqref{eng-tt} and \eqref{4.7}, we get
\begin{align}\label{4.10}
\sum_{\vert\alpha\vert+ \vert\beta\vert\leq3}&\int_{\mathbb R^6}e^{(1+\vert\alpha\vert+3\vert\beta\vert)\bar\phi}(e^{2\phi}+\vert p\vert^2)^{\frac{1}{2}}\vert\partial_x^\alpha\partial_p^\beta \tilde f\vert^2{\rm d}x{\rm d}p\nonumber\\
	&+\sum_{\vert\alpha\vert+ \vert\beta\vert\leq3}\int_0^t\int_{\mathbb R^6}e^{(\vert\alpha\vert+3\vert\beta\vert+3)\bar\phi+2\Phi}(e^{2\phi}\vert\nabla_p\partial_x^\alpha\partial_p^\beta \tilde f\vert^2+\vert p\cdot\nabla_p\partial_x^\alpha\partial_p^\beta \tilde f\vert^2){\rm d}x{\rm d}p{\rm d}s\nonumber\\
\leq& C(T_c)\Big(\mathcal E^{\delta_2}(f,\Phi)(0)+\left\|(1+|p|^2)^{\frac{\delta_2}{2}}\nabla_p^4f_0\right\|_{H^1_{x,p}}^2+\left\|\Phi_0\right\|_{L^2_x}^2\Big),
\end{align}
for $t\in(T_c,\infty)$. Combining \eqref{4.7} and \eqref{4.10} gives \eqref{4.2}. Thus, \eqref{long-t} follows from \eqref{H2-lim}
and Sobolev's inequality.

\subsection{Local existence}\label{sec-loc}
In this subsection, we porve the local-in-time existence of the Cauchy problem \eqref{vnfp-s2}. We begin by constructing the following sequence of approximate solutions:
\begin{equation}\label{fn-eq}
	\left\{
	\begin{array}{l}
		\partial_t F^{n+1}+\nabla_p\sqrt{e^{2\phi^n}+\vert p\vert^2}\cdot\nabla_x F^{n+1}-\nabla_x\sqrt{e^{2\phi^n}+\vert p\vert^2}\cdot\nabla_p F^{n+1}=e^{2\phi^n}\nabla_p\cdot(\Lambda_{\phi^n,p}\nabla_pF^{n+1}),\\
		\partial_t^2\phi^{n+1}-\Delta_x\phi^{n+1}=-e^{2\phi^n}\int_{\mathbb R^3}\frac{F^{n+1}}{\sqrt{e^{2\phi^n}+\vert p\vert^2}}{\rm d}p,\ t>0,\ x\in\mathbb R^3,\ p\in\mathbb R^3,
	\end{array}\right.
\end{equation}
with initial data
\begin{equation}\label{fn-id}
	F^{n+1}(0,x,p)=F_0(x,p),\ \phi^{n+1}(0,x)=\phi_0(x),\ \partial_t\phi^{n+1}(0,x)=\phi_1(x),
\end{equation}
and starting from $F^0(t,x,p)=F_0(x,p)$, $\phi^0(t,x)=\phi_0(x)$.

Denote $f^{n+1}=F^{n+1}-\bar F$, $\Phi^{n+1}=\phi^{n+1}-\bar\phi$.  $[f^{n+1}(t,x,p),\Phi^{n+1}(t,x)]$ satisfies
\begin{align}
	\left\{\begin{array}{rll}
		&\partial_t f^{n+1}+\nabla_p\sqrt{e^{2\phi^n}+\vert p\vert^2}\cdot\nabla_x f^{n+1}-\nabla_x\sqrt{e^{2\phi^n}+\vert p\vert^2}\cdot\nabla_p f^{n+1}\\[2mm]&\quad=\nabla_x\sqrt{e^{2\phi^n}+\vert p\vert^2}\cdot\nabla_p\bar F
+e^{2\phi^n}\nabla_p\cdot(\Lambda_{\phi^n,p}\nabla_pf^{n+1})\\[2mm]&\qquad\qquad+\big(e^{2\phi^n}\nabla_p\cdot(\Lambda_{\phi^n,p}\nabla_p\bar F)-e^{2\bar\phi}\nabla_p\cdot(\Lambda_{\bar\phi,p}\nabla_p\bar F)\big),\\[2mm]
		&\partial_t^2\Phi^{n+1}-\Delta_x\Phi^{n+1}=-e^{2\phi^n}\int_{\mathbb R^3}\frac{f^{n+1}}{\sqrt{e^{2\phi^n}+\vert p\vert^2}}{\rm d}p-\int_{\mathbb R^3}\Big(\frac{e^{2\phi^n}\bar F}{\sqrt{e^{2\phi^n}+\vert p\vert^2}}-\frac{e^{2\bar\phi}\bar F}{\sqrt{e^{2\bar\phi}+\vert p\vert^2}}\Big){\rm d}p,\notag
\end{array}
\right.
\end{align}
with
\begin{equation*}
	f^{n+1}(0,x,p)=f_0(x,p),\ \Phi^{n+1}(0,x)=\Phi_0(x),\ \partial_t\Phi^{n+1}(0,x)=\Phi_1(x),
\end{equation*}
and
\begin{equation*}
	f^{0}(t,x,p)=f_0(x,p),\ \Phi^{0}(t,x)=\Phi_0(x).
\end{equation*}
Let us now define the following energy functionals
\begin{align}
	E^\gamma(f^n,\Phi^n)(t)=&\sum_{\substack{ \vert\alpha\vert+\vert\beta\vert\leq4\\\vert\beta\vert<4}}\int_{\mathbb R^6}e^{(\vert\alpha\vert+3\vert\beta\vert)\bar\phi}(e^{2\phi^{n-1}}+\vert p\vert^2)^\gamma\vert\partial_x^\alpha\partial_p^\beta f^{n}\vert^2{\rm d}x{\rm d}p\notag\\&+\sum_{\vert\alpha\vert\leq3}\int_{\mathbb R^3}(\vert\partial_t\partial_x^\alpha\Phi^{n}\vert^2+\vert\nabla_x\partial_x^\alpha\Phi^{n}\vert^2){\rm d}x+\sum_{\vert\alpha\vert=4}\int_{\mathbb R^3}e^{\bar\phi}(\vert\partial_t\partial_x^\alpha\Phi^{n}\vert^2+\vert\nabla_x\partial_x^\alpha\Phi^{n}\vert^2){\rm d}x,\notag
\end{align}
and
\begin{align}
		E_0^\gamma(f,\Phi)(t)=&\sum_{\substack{0\leq \vert\alpha\vert+\vert\beta\vert\leq4\\\vert\beta\vert<4}}\int_{\mathbb R^6}e^{(\vert\alpha\vert+3\vert\beta\vert)\bar\phi}(1+\vert p\vert^2)^\gamma\vert\partial_x^\alpha\partial_p^\beta f\vert^2{\rm d}x{\rm d}p\notag\\&+\sum_{\vert\alpha\vert\leq3}\int_{\mathbb R^3}(\vert\partial_x^\alpha\pa_t\Phi\vert^2+\vert\nabla_x\partial_x^\alpha\Phi\vert^2){\rm d}x+\sum_{\vert\alpha\vert=4}\int_{\mathbb R^3}e^{\bar\phi}(\vert\partial_x^\alpha\pa_t\Phi\vert^2+\vert\nabla_x\partial_x^\alpha\Phi\vert^2){\rm d}x.\notag
\end{align}
Then the local existence of the Cauchy problem of \eqref{vnfp-s2} and \eqref{id-s2} can be stated in the following theorem.

\begin{theorem}[Local existence]\label{loc-ex}
There exists constant $M_0>0$ such that if $E_0^\gamma(f_0,\Phi_0)\leq M_0$, where $\partial_x^\alpha\pa_t\Phi=\pa_x^\al \Phi_1$, then there exists a $T_0>0$ such that the Cauchy problem of \eqref{vnfp-s2} and \eqref{id-s2} admits a unique strong solution $[f,\Phi]$ satisfying
\begin{align}
\CE^\ga(f,\Phi)(t)\leq 4M_0,\label{loc-bd}
\end{align}
on the time interval $[0,T_0].$
\end{theorem}

Before proving  Theorem \ref{loc-ex}, we first prove the following lemma about  the uniform bound on the approximate solution sequence $[f^n,\Phi^n].$

\begin{lemma}\label{loc-tt}
There exist $M_0>0$ and $T_\ast>0$, such that if $E_0^\gamma(f_0,\Phi_0)\leq M_0$ and $$\mathop{\sup}_{t\in[0,T_\ast]}(E^{\delta_1}(f^k,\Phi^k)(t)+E^{\delta_2}(f^k,\Phi^k)(t))\leq 2M_0,$$
for all $k\in[1,n]$, then
the Cauchy problem \eqref{fn-eq} and \eqref{fn-id} admits a unique solution $[f^{n+1},\Phi^{n+1}]$ satisfying
\begin{equation}\label{4.5}
	\mathop{\sup}_{t\in[0,T_\ast]}\left(E^{\delta_1}(f^{n+1},\Phi^{n+1})(t)+E^{\delta_2}(f^{n+1},\Phi^{n+1})(t)\right)\leq 2M_0.
\end{equation}	
	\end{lemma}
\begin{proof}
We only prove the bound \eqref{4.5}, the existence of solution to the linear system
 \eqref{fn-eq} and \eqref{fn-id} is standard.
By a similar argument used to derive \eqref{loc-eng} in Subsection \ref{sec-ft-eng}, we obtain
	\begin{equation*}
		\begin{split}
			\frac{d}{dt}&\left(E^{\delta_1}(f^{n+1},\Phi^{n+1})(t)+E^{\delta_2}(f^{n+1},\Phi^{n+1})(t)\right)\\
\leq& Ce^{CM_0^{1/2}}\left(E^{\delta_1}(f^{n+1},\Phi^{n+1})(t)+E^{\delta_2}(f^{n+1},\Phi^{n+1})(t)
+\left\|\nabla_x\Phi^n\right\|_{H^3_x}^2+\left\|\Phi^n\right\|_{L^2_x}^2\right)\\
			&\leq Ce^{CM_0^{1/2}}\left(E^{\delta_1}(f^{n+1},\Phi^{n+1})(t)+E^{\delta_2}(f^{n+1},\Phi^{n+1})(t)+M_0+t^2M_0\right).
		\end{split}
	\end{equation*}
Applying Gr\"{o}nwall's inequality to the above inequality yields
	\begin{equation}
		\begin{split}
E^{\delta_1}(f^{n+1},\Phi^{n+1})(t)+E^{\delta_2}(f^{n+1},\Phi^{n+1})(t)
\leq e^{Ce^{CM_0^{1/2}}T}\left(M_0+M_0(T+\frac{T^3}{3})\right)\notag
		\end{split}
	\end{equation}
	for all $t\in[0,T]$ for some $T>0$.  By choosing $T_\ast$ to be small enough such that $T_\ast+\frac{T_\ast^3}{3}<\frac{1}{2}$ and $M_0$ to be sufficiently small such that $e^{Ce^{CM_0^{1/2}}T_\ast}<\frac{4}{3}$,  \eqref{4.5} holds. This completes the proof of  the lemma.
\end{proof}

The following lemma is a direct consequence of  Lemma \ref{loc-ex} by noting
\begin{align}
	(e^{2\phi^{n}}+\vert p\vert^2)^\gamma\sim(e^{2\phi}+\vert p\vert^2)^\gamma\sim(1+\vert p\vert^2)^\ga,\label{eq-w}
\end{align}
for $t\in[0,T_\ast]$.

\begin{lemma}\label{ubd}
Let $M_0$ and $T_\ast$ be  given in Lemma \ref{loc-tt}. It holds that
\begin{align}\label{fn-bd}
	\sum_{\substack{\vert\alpha\vert+\vert\beta\vert\leq4\\\vert \beta\vert<4}}&\int_{\mathbb R^6}\big((1+\vert p\vert^2)^{\delta_1}+(1+\vert p\vert^2)^{\delta_2}\big)\vert\partial_x^\alpha\partial_p^\beta f^n\vert^2{\rm d}x{\rm d}p\notag\\&+\sum_{\vert\alpha\vert\leq4}\int_{\mathbb R^3}(\vert\partial_t\partial_x^\alpha\Phi^n\vert^2+\vert\nabla_x\partial_x^\alpha\Phi^n\vert^2){\rm d}x\leq 3M_0,
\end{align}
for all $t\in[0,T_\ast]$ and $n\in\mathbb N^+$.
\end{lemma}

\begin{proof}[Proof of Theorem \ref{loc-ex}]
With Lemmas \ref{loc-tt} and \ref{ubd}, to prove Theorem \ref{loc-ex}, it remains  to verify that $[f^{n},\Phi^n]$ is
a Cauchy sequence in some weighted $L^2_{x,p}\times L^2_x$. For this,
Let us denote $h^{n+1}=f^{n+1}-f^n$, $\Psi^{n+1}=\Phi^{n+1}-\Phi^n$, and consider the following Cauchy problem
\begin{equation}\label{4.69}
	\left\{
	\begin{array}{rll}
		\partial_t h^{n+1}+&\nabla_p\sqrt{e^{2\phi^n}+\vert p\vert^2}\cdot\nabla_x h^{n+1}+\nabla_x\sqrt{e^{2\phi^n}+\vert p\vert^2}\cdot\nabla_p h^{n+1}\\
		=&(\nabla_p\sqrt{e^{2\phi^{n-1}}+\vert p\vert^2}-\nabla_p\sqrt{e^{2\phi^n}+\vert p\vert^2})\cdot\nabla_x f^n\\&+(\nabla_x\sqrt{e^{2\phi^{n-1}}+\vert p\vert^2}-\nabla_x\sqrt{e^{2\phi^n}+\vert p\vert^2})\cdot\nabla_pf^n\\
		&+(\nabla_x\sqrt{e^{2\phi^{n-1}}+\vert p\vert^2}-\nabla_x\sqrt{e^{2\phi^n}+\vert p\vert^2})\cdot\nabla_p\bar F+e^{2\phi^n}\nabla_p\cdot(\Lambda_{\phi^n,p}\nabla_ph^{n+1})\\
		&+e^{2\phi^n}\nabla_p\cdot(\Lambda_{\phi^n,p}\nabla_pf^n)-e^{2\phi^{n-1}}\nabla_p\cdot(\Lambda_{\phi^{n-1},p}\nabla_p f^n)\\
		&+(e^{2\phi^n}\nabla_p\cdot(\Lambda_{\phi^n,p}\nabla_p\bar F)-e^{2\phi^{n-1}}\nabla_p\cdot(\Lambda_{\phi^{n-1},p}\nabla_p\bar F)),\\
		\partial_t^2\Psi^{n+1}&-\Delta_x\Psi^{n+1}=-e^{2\phi^n}\int_{\mathbb R^3}\frac{g^{n+1}}{\sqrt{e^{2\phi^n}+\vert p\vert^2}}{\rm d}p-\int_{\mathbb R^3}\Big(\frac{e^{2\phi^n}f^n}{\sqrt{e^{2\phi^n}+\vert p\vert^2}}-\frac{e^{2\phi^{n-1}}f^n}{\sqrt{e^{2\phi^{n-1}}+\vert p\vert^2}}\Big){\rm d}p\\
		&\qquad\qquad\qquad-\int_{\mathbb R^3}\Big(\frac{e^{2\phi^n}\bar F}{\sqrt{e^{2\phi^n}+\vert p\vert^2}}-\frac{e^{2\phi^{n-1}}\bar F}{\sqrt{e^{\phi^{n-1}}+\vert p\vert^2}}\Big){\rm d}p,
	\end{array}\right.
\end{equation}
with
\begin{equation*}
	h^{n+1}(0)=0,\ \Psi^{n+1}(0)=0,\ \partial_t\Psi^{n+1}(0)=0.
\end{equation*}
Multiplying \eqref{4.69}$_1$ by $(e^{2\phi^n}+\vert p\vert^2)^\gamma h^{n+1}$ and \eqref{4.69}$_2$ by $\partial_t\Psi^{n+1}$, respectively, and taking the summation, we have by using \eqref{fn-bd} that
\begin{align}
\frac{1}{2}\frac{d}{dt}\int_{\mathbb R^6}&(e^{2\phi^n}+\vert p\vert^2)^\gamma\vert h^{n+1}\vert^2{\rm d}x{\rm d}p+\frac{1}{2}\int_{\mathbb R^3}(\vert\partial_t\Psi^{n+1}\vert^2+\vert\nabla_x\Psi^{n+1}\vert^2){\rm d}x\nonumber\\
&+\int_{\mathbb R^6} e^{2\phi^n}(e^{2\phi^n}+\vert p\vert^2)^{\gamma-\frac{1}{2}}(e^{2\phi^n}\vert\nabla_p h^{n+1}\vert^2+\vert p\cdot\nabla_p h^{n+1}){\rm d}x{\rm d}p\nonumber\\
\leq& Ce^{CM_0^{1/2}}\Big(\int_{\mathbb R^6}(e^{2\phi^n}+\vert p\vert^2)^{\gamma}\vert h^{n+1}\vert^2{\rm d}x{\rm d}p+\int_{\mathbb R^3}(\vert\partial_t\Psi^n\vert^2+\vert\nabla_x\Psi^n\vert^2+\vert\partial_t\Psi^{n+1}\vert^2){\rm d}x\Big).\notag
\end{align}
Applying Gr\"{o}nwall's inequality to the above inequality, we  obtain for $t\in[0,T_\ast]$ that
\begin{equation}
	\begin{split}
\int_{\mathbb R^6}(e^{2\phi^n}+\vert p\vert^2)^\gamma&\vert g^{n+1}\vert^2{\rm d}x{\rm d}p+\frac{1}{2}\int_{\mathbb R^3}(\vert\partial_t\Psi^{n+1}\vert^2+\vert\nabla_x\Psi^{n+1}\vert^2){\rm d}x\\
\leq& e^{e^{CM_0^{1/2}}T_\ast}T_\ast	\Big(\int_{\mathbb R^6}(e^{2\phi^{n-1}}+\vert p\vert^2)^\gamma\vert g^{n}\vert^2{\rm d}x{\rm d}p+\frac{1}{2}\int_{\mathbb R^3}(\vert\partial_t\Psi^{n}\vert^2+\vert\nabla_x\Psi^{n}\vert^2){\rm d}x\Big).\notag
\end{split}
\end{equation}
Choose $T_0\in(0,T_\ast]$  such that $e^{e^{CM_0^{1/2}}T_0}T_0<1$. Consequently, $(f^n,\Phi^n)$ is a Cauchy sequence in the following function space
\begin{align}
	X(T_0)=\Big\{
(f,\Phi)\Big|\sup\limits_{t\in[0,T_0]}&\int_{\mathbb R^6}\big((1+\vert p\vert^2)^{\delta_1}+(1+\vert p\vert^2)^{\delta_2}\big)\vert f\vert^2{\rm d}x{\rm d}p\notag\\&+\sup\limits_{t\in[0,T_0]}\int_{\mathbb R^3}(\vert\partial_t\Phi\vert^2+\vert\Phi\vert^2){\rm d}x<+\infty\notag
\Big\}.
\end{align}
By taking $n\rightarrow\infty$, we thus obtain a unique strong solution  to \eqref{vnfp-s2} and \eqref{id-s2} in the time interval $[0,T_0]$.
Moreover, \eqref{loc-bd} follows from \eqref{fn-bd} and \eqref{eq-w}. This completes the proof of Theorem \ref{loc-ex}.
\end{proof}

Now, we are in a position to complete the proof of Theorem \ref{thm1}.
\begin{proof}[The proof of Theorem \ref{thm1}] The global existence of \eqref{vnfp-s2} and \eqref{id-s2} follows from the local existence  in Theorem \ref{loc-ex} and the \textit{a priori} energy estimate \eqref{eng-tt} proved in Subsection \ref{sec-ft-eng} and Subsection \ref{sec-ift-eng} through the continuity argument. Moreover, since $\bar{F}(t,p)\geq0$ was proved in \cite[Theorem 2.1, pp.3703]{ACP-2014}, the non-negativity of $F(t,x,p)$ follows from the large time behavior \eqref{long-t}. This completes the proof of Theorem \ref{thm1}.
\end{proof}

\section{Self-similar solution}\label{sec-sss}
In this section, we will show that the Cauchy problem \ref{pt-vnfp} and \ref{pt-id} admits a self-similar solution in the form of
\eqref{sss}. Specifically, we will prove Theorem \ref{thm2}  which is about the the global existence of the following Cauchy problem
\begin{eqnarray}\label{g-eq-s3}
	\left\{
	\begin{array}{l}
		\partial_t g-3\bar\phi' g-\bar\phi' q\cdot\nabla_qg+\nabla_q\sqrt{e^{2\Phi}+\vert q\vert^2}\cdot\nabla_x g-\nabla_x\sqrt{e^{2\Phi}+\vert q\vert^2}\cdot\nabla_q g=\nabla_x\sqrt{e^{2\Phi}+\vert q\vert^2}\cdot\nabla_q\bar G\\[2mm]
		\quad+e^{\bar\phi+2\Phi}\nabla_q\cdot(\Lambda_{\Phi,q}\nabla_qg)+(e^{\bar\phi+2\Phi}\nabla_q\cdot(\Lambda_{\Phi,q}\nabla_q\bar G)-e^{\bar\phi}\nabla_q\cdot(\Lambda_{0,q}\nabla_q\bar G)),\\[2mm]
		\partial_t^2\Phi-\Delta_x\Phi=-e^{\bar\phi+2\Phi}\int_{\mathbb R^3}\frac{g}{\sqrt{e^{2\Phi}+\vert q\vert^2}}{\rm d}q-\int_{\mathbb R^3}\Big(\frac{e^{\bar\phi+2\Phi}\bar G}{\sqrt{e^{2\Phi}+\vert q\vert^2}}-\frac{e^{\bar\phi}\bar G}{\sqrt{1+\vert p\vert^2}}\Big){\rm d}q,
	\end{array}\right.
\end{eqnarray}
with initial data
\begin{equation}
	(g(0,x,q),\Phi(0,x),\partial_t\Phi(0,x))=(g_0(x,q),\Phi_0(x),\Phi_1(x)),\ x\in\mathbb R^3,\ q\in\mathbb R^3.\notag
\end{equation}
\begin{proof}[The proof of
Theorem \ref{thm2}] The proof of Theorem \ref{thm2} follows from a similar approach as the one for   Theorem \ref{thm1}. We will first verify the \textit{a priori} estimate \eqref{eng-es-ss}  under the {\it a priori} assumption
\begin{equation}\label{g-aps}
\tilde{\mathcal E}^\lambda(t)\leq 2\eps_1,
\end{equation}
for any $t>0$.
The proof is divided into the following three steps.

\noindent\underline{{\it Step 1. Zeroth-order energy estimates.}}
Choosing $\la\in(\frac{1}{2},\frac{3}{2})$, multiplying \eqref{g-eq-s3}$_{1}$ by $(e^{2\Phi}+\vert q\vert^2)^\lambda g$, and integrating the resulting equality with respect to $(x,q)\in\mathbb R^3\times\mathbb R^3$, one has
\begin{align}\label{6.1}
	\frac{1}{2}\frac{d}{dt}\int_{\mathbb R^6}&(e^{2\Phi}+\vert q\vert^2)^\lambda\vert g\vert^2{\rm d}x{\rm d}q\nonumber\\
	=&\int_{\mathbb R^6}\Big(\lambda\partial_t\Phi e^{2\Phi}(e^{2\Phi}+\vert q\vert^2)^{\lambda-1}\vert g\vert^2+(3\bar\phi' g+\bar\phi' q\cdot\nabla_qg)(e^{2\Phi}+\vert q\vert^2)^{\lambda}g\Big){\rm d}x{\rm d}q\nonumber\\
	&+\int_{\mathbb R^6}\nabla_x\sqrt{e^{2\Phi}+\vert q\vert^2}\cdot\nabla_q\bar F(e^{2\Phi}+\vert q\vert^2)^\lambda g{\rm d}x{\rm d}q\nonumber\\
	&+\int_{\mathbb R^6}e^{\bar\phi+2\Phi}\nabla_q\cdot(\Lambda_{\Phi,q}\nabla_q g)(e^{2\Phi}+\vert q\vert^2)^\lambda g{\rm d}x{\rm d}q\nonumber\\
	&+\int_{\mathbb R^6}(e^{\bar\phi+2\phi}\nabla_q\cdot(\Lambda_{\Phi,q}\nabla_q\bar G)-e^{\bar\phi}\nabla_q\cdot(\Lambda_{0,q}\nabla_q\bar G))(e^{2\Phi}+\vert q\vert^2)^\lambda g{\rm d}x{\rm d}q.\nonumber\\
	:=&K_1+K_2+K_3+K_4.
\end{align}
We now  estimate $K_i$ $(1\leq i\leq4)$ term by term.
For $K_1$, by noting that $|\partial_t\bar\phi|$ is bounded (see \cite[Theorem 2.1, pp.3703]{ACP-2014})) and using Sobolev's inequality, we have
\begin{equation}\label{6.2}
	\begin{split}
		K_1&=\bar\phi'(\frac{3}{2}-\lambda)\int_{\mathbb R^6}(e^{2\Phi}+\vert q\vert^2)^\lambda\vert g\vert^2{\rm d}x{\rm d}q+\lambda(\bar\phi'+\partial_t\Phi)\int_{\mathbb R^6}e^{2\Phi}(e^{2\Phi}+\vert q\vert^2 )^{\lambda-1}\vert g\vert^2{\rm d}x{\rm d}q\\
		&\leq \bar\phi'(\frac{3}{2}-\lambda)\int_{\mathbb R^6}(e^{2\Phi}+\vert q\vert^2)^\lambda\vert g\vert^2{\rm d}x{\rm d}q+C(\partial_t\phi)_+\int_{\mathbb R^6}(e^{2\Phi}+\vert q\vert^2 )^{\lambda}\vert g\vert^2{\rm d}x{\rm d}q.
	\end{split}
\end{equation}
For $K_2$, $K_3$ and $K_4$, by  Lemma \ref{lemma A6}, we get
\begin{align}\label{K2}
|	K_2|+K_3+|K_4|&\leq -(1-\eta)\int_{\mathbb R^6}e^{\bar\phi+2\Phi}(e^{2\Phi}+\vert q\vert^2)^{\lambda-\frac{1}{2}}(e^{2\Phi}\vert\nabla_q g\vert^2+\vert q\cdot\nabla_q g\vert^2){\rm d}x{\rm d}q\nonumber\\
&\quad+ Ce^{(\frac{3}{2}-\lambda)\bar\phi}\int_{\mathbb R^6}(e^{2\Phi}+\vert q\vert^2)^\lambda\vert g\vert^2{\rm d}x{\rm d}q+Ce^{(\frac{3}{2}-\lambda)\bar\phi}\left\|\nabla_x\Phi\right\|_{L^2_x}^2\nonumber\\
&\quad+C_\eta e^{\bar\phi}\left\|\Phi\right\|_{L^2_x}^2\int_{\mathbb R^3}(1+\vert q\vert^2)^{\lambda-\frac{1}{2}}(\vert\nabla_q\bar G\vert^2+\vert q\cdot\nabla_q\bar G\vert^2){\rm d}q.
\end{align}
Substituting \eqref{6.2} and \eqref{K2}into \eqref{6.1}, we have
\begin{align}\label{6.6}
		\frac{d}{dt}&\int_{\mathbb R^6}(e^{2\Phi}+\vert q\vert^2)^\lambda\vert g\vert^2{\rm d}x{\rm d}q+\int_{\mathbb R^6}e^{\bar\phi+2\Phi}(e^{2\Phi}+\vert q\vert^2)^{\lambda-\frac{1}{2}}(e^{2\Phi}\vert\nabla_q g\vert^2+\vert q\cdot\nabla_q g\vert^2){\rm d}x{\rm d}q\nonumber\\
		&\leq \Big(2\bar\phi'(\frac{3}{2}-\lambda)+C(\partial_t\phi)_+\Big)\int_{\mathbb R^6}(e^{2\Phi}+\vert q\vert^2)^\lambda\vert g\vert^2{\rm d}x{\rm d}q+Ce^{(\frac{3}{2}-\lambda)\bar\phi}\tilde{\mathcal E}_{0,0}^\lambda(t)\notag\\
		&\quad+Ce^{\bar\phi}\left\|\Phi\right\|_{L^2_x}^2\int_{\mathbb R^3}(1+\vert p\vert^2)^{\lambda-\frac{1}{2}}(\vert\nabla_q\bar G\vert^2+\vert q\cdot\nabla_q\bar G\vert^2){\rm d}q.
\end{align}
We now turn to estimate the gravitational field $\Phi$. Multiplying \eqref{g-eq-s3}$_2$ by $\partial_t\Phi$ and integrating over $\mathbb R^3$ yield
\begin{align}\label{6.7}
	\frac{1}{2}\frac{d}{dt}\int_{\mathbb R^3}&(\vert\partial_t\Phi\vert^2+\vert\nabla_x\Phi\vert^2){\rm d}x\nonumber\\
	&=-\int_{\mathbb R^3}-e^{\bar\phi+2\Phi}\int_{\mathbb R^3}\frac{g}{\sqrt{e^{2\Phi}+\vert q\vert^2}}{\rm d}q-\int_{\mathbb R^3}\int_{\mathbb R^3}\Big(\frac{e^{\bar\phi+2\Phi}\bar G}{\sqrt{e^{2\Phi}+\vert q\vert^2}}-\frac{e^{\bar\phi}\bar G}{\sqrt{1+\vert p\vert^2}}\Big){\rm d}p\partial_t\Phi{\rm d}x\nonumber\\
	&:=L_1+L_2.
\end{align}
Then it is direct to check that
\begin{equation*}
	|L_1|\leq Ce^{\bar\phi}\left\|\partial_t\Phi\right\|_{L^2_x}\left\|\int_{\mathbb R^3}\frac{g}{\sqrt{e^{2\Phi}+\vert q\vert^2}}{\rm d}q\right\|_{L^2_x}\\
	\leq Ce^{\bar\phi}\Big(\left\|\partial_t\Phi\right\|_{L^2_x}^2+\left\|(e^{2\Phi}+\vert p\vert^2)^{\frac{\lambda}{2}}g\right\|_{L^2_{x,q}}^2\Big),
\end{equation*}
and
\begin{align*}
	|L_2|\leq&C\int_{\mathbb R^3}\int_{\mathbb R^3}\vert\Phi\vert\frac{e^{\bar\phi}}{\sqrt{1+\vert q\vert^2}}\vert\bar G\vert{\rm d}q\vert\partial_t\Phi\vert{\rm d}x\leq Ce^{\bar\phi}\Big(\left\|\partial_t\Phi\right\|_{L^2_x}^2+\left\|\Phi\right\|_{L^2_x}^2\int_{\mathbb R^3}(1+\vert q\vert^2)^{\lambda}\vert\bar G\vert^2{\rm d}q\Big)\notag\\
\leq& Ce^{\bar\phi}(\left\|\partial_t\Phi\right\|_{L^2_x}^2+\left\|\Phi\right\|_{L^2_x}^2).
\end{align*}
By substituting the above two inequalities into \eqref{6.7} and by Lemma \ref{lemma A2}, we obtain
\begin{equation}
	\frac{d}{dt}\int_{\mathbb R^3}(\vert\partial_t\Phi\vert^2+\vert\nabla_x\Phi\vert^2){\rm d}x\leq Ce^{\bar\phi}\tilde{\mathcal E}_{0,0}^{\lambda}(t)+Ce^{\bar\phi}\left\|\Phi\right\|_{L^2_x}^2,\notag
\end{equation}
which combines with \eqref{6.6} and Lemma \ref{lemma A2} gives
\begin{align}\label{6.9}
	\frac{\rm d}{dt}&\tilde{\mathcal E}_{0,0}^\lambda(t)+\tilde{\mathcal D}^\lambda_{0,0}(t)\notag\\
	&\leq \Big(2\bar\phi'(\frac{3}{2}-\lambda)+C(\partial_t\phi)_+\Big)\int_{\mathbb R^6}(e^{2\Phi}+\vert q\vert^2)^\lambda\vert g\vert^2{\rm d}x{\rm d}q+Ce^{(\frac{3}{2}-\lambda)\bar\phi}\tilde{\mathcal E}_{0,0}^\lambda(t)\notag\\
	&\quad+Ce^{\bar\phi}\left\|\Phi\right\|_{L^2_x}^2\int_{\mathbb R^3}(1+\vert p\vert^2)^{\lambda-\frac{1}{2}}(\vert\nabla_q\bar G\vert^2+\vert q\cdot\nabla_q\bar G\vert^2){\rm d}q+Ce^{(\frac{3}{2}-\lambda)\bar\phi}\left\|\Phi\right\|_{L^2_x}^2.
\end{align}
Note that $\frac{3}{2}-\lambda>0$ and $\bar\phi'(\infty)<0$, and  $C(\partial_t\phi)_++e^{(\frac{3}{2}-\la)\bar\phi}$ is integrable in time.
Similar to the estimate for  \eqref{z-lt}, we have from \eqref{6.9} that
\begin{equation}
	\tilde{\mathcal E}_{0,0}^{\lambda}(t)+\int_0^t\tilde{\mathcal D}_{0,0}^{\lambda}(s){\rm d}s\leq	C(T_c)(\tilde{\mathcal E}_{0,0}^{\la}(0)+\left\|\Phi_0\right\|_{L^2_x}^2),\notag
\end{equation}
where $0<T_c<\infty$ is given in \eqref{tc}.

Moreover, from \eqref{6.6}, we also have
\begin{align*}
		&\frac{d}{dt}\int_{\mathbb R^6}(e^{2\Phi}+\vert q\vert^2)^\lambda\vert g\vert^2{\rm d}x{\rm d}q-\frac{1}{2}\bar\phi'(\frac{3}{2}-\lambda)\int_{\mathbb R^6}(e^{2\Phi}+\vert q\vert^2)^\lambda\vert g\vert^2{\rm d}x{\rm d}q+C\tilde{\mathcal D}_{0,0}^\lambda(g,\Phi)(t)\\
		&\leq \Big(\frac{3}{2}\bar\phi'(\frac{3}{2}-\lambda)+C(\partial_t\phi)_+\Big)\int_{\mathbb R^6}(e^{2\Phi}+\vert q\vert^2)^\lambda\vert g\vert^2{\rm d}x{\rm d}q\\
			&\quad+Ce^{(\frac{3}{2}-\lambda)\bar\phi}\Big(\int_{\mathbb R^6}(e^{2\Phi}+\vert q\vert^2)^\lambda\vert g\vert^2{\rm d}x{\rm d}q+C(\tilde{\mathcal E}_{0,0}^{\lambda}(0)+\left\|\Phi_0\right\|_{L^2_x}^2)\Big)\\
		&\quad+C(T_c)t^2e^{\bar\phi}(\tilde{\mathcal E}_{0,0}^{\lambda}(0)+\left\|\Phi_0\right\|_{L^2_x}^2)\int_{\mathbb R^3}(1+\vert p\vert^2)^{\lambda-\frac{1}{2}}(\vert\nabla_q\bar G\vert^2+\vert q\cdot\nabla_q\bar G\vert^2){\rm d}q.
\end{align*}
As a consequence, Gr\"{o}nwall's inequality  and Lemma \ref{lemma A2} give
\begin{align}\label{6.11}
	e^{\frac{1}{2}\bar\phi(\lambda-\frac{3}{2})}&\int_{\mathbb R^6}(e^{2\Phi}+\vert q\vert^2)^\lambda\vert g\vert^2{\rm d}x{\rm d}q+C\int_0^te^{\frac{1}{2}\bar\phi(\frac{3}{2}-\lambda)}\bar{\tilde{\mathcal D}}_{0,0}^\lambda(s){\rm d}s\notag\\
\leq& C(T_c)(\tilde{\mathcal E}_{0,0}^{\lambda}(0)+\left\|\Phi_0\right\|_{L^2_x}^2).
\end{align}

\noindent\underline{{\it Step 2. First-order energy estimates.}}
Applying $\partial_{q_i}$ to \eqref{g-eq-s3}$_1$, multiplying the resulting equality by $(e^{2\Phi}+\vert q\vert^2)^\lambda\partial_{q_i}g$, and integrating the resulting equality over $\mathbb R^6$ yield
\begin{align}\label{6.12}
\frac{1}{2}\frac{d}{dt}\int_{\mathbb R^6}&(e^{2\Phi}+\vert q\vert^2)^\lambda\vert\partial_{q_i}g\vert^2{\rm d}x{\rm d}q\nonumber\\
	=&\int_{\mathbb R^6}\Big(\lambda\partial_t\Phi e^{2\Phi}(e^{2\Phi}+\vert q\vert^2)^{\lambda-1}\vert\partial_{q_i}g\vert^2+\partial_{q_i}(3\bar\phi' g+\bar\phi' q\cdot\nabla_qg)(e^{2\Phi}+\vert q\vert^2)^{\lambda}\partial_{q_i}g\Big){\rm d}x{\rm d}q\nonumber\\
	&+\int_{\mathbb R^6}(\nabla_x\partial_{q_i}\sqrt{e^{2\Phi}+\vert q\vert^2}\cdot\nabla_qg-\nabla_q\partial_{q_i}\sqrt{e^{2\Phi}+\vert q\vert^2}\cdot\nabla_xg)(e^{2\Phi}+\vert q\vert^2)^\lambda\partial_{q_i}g{\rm d}x{\rm d}q\nonumber\\
	&+\int_{\mathbb R^6}\partial_{q_i}(\nabla_x\sqrt{e^{2\Phi}+\vert q\vert^2}\cdot\nabla_q\bar G)(e^{2\Phi}+\vert q\vert^2)^\lambda \partial_{q_i}g{\rm d}x{\rm d}q\nonumber\\
	&+\int_{\mathbb R^6}\partial_{q_i}(e^{\bar\phi+2\Phi}\nabla_q\cdot(\Lambda_{\Phi,q}\nabla_q g))(e^{2\Phi}+\vert q\vert^2)^\lambda \partial_{q_i}g{\rm d}x{\rm d}q\nonumber\\
	&+\int_{\mathbb R^6}\partial_{q_i}(e^{\bar\phi+2\phi}\nabla_q\cdot(\Lambda_{\Phi,q}\nabla_q\bar G)-e^{\bar\phi}\nabla_q\cdot(\Lambda_{0,q}\nabla_q\bar G))(e^{2\Phi}+\vert q\vert^2)^\lambda \partial_{q_i}g{\rm d}x{\rm d}q\nonumber\\
	:=&K_5+K_6+K_7+K_8+K_9.
\end{align}
We estimate $K_{i}$ $(5\leq i\leq9)$ separately.
Similar to $K_1$,
\begin{equation}
	\begin{split}
		K_5&=\bar\phi'(\frac{5}{2}-\lambda)\int_{\mathbb R^6}(e^{2\Phi}+\vert q\vert^2)^\lambda\vert \partial_{q_i}g\vert^2{\rm d}x{\rm d}q+\lambda(\bar\phi'+\partial_t\Phi)\int_{\mathbb R^6}(e^{2\Phi}+\vert q\vert^2 )^{\lambda-1}\vert\partial_{q_i}g\vert^2{\rm d}x{\rm d}q\\
		&\leq \bar\phi'(\frac{5}{2}-\lambda)\int_{\mathbb R^6}e^{2\Phi}(e^{2\Phi}+\vert q\vert^2)^\gamma\vert \partial_{q_i}g\vert^2{\rm d}x{\rm d}q+C(\partial_t\phi)_+\int_{\mathbb R^6}(e^{2\Phi}+\vert q\vert^2)^\lambda\vert \partial_{q_i}g\vert^2{\rm d}x{\rm d}q.\notag
	\end{split}
\end{equation}
For $K_6$ and $K_7$,
by  H\"{o}lder's inequality and Lemma \ref{lemma A6}, it holds
\begin{equation*}
	\begin{split}
		|K_6|&\leq \left\|\nabla_x\Phi\right\|_{L^\infty_x}\int_{\mathbb R^6}(e^{2\Phi}+\vert q\vert^2)^\lambda\vert\nabla_q g\vert\vert\partial_{q_i}g\vert{\rm d}x{\rm d}q+\int_{\mathbb R^6}(e^{2\Phi}+\vert q\vert^2)^{\lambda}\vert\nabla_xg\vert\vert\partial_{q_i}g\vert{\rm d}x{\rm d}q,
	\end{split}
\end{equation*}
and
\begin{equation}
	\begin{split}
		|K_7|\leq Ce^{(\frac{3}{2}-\lambda)\bar\phi}\int_{\mathbb R^6}(e^{2\Phi}+\vert q\vert^2)^\lambda\vert\partial_{q_i}g\vert^2{\rm d}x{\rm d}q+Ce^{(\frac{3}{2}-\lambda)\bar\phi}\left\|\nabla_x\Phi\right\|_{L^2_x}^2.\notag
	\end{split}
\end{equation}
For $K_8$ and $K_9$, we apply Lemma \ref{lemma A6} to obtain
\begin{align*}
	K_8
	\leq&-(1-\eta)\int_{\mathbb R^6}e^{\bar\phi+2\Phi}(e^{2\Phi}+\vert q\vert^2)^{\lambda-\frac{1}{2}}(e^{2\Phi}\vert\nabla_q\partial_{q_i}g\vert^2+\vert q\cdot\nabla_q\partial_{q_i}g\vert^2){\rm d}x{\rm d}q
\notag\\&+C_\eta e^{\bar\phi}\int_{\mathbb R^6}(e^{2\Phi}+\vert q\vert^2)^\lambda\vert \partial_{q_i} g\vert^2{\rm d}x{\rm d}q+C\tilde{\mathcal D}^\la_{0,0}(t),
\end{align*}
and
\begin{align}
	|K_9|
	&\leq\eta\int_{\mathbb R^6} e^{\bar\phi+2\Phi}(e^{2\Phi}+\vert q\vert^2)^{\lambda-\frac{1}{2}}(e^{2\Phi}\vert\nabla_q\partial_{q_i}g\vert^2+\vert q\cdot\nabla_q\partial_{q_i}g\vert^2){\rm d}x{\rm d}q\nonumber\\
	&\quad+Ce^{\bar\phi}\int_{\mathbb R^6}(e^{2\Phi}+\vert q\vert^2)^\lambda\vert\nabla_qf\vert^2{\rm d}x{\rm d}q\nonumber\\
	&\quad+C_\eta\sum_{\vert\beta'\vert\leq1} e^{\bar\phi}\left\|\Phi\right\|_{L^2_x}^2\int_{\mathbb R^3}(1+\vert q\vert^2)^{\lambda-\frac{1}{2}}(\vert\nabla_q\partial^{\beta'}_{q}\bar G\vert^2+\vert q\cdot\nabla_q\partial_{q}^{\beta'}\bar G\vert^2){\rm d}q.\notag
\end{align}
Hence, substituting the above estimates into \eqref{6.12} gives
\begin{align}\label{6.16}
\frac{d}{dt}\int_{\mathbb R^6}&(e^{2\Phi}+\vert q\vert^2)^\lambda\vert\partial_{q_i}g\vert^2{\rm d}x{\rm d}q+\int_{\mathbb R^6}e^{\bar\phi}(e^{2\Phi}+\vert q\vert^2)^{\lambda-\frac{1}{2}}(e^{2\Phi}\vert\nabla_q\partial_{q_i}g\vert^2+\vert q\cdot\nabla_q\partial_{q_i}g\vert^2){\rm d}x{\rm d}q\nonumber\\
	\leq& \Big(2\bar\phi'(\frac{3}{2}-\lambda)+C(\partial_t\phi)_++Ce^{(\frac{3}{2}-\lambda)\bar\phi}\Big)\int_{\mathbb R^6}(e^{2\Phi}+\vert q\vert^2)^\lambda\vert\partial_{q_i}g\vert^2{\rm d}x{\rm d}p\nonumber\\
	&+(2\bar\phi'+C\left\|\nabla_x\Phi\right\|_{L^\infty_x})\int_{\mathbb R^6}(e^{2\Phi}+\vert q\vert^2)^\lambda\vert\nabla_qg\vert^2{\rm d}x{\rm d}q+C\tilde{\mathcal D}^\la_{0,0}(t)\notag\\&+\int_{\mathbb R^6}(e^{2\Phi}+\vert q\vert^2)^\lambda\vert\nabla_xg\vert\vert\nabla_qg\vert{\rm d}x{\rm d}q+Ce^{(\frac{3}{2}-\lambda)\bar\phi}\left\|\nabla_x\Phi\right\|_{L^2_x}^2\nonumber\\
	&+C_\eta\sum_{\vert\beta'\vert\leq1} e^{\bar\phi}\left\|\Phi\right\|_{L^2_x}^2\int_{\mathbb R^3}(1+\vert q\vert^2)^{\lambda-\frac{1}{2}}(\vert\nabla_q\partial^{\beta'}_{q}\bar G\vert^2+\vert q\cdot\nabla_q\partial_{q}^{\beta'}\bar G\vert^2){\rm d}q.
\end{align}
Next, applying $\partial_{x_i}$ to \eqref{g-eq-s3}$_1$, multiplying the resulting equality by $(e^{2\Phi}+\vert q\vert^2)^\lambda\partial_{x_i}g$, and integrating the resulting equality yield that
\begin{align}\label{6.17}
\frac{1}{2}\frac{d}{dt}\int_{\mathbb R^6}&(e^{2\Phi}+\vert q\vert^2)^\lambda\vert\partial_{x_i}g\vert^2{\rm d}x{\rm d}q\nonumber\\
	=&\int_{\mathbb R^6}\Big(\lambda\partial_t\Phi e^{2\Phi}(e^{2\Phi}+\vert q\vert^2)^{\lambda-1}\vert\partial_{x_i}g\vert^2+\partial_{x_i}(3\bar\phi' g+\bar\phi' q\cdot\nabla_qg)(e^{2\Phi}+\vert q\vert^2)^{\lambda}\partial_{x_i}g\Big){\rm d}x{\rm d}q\nonumber\\
	&+\int_{\mathbb R^6}(\nabla_x\partial_{x_i}\sqrt{e^{2\Phi}+\vert q\vert^2}\cdot\nabla_qg-\nabla_q\partial_{x_i}\sqrt{e^{2\Phi}+\vert q\vert^2}\cdot\nabla_xg)(e^{2\Phi}+\vert q\vert^2)^\lambda\partial_{x_i}g{\rm d}x{\rm d}q\nonumber\\
	&+\int_{\mathbb R^6}\nabla_x\partial_{x_i}\sqrt{e^{2\Phi}+\vert q\vert^2}\cdot\nabla_q\bar G(e^{2\Phi}+\vert q\vert^2)^\lambda \partial_{x_i}g{\rm d}x{\rm d}q\nonumber\\&+\int_{\mathbb R^6}\partial_{x_i}\big(e^{\bar\phi+2\Phi}\nabla_q\cdot(\Lambda_{\Phi,q}\nabla_q g)\big)(e^{2\Phi}+\vert q\vert^2)^\lambda \partial_{x_i}g{\rm d}x{\rm d}q\nonumber\\
	&+\int_{\mathbb R^6}\partial_{x_i}\big(e^{\bar\phi+2\Phi}\nabla_q\cdot(\Lambda_{\Phi,q}\nabla_q\bar G)\big)(e^{2\Phi}+\vert q\vert^2)^\lambda \partial_{x_i}g{\rm d}x{\rm d}q\nonumber\\
:=&K_{10}+K_{11}+K_{12}+K_{13}+K_{14}.
\end{align}
We estimate $K_{i}$ $(10\leq i\leq14)$ term by term.
Similar to  \eqref{6.16}, we have
\begin{equation}
	\begin{split}
	K_{10}&=\bar\phi'(\frac{3}{2}-\lambda)\int_{\mathbb R^6}(e^{2\Phi}+\vert q\vert^2)^\lambda\vert\partial_{x_i}g\vert^2{\rm d}x{\rm d}q+\lambda(\bar\phi'+\partial_t\Phi)\int_{\mathbb R^6}e^{2\Phi}(e^{2\Phi}+\vert q\vert^2 )^{\lambda-1}\vert \partial_{x_i}g\vert^2{\rm d}x{\rm d}q\\
		&\leq\bar\phi'(\frac{3}{2}-\lambda)\int_{\mathbb R^6}(e^{2\Phi}+\vert q\vert^2)^\lambda\vert\partial_{x_i}g\vert^2{\rm d}x{\rm d}q+C(\partial_t\phi)_+\int_{\mathbb R^6}(e^{2\Phi}+\vert q\vert^2 )^{\lambda}\vert \partial_{x_i}g\vert^2{\rm d}x{\rm d}q,\notag
	\end{split}
\end{equation}
\begin{align}
	|K_{11}|&\leq C\int_{\mathbb R^6}(e^{2\Phi}+\vert q\vert^2)^{\frac{\lambda}{2}}\vert\nabla_qg\vert (e^{2\Phi}+\vert q\vert^2)^{\frac{\lambda}{2}}\vert\partial_{x_i}g\vert(\vert\nabla_x\Phi\vert^2+\vert\nabla_x^2\Phi\vert){\rm d}x{\rm d}q\nonumber\\
	&\quad+C\int_{\mathbb R^6}(e^{2\Phi}+\vert q\vert^2)^\lambda\vert\nabla_xg\vert^2\vert\nabla_x\Phi\vert{\rm d}x{\rm d}q\nonumber\\
	&\leq C\left\|(\nabla_x\Phi,\nabla_x^2\Phi)\right\|_{L^\infty_x}\int_{\mathbb R^6}(e^{2\Phi}+\vert q\vert^2)^{\frac{\lambda}{2}}\vert\nabla_qg\vert (e^{2\Phi}+\vert q\vert^2)^{\frac{\gamma}{2}}\vert\partial_{x_i}g\vert{\rm d}x{\rm d}q\nonumber\\
	&\quad+C\left\|\nabla_x\Phi\right\|_{L^\infty_x}\int_{\mathbb R^6}(e^{2\Phi}
+\vert q\vert^2)^\lambda\vert\nabla_xg\vert^2{\rm d}x{\rm d}q,\notag
\end{align}
\begin{align}
|K_{12}|\leq Ce^{(\frac{3}{2}-\lambda)\bar\phi}\int_{\mathbb R^6} (e^{2\Phi}+\vert q\vert^2)^{\lambda}\vert\partial_{x_i}g\vert^2{\rm d}x{\rm d}q+Ce^{(\frac{3}{2}-\lambda)\bar\phi}(\left\|\nabla_x\Phi\right\|_{L^2_x}^2+\left\|\nabla_x^2\Phi\right\|_{L^2_x}^2),\notag
\end{align}
\begin{align}
	K_{13}&\leq -(1-\eta)\int_{\mathbb R^6}e^{\bar\phi+2\Phi}(e^{2\Phi}+\vert q\vert^2)^{\lambda-\frac{1}{2}}(e^{2\Phi}\vert\nabla_q\partial_{x_i}g\vert^2+\vert q\cdot\nabla_q\partial_{x_i}g\vert^2){\rm d}x{\rm d}q\nonumber\\
	&\quad+C_\eta\int_{\mathbb R^6}e^{\bar\phi}(e^{2\Phi}+\vert q\vert^2)^{\lambda-\frac{1}{2}}(\vert\nabla_q g\vert^2+\vert q\cdot\nabla_qg\vert^2){\rm d}x{\rm d}q+C_\eta e^{\bar\phi}\int_{\mathbb R^6}(e^{2\Phi}+\vert q\vert^2)^\lambda\vert\partial_{x_i}g\vert^2{\rm d}x{\rm d}q,\notag
\end{align}
and
\begin{align}
	|K_{14}|
	\leq& \eta\int_{\mathbb R^6}e^{\bar\phi+2\Phi}(e^{2\Phi}+\vert q\vert^2)^{\lambda-\frac{1}{2}}(e^{2\Phi}\vert\nabla_q\partial_{x_i}g\vert^2+\vert q\cdot\nabla_q\partial_{x_i}g\vert^2){\rm d}x{\rm d}q\notag\\
	&+Ce^{\bar\phi}\int_{\mathbb R^3}(e^{2\Phi}+\vert q\vert^2)^\lambda\vert\nabla_xf\vert^2{\rm d}x{\rm d}q\nonumber\\
	&+C_\eta e^{\bar\phi}\left\|\partial_{x_i}\Phi\right\|_{L^2_x}^2\int_{\mathbb R^3}(1+\vert p\vert^2)^{\lambda-\frac{1}{2}}(\vert\nabla_q\bar G\vert^2+\vert q\cdot\nabla_q\bar G\vert^2){\rm d}q.\notag
\end{align}
By substituting the above estimates for $K_{i}$ $(10\leq i\leq14)$ into \eqref{6.17} and combining them with \eqref{6.16} and using the {\it a priori} assumption \eqref{g-aps}, we get
\begin{align}\label{6.23}
\frac{d}{dt}\int_{\mathbb R^6}&(e^{2\Phi}+\vert q\vert^2)^\lambda(\vert\partial_{q_i}g\vert^2+A\vert\partial_{x_i}g\vert^2){\rm d}x{\rm d}q\nonumber\\
	&+\int_{\mathbb R^6}e^{\bar\phi+2\Phi}(e^{2\Phi}+\vert q\vert^2)^{\lambda-\frac{1}{2}}(e^{2\Phi}\vert\nabla_q\partial_{q_i}g\vert^2+\vert q\cdot\nabla_q\partial_{q_i}g\vert^2){\rm d}x{\rm d}q\nonumber\\
	&+A\int_{\mathbb R^6}e^{\bar\phi+2\Phi}(e^{2\Phi}+\vert q\vert^2)^{\lambda-\frac{1}{2}}(e^{2\Phi}\vert\nabla_q\partial_{x_i} g\vert^2+\vert q\cdot\nabla_q\partial_{x_i}g\vert^2){\rm d}x{\rm d}q\nonumber\\
	\leq& \Big(2\bar\phi'(\frac{3}{2}-\lambda)+C(\partial_t\phi)_++Ce^{(\frac{3}{2}-\lambda)\bar\phi}+C\eps_1\Big)\int_{\mathbb R^6}(e^{2\Phi}+\vert q\vert^2)^\lambda(\vert\partial_{q_i}g\vert^2+A\vert\partial_{x_i}g\vert^2){\rm d}x{\rm d}q\nonumber\\
	&+C_\eta\sum_{\vert\beta'\vert\leq1} e^{\bar\phi}\left\|(\Phi,\nabla_x\Phi)\right\|_{L^2_x}^2\int_{\mathbb R^3}(1+\vert p\vert^2)^{\lambda-\frac{1}{2}}(\vert\nabla_q\partial_{q}^{\beta'}\bar G\vert^2+\vert q\cdot\nabla_q\partial_{q}^{\beta'}\bar G\vert^2){\rm d}q\nonumber\\
	&+\int_{\mathbb R^6}(e^{2\Phi}+\vert q\vert^2)^{\lambda}\vert\partial_{x_i} g\vert\vert\partial_{q_i} g\vert{\rm d}x{\rm d}q+C\tilde{\mathcal D}_{0,0}^\lambda(t)+Ce^{(\frac{3}{2}-\lambda)\bar\phi}\left\|(\Phi,\nabla_x\Phi,\nabla_x^2\Phi)\right\|_{L^2_x}^2,
\end{align}
where $A>0$ can be suitably large.
Note that if $t>T_c$, we have $-\bar\phi'(t)>-\frac{\bar\phi'(\infty)}{2}$. Then 
\begin{equation*}
	\int_{\mathbb R^6}(e^{2\Phi}+\vert q\vert^2)^{\gamma}\vert\partial_{x_i} g\vert\vert\partial_{q_i} g\vert{\rm d}x{\rm d}q
\end{equation*}
can be absorbed by
\begin{equation}
	-\bar\phi'(t)(\frac{3}{2}-\gamma)\int_{\mathbb R^6}(e^{2\Phi}+\vert q\vert^2)^\gamma(\vert\partial_{q_i}g\vert^2+A\vert\partial_{x_i}g\vert^2){\rm d}x{\rm d}q,\notag
\end{equation}
for large $t$ and large $A$.

Finally, applying $\partial_{x_i}$ to \eqref{g-eq-s3}$_2$, multiplying the resulting identity by $\partial_t\partial_{x_i}\Phi$, and integrating the resulting equality over $\mathbb R^3$ yield
\begin{align*}
\frac{1}{2}\frac{d}{dt}\int_{\mathbb R^3}&(\vert\partial_t\partial_{x_i}\Phi\vert^2+\vert\nabla_x\partial_{x_i}\Phi\vert^2){\rm d}x\\
	=&-\int_{\mathbb R^3}\partial_{x_i}\Big(e^{\bar\phi+2\Phi}\int_{\mathbb R^3}\frac{g}{\sqrt{e^{2\Phi}+\vert q\vert^2}}{\rm d}q\Big)\partial_t\partial_{x_i}\Phi{\rm d}x-\int_{\mathbb R^3}\int_{\mathbb R^3}\partial_{x_i}\Big(e^{\bar\phi}\frac{\bar G}{\sqrt{e^{2\Phi}+\vert q\vert^2}}\Big){\rm d}q\partial_t\partial_{x_i}\Phi{\rm d}x\\
:=&L_3+L_4.
\end{align*}
Direct calculation yields
\begin{equation*}\begin{split}
		|L_3|\leq e^{\bar\phi}\Big(\left\|\partial_t\partial_{x_i}\Phi\right\|_{L^2_x}^2+\left\|(e^{2\Phi}+\vert q\vert^2)^{\frac{\lambda}{2}}g\right\|_{L_{x,q}^2}^2+\left\|(e^{2\Phi}+\vert q\vert^2)^{\frac{\lambda}{2}}\partial_{x_i}g\right\|_{L_{x,q}^2}^2\Big),
	\end{split}
\end{equation*}
and
\begin{equation*}
	\begin{split}
		|L_4|&\leq e^{\bar\phi}\Big(\left\|\partial_t\partial_{x_i}\Phi\right\|_{L^2_x}^2+\left\|(1+\vert q\vert^2)^{\frac{\lambda}{2}}\bar G\right\|_{L_{q}^2}^2\left\|\partial_{x_i}\Phi\right\|_{L^2_x}^2\Big)\\
		&\leq e^{\bar\phi}(\left\|\partial_t\partial_{x_i}\Phi\right\|_{L^2_x}^2+\left\|\partial_{x_i}\Phi\right\|_{L^2_x}^2).
	\end{split}
\end{equation*}
Thus, we have the energy estimate for gravitational field $\Phi$
\begin{equation}\label{6.25}
	\frac{d}{dt}\int_{\mathbb R^3}(\vert\partial_t\partial_{x_i}\Phi\vert^2+\vert\nabla_x\partial_{x_i}\Phi\vert^2){\rm d}x\leq Ce^{\bar\phi}\sum_{m+n\leq1}\tilde{\mathcal E}_{m,n}^{\lambda}(t).
\end{equation}
Combining \eqref{6.23} and \eqref{6.25}, we have
\begin{align*}
\frac{d}{dt}&\big(\tilde{\mathcal E}_{0,1}^{\lambda}(t)+A\tilde{\mathcal E}_{1,0}^{\lambda}(t)\big)+\tilde{\mathcal D}_{0,1}^{\lambda}(t)+A\tilde{\mathcal D}_{1,0}^{\lambda}(t)\\
	\leq& \Big(2\bar\phi'(\frac{3}{2}-\lambda)+C(\partial_t\phi)_++Ce^{(\frac{3}{2}-\lambda)\bar\phi}
+C\eps_1\Big)\int_{\mathbb R^6}(e^{2\Phi}+\vert q\vert^2)^\lambda(\vert\partial_{q_i}g\vert^2+A\vert\partial_{x_i}g\vert^2){\rm d}x{\rm d}q\\
	&+Ce^{\bar\phi}\sum_{m+n\leq1}\tilde{\mathcal E}_{m,n}^{\lambda}(t)+\int_{\mathbb R^6}(e^{2\Phi}+\vert q\vert^2)^{\lambda}\vert\partial_{x_i} g\vert\vert\partial_{q_i} g\vert{\rm d}x{\rm d}q+C\tilde{\mathcal D}_{0,0}^{\gamma}(t)+Ce^{(\frac{3}{2}-\lambda)\bar\phi}\left\|\Phi\right\|_{L^2_x}^2\\
	&+C_\eta\sum_{\vert\beta'\vert\leq1} e^{\bar\phi}\left\|(\Phi,\partial_{x_i}\Phi)\right\|_{L^2_x}^2\int_{\mathbb R^3}(1+\vert p\vert^2)^{\lambda-\frac{1}{2}}(\vert\nabla_q\partial_{q}^{\beta'}\bar G\vert^2+\vert q\cdot\nabla_q\partial_{q}^{\beta'}\bar G\vert^2){\rm d}q,
\end{align*}
which, together with \eqref{6.11}, gives
\begin{equation}\label{6.26}
	\begin{split}
\tilde{\mathcal E}_{0,1}^{\lambda}(t)&+\tilde{\mathcal E}_{1,0}^{\lambda}(t)
+C\int_0^t(\tilde{\mathcal D}_{0,1}^{\lambda}(s)+\tilde{\mathcal D}_{1,0}^{\lambda}(s)){\rm d}s\\
		&\leq C(T_c)\sum_{m+n\leq1}\tilde{\mathcal E}_{m,n}^{\lambda}(0)+C(T_c)\left\|\Phi_0\right\|_{L^2_x}^2,	
	\end{split}
\end{equation}
and
\begin{equation}
	\begin{split}
e^{\frac{1}{2}\bar\phi(\lambda-\frac{3}{2})}&\sum_{\vert\alpha\vert+\vert\beta\vert=1}\int_{\mathbb R^6}(e^{2\Phi}+\vert q\vert^2)^\lambda\vert\partial_x^\alpha\partial_q^\beta g\vert^2{\rm d}x{\rm d}q+\int_0^t	e^{\frac{1}{2}\bar\phi(\lambda-\frac{3}{2})}(\bar{\tilde{\mathcal D}}_{0,1}^{\lambda}(s)+\bar{\tilde{\mathcal D}}_{1,0}^{\lambda}(s)){\rm d}s\\
		\leq& C(T_c)\sum_{m+n\leq1}\tilde{\mathcal E}_{m,n}^{\lambda}(0)+C(T_c)\left\|\Phi_0\right\|_{L^2_x}^2.\notag
	\end{split}
\end{equation}
Here, we apply a similar approach as for obtaining \eqref{z-lt}.

\noindent\underline{{\it Step 3. High-order energy estimates.}}
Let $1\leq\vert\alpha\vert+\vert\beta\vert\leq N-1$. $\partial_x^\alpha\partial_q^\beta g$ and $\partial_x^\alpha\Phi$ satisfy the following equations
\begin{align}\label{6.28}
\left\{
\begin{array}{rll}
&\partial_t\partial_x^\alpha\partial_q^\beta g-3\bar\phi'\partial_x^\alpha\partial_q^\beta g-\bar\phi' q\cdot\nabla_q\partial_x^\alpha\partial_q^\beta g+\nabla_q\sqrt{e^{2\Phi}+\vert q\vert^2}\cdot\nabla_x\partial_x^\alpha\partial_q^\beta g-\nabla_x\sqrt{e^{2\Phi}+\vert q\vert^2}\cdot\nabla_q\partial_x^\alpha\partial_q^\beta g\\[2mm]
&\qquad=\nabla_x\sqrt{e^{2\Phi}+\vert q\vert^2}\cdot\nabla_q\partial_x^\alpha\partial_q^\beta\bar G
	+e^{\bar\phi+2\Phi}\nabla_q\cdot(\Lambda_{\Phi,q}\nabla_q\partial_x^\alpha\partial_q^\beta g)\\[2mm]
&\qquad\qquad+(e^{\bar\phi+2\Phi}\nabla_q\cdot(\Lambda_{\Phi,q}\nabla_q\partial_x^\alpha\partial_q^\beta\bar G)-e^{\bar\phi}\nabla_q\cdot(\Lambda_{\bar\phi,q}\nabla_q\partial_x^\alpha\partial_q^\beta\bar G))+\mathcal R_1,\\[2mm]
		&\partial_t^2\partial_x^\alpha\Phi-\Delta_x\partial_x^\alpha\Phi=-e^{\bar\phi+2\Phi}\int_{\mathbb R^3}\frac{\partial_x^\alpha g}{\sqrt{e^{2\Phi}+\vert q\vert^2}}{\rm d}q-\int_{\mathbb R^3}\Big(\frac{e^{\bar\phi+2\Phi}\partial_x^\alpha\bar G}{\sqrt{e^{2\Phi}+\vert q\vert^2}}-\frac{e^{\bar\phi}\partial_x^\alpha\bar G}{\sqrt{1+\vert p\vert^2}}\Big){\rm d}q+\mathcal R_2,
	\end{array}\right.
\end{align}
with
\begin{align}\label{cr12}
	\left\{
	\begin{array}{rll}
		\mathcal R_1=&[-\bar\phi' q\cdot\nabla_q+\nabla_q\sqrt{e^{2\Phi}+\vert q\vert^2}\cdot\nabla_x-\nabla_x\sqrt{e^{2\Phi}+\vert q\vert^2}\cdot\nabla_q,\partial_x^\alpha\partial_q^\beta ]g\\[2mm]&+[\partial_x^\alpha\partial_q^\beta,\nabla_x\sqrt{e^{2\Phi}+\vert q\vert^2}\cdot\nabla_q]\bar G+\partial_x^\alpha\partial_q^\beta \big(e^{\bar\phi+2\Phi}\nabla_q\cdot(\Lambda_{\Phi,q}\nabla_q g)\big)- e^{\bar\phi+2\Phi}\nabla_q\cdot(\Lambda_{\Phi,q}\nabla_q\partial_x^\alpha\partial_q^\beta g)\\[2mm]
		&+\Big(\partial_x^\alpha\partial_q^\beta(e^{\bar\phi+2\Phi}\nabla_q\cdot(\Lambda_{\Phi,q}\nabla_q\bar G)-e^{\bar\phi}\nabla_q\cdot(\Lambda_{\bar\phi,q}\nabla_q\bar G))\\[2mm]
		&-(e^{\bar\phi+2\Phi}\nabla_q\cdot(\Lambda_{\Phi,q}\nabla_q\partial_x^\alpha\partial_q^\beta\bar G)-e^{\bar\phi}\nabla_q\cdot(\Lambda_{\bar\phi,q}\nabla_q\partial_x^\alpha\partial_q^\beta\bar G))\Big),\\[2mm]
		\mathcal R_2=&-\partial_x^\alpha\Big(e^{\bar\phi+2\Phi}\int_{\mathbb R^3}\frac{ g}{\sqrt{e^{2\Phi}+\vert q\vert^2}}{\rm d}q\Big)+e^{\bar\phi+2\Phi}\int_{\mathbb R^3}\frac{\partial_x^\alpha g}{\sqrt{e^{2\Phi}+\vert q\vert^2}}{\rm d}q\\[2mm]
		&-\Big(\partial_x^\alpha\int_{\mathbb R^3}\Big(\frac{e^{\bar\phi+2\Phi}\bar G}{\sqrt{e^{2\Phi}+\vert q\vert^2}}-\frac{e^{\bar\phi}\bar G}{\sqrt{1+\vert p\vert^2}}\Big){\rm d}q\Big)+\int_{\mathbb R^3}\Big(\frac{e^{\bar\phi+2\Phi}\partial_x^\alpha\bar G}{\sqrt{e^{2\Phi}+\vert q\vert^2}}-\frac{e^{\bar\phi}\partial_x^\alpha\bar G}{\sqrt{1+\vert p\vert^2}}\Big){\rm d}q.
	\end{array}
	\right.
\end{align}	
Similar to the derivation of equations \eqref{6.23} and \eqref{6.25}, we obtain from \eqref{6.28} that for sufficiently large $A>0$
\begin{align}\label{6.29}
\frac{d}{dt}\int_{\mathbb R^6}&(e^{2\Phi}+\vert q\vert^2)^\lambda(\vert\partial_{q_i}\partial_x^\alpha\partial_q^\beta g\vert^2+A\vert\partial_{x_i}\partial_x^\alpha\partial_q^\beta g\vert^2){\rm d}x{\rm d}q\nonumber\\
	&+\int_{\mathbb R^6}e^{\bar\phi+2\Phi}(e^{2\Phi}+\vert q\vert^2)^{\lambda-\frac{1}{2}}(e^{2\Phi}\vert\nabla_q\partial_{q_i}\partial_x^\alpha\partial_q^\beta g\vert^2+\vert q\cdot\nabla_q\partial_{q_i}\partial_x^\alpha\partial_q^\beta g\vert^2){\rm d}x{\rm d}q\nonumber\\
	&+A\int_{\mathbb R^6}e^{\bar\phi+2\Phi}(e^{2\Phi}+\vert q\vert^2)^{\lambda-\frac{1}{2}}(e^{2\Phi}\vert\nabla_q\partial_{x_i}\partial_x^\alpha\partial_q^\beta g\vert^2+\vert q\cdot\nabla_q\partial_{x_i}\partial_x^\alpha\partial_q^\beta g\vert^2){\rm d}x{\rm d}q\nonumber\\
	\leq& \Big(2\bar\phi'(\frac{3}{2}-\lambda)+C(\partial_t\phi)_++Ce^{(\frac{3}{2}-\lambda)\bar\phi}
+C\eps_1\Big)\int_{\mathbb R^6}(e^{2\Phi}+\vert q\vert^2)^\lambda(\vert\partial_{q_i}\partial_x^\alpha\partial_q^\beta g\vert^2+A\vert\partial_{x_i}\partial_x^\alpha\partial_q^\beta g\vert^2){\rm d}x{\rm d}q\nonumber\\
	&+\int_{\mathbb R^6}(e^{2\Phi}+\vert q\vert^2)^{\lambda}\vert\partial_{x_i}\partial_x^\alpha\partial_q^\beta g\vert\vert\partial_{q_i}\partial_x^\alpha\partial_q^\beta g\vert{\rm d}x{\rm d}q\nonumber\\
	&+C_\eta\sum_{0\leq\vert\beta'\vert\leq1} e^{\bar\phi}\left\|(\Phi,\partial_{x_i}\Phi)\right\|_{L^2_x}^2\int_{\mathbb R^3}(1+\vert p\vert^2)^{\lambda-\frac{1}{2}}(\vert\nabla_q\partial_{q_i}\partial_q^\beta\bar G\vert^2+\vert q\cdot\nabla_q\partial_{q_i}\partial_q^\beta\bar G\vert^2){\rm d}q\nonumber\\
	&+C\sum_{m=\vert\alpha\vert, n=\vert\beta\vert}\tilde{\mathcal D}_{m,n}^\lambda(t)+Ce^{(\frac{3}{2}-\lambda)\bar\phi}\left\|(\Phi,\nabla_x\Phi,\nabla_x^2\Phi)\right\|_{L^2_x}^2\nonumber\\
	&+\underbrace{\int_{\mathbb R^6}(e^{2\Phi}+\vert q\vert^2)^{\lambda}(\partial_{q_i}\mathcal R_1\partial_{q_i}\partial_x^\alpha\partial_q^\beta g+A\partial_{x_i}\mathcal R_1\partial_{x_i}\partial_x^\alpha\partial_q^\beta g){\rm d}x{\rm d}q}_{\CK_1},
\end{align}	
and
\begin{align}\label{6.30}
	\frac{d}{dt}\int_{\mathbb R^3}(\vert\partial_t\partial_{x_i}\partial_x^\alpha\Phi\vert^2+\vert\nabla_x\partial_{x_i}\partial_x^\alpha\Phi\vert^2){\rm d}x\leq Ce^{\bar\phi}\sum_{ m\leq\vert\alpha\vert+1}\tilde{\mathcal E}_{m,0}^{\lambda}(t)+\underbrace{\int_{\mathbb R^3}\partial_{x_i}\mathcal R_2\partial_{x_i}\partial_x^\alpha\partial_t\Phi{\rm d}x}_{\CK_2}.
\end{align}
We now turn to estimate $\CK_1$ and $\CK_2$. Recall the definition for $\CR_1$ in
\eqref{cr12}.
If $\vert\beta\vert=0$, then $[-\partial_t\bar\phi q\cdot\nabla_q,\partial_x^\alpha\partial_q^\beta]g=0$, and the term
\begin{equation*}
	\int_{\mathbb R^6}(e^{2\Phi}+\vert q\vert^2)^{\lambda}\Big(\partial_{q_i}[-\bar\phi' q\cdot\nabla_q,\partial_x^\alpha\partial_q^\beta]g\partial_x^\alpha\partial_q^\beta\partial_{q_i} g+A\partial_{x_i}[-\bar\phi' q\cdot\nabla_q,\partial_x^\alpha\partial_q^\beta]g\partial_x^\alpha\partial_q^\beta\partial_{x_i} g\Big){\rm d}x{\rm d}q
\end{equation*}
vanishes. If $\vert\beta\vert>0$, without loss of generality, we assume $\partial_q^\beta=\partial_{q_j}\partial_q^{\beta-e_j}$. Note that
\begin{equation*}
	[-\partial_t\bar\phi q\cdot\nabla_q,\partial_x^\alpha\partial_q^\beta]g=\bar\phi'(t)\partial_x^\alpha\partial_q^\beta g,	
\end{equation*}
then
\begin{align*}
\int_{\mathbb R^6}&(e^{2\Phi}+\vert q\vert^2)^{\lambda}\Big(\partial_{q_i}[-\partial_t\bar\phi q\cdot\nabla_q,\partial_x^\alpha\partial_q^\beta]g\partial_x^\alpha\partial_q^\beta\partial_{q_i} g+A\partial_{x_i}[-\partial_t\bar\phi q\cdot\nabla_q,\partial_x^\alpha\partial_q^\beta]g\partial_x^\alpha\partial_q^\beta\partial_{x_i} g{\rm d}x{\rm d}q\\
	=&\bar\phi'\int_{\mathbb R^6}(e^{2\Phi}+\vert q\vert^2)^{\lambda}(\vert\partial_{q_i}\partial_x^\alpha\partial_q^\beta g\vert^2+A\vert\partial_{x_i} \partial_x^\alpha\partial_q^\beta g\vert^2){\rm d}x{\rm d}q.
\end{align*}
Next, by using the smallness of  $x$-derivatives of $\sqrt{e^{2\Phi}+\vert q\vert^2}$, along with
\eqref{g-aps} and Cauchy-Schwarz's inequality, there exists a small parameter $\eta>0$ such that
\begin{align*}
\int_{\mathbb R^6}&(e^{2\Phi}+\vert q\vert^2)^{\lambda}\Big(\partial_{q_i}\left[\nabla_q\sqrt{e^{2\Phi}+\vert q\vert^2}\cdot\nabla_x-\nabla_x\sqrt{e^{2\Phi}+\vert q\vert^2}\cdot\nabla_q,\partial_x^\alpha\partial_q^\beta \right]g\partial_x^\alpha\partial_q^\beta\partial_{q_i} g\\
	&+A\partial_{x_i}\left[\nabla_q\sqrt{e^{2\Phi}+\vert q\vert^2}\cdot\nabla_x-\nabla_x\sqrt{e^{2\Phi}+\vert q\vert^2}\cdot\nabla_q,\partial_x^\alpha\partial_q^\beta\right ]g\partial_x^\alpha\partial_q^\beta\partial_{x_i} g\Big){\rm d}x{\rm d}q\\
	\leq& C\eps_1\sum_{\vert\alpha'\vert+\vert\beta'\vert\leq\vert\alpha\vert+\vert\beta\vert+1}\int_{\mathbb R^6}(e^{2\Phi}+\vert q\vert^2)^\lambda\vert\partial_x^{\alpha'}\partial_q^{\beta'}g\vert^2{\rm d}x{\rm d}q+\eta\int_{\mathbb R^6}(e^{2\Phi}+\vert q\vert^2)^\lambda\vert\partial_{p_i}\partial_x^\alpha\partial_q^\beta g\vert^2{\rm d}x{\rm d}q\\
	&+C_\eta\sum_{\vert\alpha'\vert+\vert\beta'\vert\leq\vert\alpha\vert+\vert\beta\vert}\int_{\mathbb R^6}(e^{2\Phi}+\vert q\vert^2)^\lambda\vert\partial_x^{\alpha'}\partial_q^{\beta'}g\vert^2{\rm d}x{\rm d}q\\
	&+|\beta|\sum_{|\beta'|=1}\sum_{|\alpha'|=1}\int_{\mathbb R^6}(e^{2\Phi}+\vert q\vert^2)^\lambda|\partial_x^{\alpha'}\partial_x^{\alpha}\partial_q^{\beta-\beta'}\partial_{q_i}g|
|\partial_{q_i}\partial_x^\alpha\partial_q^\beta g|{\rm d}x{\rm d}q.
\end{align*}
For the term involving the second component in $\CR_1$,  Lemma \ref{lemma A2} implies that
\begin{align*}
\int_{\mathbb R^6}&(e^{2\Phi}+\vert q\vert^2)^{\lambda}\Big(\partial_{q_i}[\partial_x^\alpha\partial_q^\beta,\nabla_x\sqrt{e^{2\Phi}+\vert q\vert^2}\cdot\nabla_q]\bar G\partial_{q_i}\partial_x^\alpha\partial_q^\beta g\\&\qquad\qquad+A\partial_{x_i}[\partial_x^\alpha\partial_q^\beta,\nabla_x\sqrt{e^{2\Phi}+\vert q\vert^2}\cdot\nabla_q]\bar G\partial_{x_i}\partial_x^\alpha\partial_q^\beta g\Big){\rm d}x{\rm d}q\\
	\leq& Ce^{(\frac{3}{2}-\lambda)\bar\phi}\int_{\mathbb R^6}(e^{2\Phi}+\vert q\vert^2)^{\lambda}(\vert\partial_{q_i}\partial_x^\alpha\partial_q^\beta g\vert^2+A\vert\partial_{x_i}\partial_x^\alpha\partial_q^\beta g\vert^2){\rm d}x{\rm d}q+Ce^{(\frac{3}{2}-\lambda)\bar\phi}\left\|\nabla_x\Phi\right\|_{H^{\vert\alpha\vert+1}_x}^2.	
\end{align*}
For terms involving diffusion, we get from Cauchy-Schwarz's inequality that 
\begin{align*}
\int_{\mathbb R^6}&\partial_{q_i}\Big(\partial_x^\alpha\partial_q^\beta \big(e^{\bar\phi+2\Phi}\nabla_q\cdot(\Lambda_{\Phi,q}\nabla_q g)\big)- e^{\bar\phi+2\Phi}\nabla_q\cdot(\Lambda_{\Phi,q}\nabla_q\partial_x^\alpha\partial_q^\beta g)\Big)(e^{2\Phi}+\vert q\vert^2)^{\lambda}\partial_{q_i}\partial_x^\alpha\partial_q^\beta g{\rm d}x{\rm d}q\\
	\leq& \eta\int_{\mathbb R^6}e^{\bar\phi+2\Phi}(e^{2\Phi}+\vert q\vert^2)^{\lambda-\frac{1}{2}}(e^{2\Phi}\vert\nabla_q\partial_{q_i}\partial_x^\alpha\partial_q^\beta g\vert^2+\vert q\cdot\nabla_q\partial_{q_i}\partial_x^\alpha\partial_q^\beta g\vert^2){\rm d}x{\rm d}q\\
	&+Ce^{\bar\phi}\int_{\mathbb R^6} \int_{\mathbb R^6}(e^{2\Phi}+\vert q\vert^2)^{\lambda}\vert\partial_{q_i}\partial_x^\alpha\partial_q^\beta g\vert^2{\rm d}x{\rm d}q+ C_\eta\sum_{\vert\alpha'\vert+\vert\beta'\vert\leq\vert\alpha\vert+\vert\beta\vert}\tilde{\mathcal D}^\lambda_{\vert\alpha'\vert,\vert\beta'\vert}(g,\Phi),
\end{align*}
\begin{align*}
\int_{\mathbb R^6}&\partial_{x_i}\Big(\partial_x^\alpha\partial_q^\beta \big(e^{\bar\phi+2\Phi}\nabla_q\cdot(\Lambda_{\Phi,q}\nabla_q g)\big)- e^{\bar\phi+2\Phi}\nabla_q\cdot(\Lambda_{\Phi,q}\nabla_q\partial_x^\alpha\partial_q^\beta g)\Big)(e^{2\Phi}+\vert q\vert^2)^{\lambda}\partial_{x_i}\partial_x^\alpha\partial_q^\beta g{\rm d}x{\rm d}q\\
	\leq& \eta\int_{\mathbb R^6}e^{\bar\phi+2\Phi}(e^{2\Phi}+\vert q\vert^2)^{\lambda-\frac{1}{2}}(e^{2\Phi}\vert\nabla_q\partial_{x_i}\partial_x^\alpha\partial_q^\beta g\vert^2+\vert q\cdot\nabla_q\partial_{x_i}\partial_x^\alpha\partial_q^\beta g\vert^2){\rm d}x{\rm d}q\\
	&+Ce^{\bar\phi}\int_{\mathbb R^6}(e^{2\Phi}+\vert q\vert^2)^{\lambda}\vert\partial_{x_i}\partial_x^\alpha\partial_q^\beta g\vert^2{\rm d}x{\rm d}q+ C_\eta\sum_{\vert\alpha'\vert+\vert\beta'\vert\leq\vert\alpha\vert+\vert\beta\vert}\tilde{\mathcal D}_{\vert\alpha'\vert,\vert\beta'\vert}^\la(t),
\end{align*}
and
\begin{align*}
\int_{\mathbb R^6}&\partial_{q_i}\Big(\partial_x^\alpha\partial_q^\beta(e^{\bar\phi+2\Phi}\nabla_q\cdot(\Lambda_{\Phi,q}\nabla_q\bar G)-e^{\bar\phi}\nabla_q\cdot(\Lambda_{\bar\phi,q}\nabla_q\bar G))\\
	&-(e^{\bar\phi+2\Phi}\nabla_q\cdot(\Lambda_{\Phi,q}\nabla_q\partial_x^\alpha\partial_q^\beta\bar G)-e^{\bar\phi}\nabla_q\cdot(\Lambda_{\bar\phi,q}\nabla_q\partial_x^\alpha\partial_q^\beta\bar G))\Big)(e^{2\Phi}+\vert q\vert^2)^{\lambda}	\partial_{q_i}\partial_x^\alpha\partial_q^\beta g{\rm d}x{\rm d}q\\
	&+A\int_{\mathbb R^6}\partial_{x_i}\Big(\partial_x^\alpha\partial_q^\beta(e^{\bar\phi+2\Phi}\nabla_q\cdot(\Lambda_{\Phi,q}\nabla_q\bar G)-e^{\bar\phi}\nabla_q\cdot(\Lambda_{\bar\phi,q}\nabla_q\bar G))\\
	&-(e^{\bar\phi+2\Phi}\nabla_q\cdot(\Lambda_{\Phi,q}\nabla_q\partial_x^\alpha\partial_q^\beta\bar G)-e^{\bar\phi}\nabla_q\cdot(\Lambda_{\bar\phi,q}\nabla_q\partial_x^\alpha\partial_q^\beta\bar G))\Big)(e^{2\Phi}+\vert q\vert^2)^{\lambda}	\partial_{x_i}\partial_x^\alpha\partial_q^\beta g{\rm d}x{\rm d}q\\
	\leq& \eta\int_{\mathbb R^6}e^{\bar\phi+2\Phi}(e^{2\Phi}+\vert q\vert^2)^{\lambda-\frac{1}{2}}(e^{2\Phi}\vert\nabla_q\partial_{q_i}\partial_x^\alpha\partial_q^\beta g\vert^2+\vert q\cdot\nabla_q\partial_{q_i}\partial_x^\alpha\partial_q^\beta g\vert^2){\rm d}x{\rm d}q\\
	&+\eta\int_{\mathbb R^6}e^{\bar\phi+2\Phi}(e^{2\Phi}+\vert q\vert^2)^{\lambda-\frac{1}{2}}(e^{2\Phi}\vert\nabla_q\partial_{x_i}\partial_x^\alpha\partial_q^\beta g\vert^2+\vert q\cdot\nabla_q\partial_{x_i}\partial_x^\alpha\partial_q^\beta g\vert^2){\rm d}x{\rm d}q\\
	&+Ce^{\bar\phi}\int_{\mathbb R^3}(e^{2\Phi}+\vert q\vert^2)^\gamma(\vert\nabla_p\partial_x^\alpha\partial_q^\beta g\vert^2+A\vert\nabla_x\partial_x^\alpha\partial_q^\beta g\vert^2){\rm d}x{\rm d}q\\
	&+C_\eta\sum_{\vert\beta'\vert\leq\vert\beta\vert+1} e^{\bar\phi}\left\|\Phi\right\|_{H^{\vert\alpha\vert+1}_x}^2\int_{\mathbb R^3}(1+\vert q\vert^2)^{\lambda-\frac{1}{2}}(\vert\nabla_q\partial^{\beta'}_{q}\bar G\vert^2+\vert q\cdot\nabla_q\partial_{q}^{\beta'}\bar G\vert^2){\rm d}q.
\end{align*}
By combining the above estimates together, we conclude that for $\vert\beta\vert>0$
\begin{align}\label{6.31}
	\CK_1\leq& \eta\int_{\mathbb R^6}e^{\bar\phi+2\Phi}(e^{2\Phi}+\vert q\vert^2)^{\lambda-\frac{1}{2}}(e^{2\Phi}\vert\nabla_q\partial_{q_i}\partial_x^\alpha\partial_q^\beta g\vert^2+\vert q\cdot\nabla_q\partial_{q_i}\partial_x^\alpha\partial_q^\beta g\vert^2){\rm d}x{\rm d}q\nonumber\\
	&+\eta\int_{\mathbb R^6}e^{\bar\phi+2\Phi}(e^{2\Phi}+\vert q\vert^2)^{\lambda-\frac{1}{2}}(e^{2\Phi}\vert\nabla_q\partial_{x_i}\partial_x^\alpha\partial_q^\beta g\vert^2+\vert q\cdot\nabla_q\partial_{x_i}\partial_x^\alpha\partial_q^\beta g\vert^2){\rm d}x{\rm d}q\nonumber\\
	&+\eta\int_{\mathbb R^6}(e^{2\Phi}+\vert q\vert^2)^\lambda\vert\partial_{p_i}\partial_x^\alpha\partial_q^\beta g\vert^2{\rm d}x{\rm d}q+C_\eta\sum_{\vert\alpha'\vert+\vert\beta'\vert\leq\vert\alpha\vert+\vert\beta\vert}\int_{\mathbb R^6}(e^{2\Phi}+\vert q\vert^2)^\lambda\vert\partial_x^{\alpha'}\partial_q^{\beta'}g\vert^2{\rm d}x{\rm d}q\nonumber\\
	&+|\beta|\sum_{|\beta'|=1}\sum_{|\alpha'|=1}\int_{\mathbb R^6}(e^{2\Phi}+\vert q\vert^2)^\lambda|\partial_x^{\alpha'}\partial_x^{\alpha}\partial_q^{\beta-\beta'}\partial_{q_i}g||\partial_{q_i}\partial_x^\alpha\partial_q^\beta g|{\rm d}x{\rm d}q\notag\\
	&+\bar\phi'\int_{\mathbb R^6}(e^{2\Phi}+\vert q\vert^2)^{\lambda}(\vert\partial_{q_i}\partial_x^\alpha\partial_q^\beta g\vert^2+A\vert\partial_{x_i} \partial_x^\alpha\partial_q^\beta g\vert^2){\rm d}x{\rm d}q\nonumber\\
	&\quad+C\eps_1\sum_{\vert\alpha'\vert+\vert\beta'\vert\leq\vert\alpha\vert+\vert\beta\vert+1}\int_{\mathbb R^6}(e^{2\Phi}+\vert q\vert^2)^\gamma\vert\partial_x^{\alpha'}\partial_q^{\beta'}g\vert^2{\rm d}x{\rm d}q\nonumber\\
	&+Ce^{(\frac{3}{2}-\lambda)\bar\phi}\int_{\mathbb R^6}(e^{2\Phi}+\vert q\vert^2)^{\lambda}(\vert\partial_{q_i}\partial_x^\alpha\partial_q^\beta g\vert^2+A\vert\partial_{x_i}\partial_x^\alpha\partial_q^\beta g\vert^2){\rm d}x{\rm d}q+Ce^{(\frac{3}{2}-\lambda)\bar\phi}\left\|\nabla_x\Phi\right\|_{H^{\vert\alpha\vert+1}_x}^2\nonumber\\
	&+C_\eta\sum_{\vert\beta'\vert\leq\vert\beta\vert+1} e^{\bar\phi}\left\|\Phi\right\|_{H^{\vert\alpha\vert+1}_x}^2\int_{\mathbb R^3}(1+\vert q\vert^2)^{\lambda-\frac{1}{2}}(\vert\nabla_q\partial^{\beta'}_{q}\bar G\vert^2+\vert q\cdot\nabla_q\partial_{q}^{\beta'}\bar G\vert^2){\rm d}q,
\end{align}
 and for $\vert\beta\vert=0$,
\begin{align}\label{6.32}
	|\CK_1|\leq& \eta\int_{\mathbb R^6}e^{\bar\phi+2\Phi}(e^{2\Phi}+\vert q\vert^2)^{\lambda-\frac{1}{2}}(e^{2\Phi}\vert\nabla_q\partial_{q_i}\partial_x^\alpha g\vert^2+\vert q\cdot\nabla_q\partial_{q_i}\partial_x^\alpha g\vert^2){\rm d}x{\rm d}q\nonumber\\
	&+\eta\int_{\mathbb R^6}e^{\bar\phi+2\Phi}(e^{2\Phi}+\vert q\vert^2)^{\lambda-\frac{1}{2}}(e^{2\Phi}\vert\nabla_q\partial_{x_i}\partial_x^\alpha g\vert^2+\vert q\cdot\nabla_q\partial_{x_i}\partial_x^\alpha g\vert^2){\rm d}x{\rm d}q\nonumber\\
	&+\eta\int_{\mathbb R^6}(e^{2\Phi}+\vert q\vert^2)^\lambda\vert\partial_{p_i}\partial_x^\alpha g\vert^2{\rm d}x{\rm d}q+C_\eta\sum_{\vert\alpha'\vert\leq\vert\alpha\vert}\int_{\mathbb R^6}(e^{2\Phi}+\vert q\vert^2)^\lambda\vert\partial_x^{\alpha'}g\vert^2{\rm d}x{\rm d}q\nonumber\\
	&+C\eps_1\sum_{\vert\alpha'\vert+\vert\beta'\vert\leq\vert\alpha\vert+1}\int_{\mathbb R^6}(e^{2\Phi}+\vert q\vert^2)^\gamma\vert\partial_x^{\alpha'}\partial_q^{\beta'}g\vert^2{\rm d}x{\rm d}q\nonumber\\
	&+Ce^{(\frac{3}{2}-\lambda)\bar\phi}\int_{\mathbb R^6}(e^{2\Phi}+\vert q\vert^2)^{\lambda}(\vert\partial_{q_i}\partial_x^\alpha g\vert^2+A\vert\partial_{x_i}\partial_x^\alpha g\vert^2){\rm d}x{\rm d}q+Ce^{(\frac{3}{2}-\lambda)\bar\phi}\left\|\nabla_x\Phi\right\|_{H^{\vert\alpha\vert+1}_x}^2\nonumber\\
	&+C_\eta\sum_{\vert\beta'\vert\leq1} e^{\bar\phi}\left\|\Phi\right\|_{H^{\vert\alpha\vert+1}_x}^2\int_{\mathbb R^3}(1+\vert q\vert^2)^{\lambda-\frac{1}{2}}(\vert\nabla_q\partial^{\beta'}_{q}\bar G\vert^2+\vert q\cdot\nabla_q\partial_{q}^{\beta'}\bar G\vert^2){\rm d}q.
\end{align}
As for $\CK_2$,
it is straightforward to show
\begin{align}\label{6.33}
	|\CK_2|\leq Ce^{\bar\phi}\sum_{ m\leq\vert\alpha\vert+1}\tilde{\mathcal E}_{m,0}^{\lambda}(t).
\end{align}

Now, plugging either \eqref{6.31} or \eqref{6.32} into \eqref{6.29}, and substituting \eqref{6.33} into \eqref{6.30}, we obtain for $\vert\beta\vert>0$
\begin{align}
\frac{d}{dt}\int_{\mathbb R^6}&(e^{2\Phi}+\vert q\vert^2)^\lambda(\vert\partial_{q_i}\partial_x^\alpha\partial_q^\beta g\vert^2+A\vert\partial_{x_i}\partial_x^\alpha\partial_q^\beta g\vert^2){\rm d}x{\rm d}q\nonumber\\
	&+\int_{\mathbb R^6}e^{\bar\phi+2\Phi}(e^{2\Phi}+\vert q\vert^2)^{\lambda-\frac{1}{2}}(e^{2\Phi}\vert\nabla_q\partial_{q_i}\partial_x^\alpha\partial_q^\beta g\vert^2+\vert q\cdot\nabla_q\partial_{q_i}\partial_x^\alpha\partial_q^\beta g\vert^2){\rm d}x{\rm d}q\nonumber\\
	&+A\int_{\mathbb R^6}e^{\bar\phi+2\Phi}(e^{2\Phi}+\vert q\vert^2)^{\lambda-\frac{1}{2}}(e^{2\Phi}\vert\nabla_q\partial_{x_i}\partial_x^\alpha\partial_q^\beta g\vert^2+\vert q\cdot\nabla_q\partial_{x_i}\partial_x^\alpha\partial_q^\beta g\vert^2){\rm d}x{\rm d}q\nonumber\\
\leq& \Big(2\bar\phi'(\frac{5}{2}-\lambda)+C(\partial_t\phi)_++Ce^{(\frac{3}{2}-\lambda)\bar\phi}+C\eps_1\Big)\notag\\
&\quad\times\int_{\mathbb R^6}(e^{2\Phi}+\vert q\vert^2)^\lambda(\vert\partial_{q_i}\partial_x^\alpha\partial_q^\beta g\vert^2+A\vert\partial_{x_i}\partial_x^\alpha\partial_q^\beta g\vert^2){\rm d}x{\rm d}q\nonumber\\
	&	+(1+|\beta|)\sum_{|\beta'|=1}\sum_{|\alpha'|=1}\int_{\mathbb R^6}(e^{2\Phi}+\vert q\vert^2)^\lambda|\partial_x^{\alpha'}\partial_x^{\alpha}\partial_q^{\beta-\beta'}\partial_{q_i}g||\partial_{q_i}\partial_x^\alpha\partial_q^\beta g|{\rm d}x{\rm d}q\nonumber\\
		&+\eta\int_{\mathbb R^6}(e^{2\Phi}+\vert q\vert^2)^\lambda\vert\partial_{p_i}\partial_x^\alpha\partial_q^\beta g\vert^2{\rm d}x{\rm d}q+C_\eta\sum_{\vert\alpha'\vert+\vert\beta'\vert\leq\vert\alpha\vert+\vert\beta\vert}\int_{\mathbb R^6}(e^{2\Phi}+\vert q\vert^2)^\lambda\vert\partial_x^{\alpha'}\partial_q^{\beta'}g\vert^2{\rm d}x{\rm d}q\nonumber\\
	&+C\sum_{ m+n\leq\vert\alpha\vert+\vert\beta\vert}\tilde{\mathcal D}_{m,n}^\lambda(t)^2+Ce^{(\frac{3}{2}-\lambda)\bar\phi}(\left\|\Phi\right\|_{L^2_x}^2
+\left\|\nabla_x\Phi\right\|_{H^{\vert\alpha\vert+1}_x}^2)\notag\\
	&+C\eps_1\sum_{\vert\alpha'\vert+\vert\beta'\vert\leq\vert\alpha\vert+\vert\beta\vert+1}\int_{\mathbb R^6}(e^{2\Phi}+\vert q\vert^2)^\lambda\vert\partial_x^{\alpha'}\partial_q^{\beta'}g\vert^2{\rm d}x{\rm d}q\notag\\
	&+Ce^{(\frac{5}{2}-\gamma)\bar\phi}\int_{\mathbb R^6}(e^{2\Phi}+\vert q\vert^2)^{\lambda}(\vert\partial_{q_i}\partial_x^\alpha\partial_q^\beta g\vert^2+A\vert\partial_{x_i}\partial_x^\alpha\partial_q^\beta g\vert^2){\rm d}x{\rm d}q\notag\\
	&+C_\eta\sum_{\vert\beta'\vert\leq\vert\beta\vert+1} e^{\bar\phi}\left\|\Phi\right\|_{H^{\vert\alpha\vert+1}_x}^2\int_{\mathbb R^3}(1+\vert q\vert^2)^{\lambda-\frac{1}{2}}(\vert\nabla_q\partial^{\beta'}_{q}\bar G\vert^2
+\vert q\cdot\nabla_q\partial_{q}^{\beta'}\bar G\vert^2){\rm d}q,\label{hg1}
\end{align}
 and for $\vert\beta\vert=0$,
\begin{align}
\frac{d}{dt}&\int_{\mathbb R^6}(e^{2\Phi}+\vert q\vert^2)^\lambda(\vert\partial_{q_i}\partial_x^\alpha\partial_q^\beta g\vert^2+A\vert\partial_{x_i}\partial_x^\alpha\partial_q^\beta g\vert^2){\rm d}x{\rm d}q\nonumber\\
	&+\int_{\mathbb R^6}e^{\bar\phi+2\Phi}(e^{2\Phi}+\vert q\vert^2)^{\lambda-\frac{1}{2}}(e^{2\Phi}\vert\nabla_q\partial_{q_i}\partial_x^\alpha\partial_q^\beta g\vert^2+\vert q\cdot\nabla_q\partial_{q_i}\partial_x^\alpha\partial_q^\beta g\vert^2){\rm d}x{\rm d}q\nonumber\\
	&+A\int_{\mathbb R^6}e^{\bar\phi+2\Phi}(e^{2\Phi}+\vert q\vert^2)^{\lambda-\frac{1}{2}}(e^{2\Phi}\vert\nabla_q\partial_{x_i}\partial_x^\alpha\partial_q^\beta g\vert^2+\vert q\cdot\nabla_q\partial_{x_i}\partial_x^\alpha\partial_q^\beta g\vert^2){\rm d}x{\rm d}q\nonumber\\
	\leq& \Big(2\partial_t\bar\phi(\frac{3}{2}-\lambda)+C(\partial_t\phi)_++Ce^{(\frac{3}{2}-\lambda)\bar\phi}+C\eps_1\Big)\notag\\
	&\quad\times\int_{\mathbb R^6}(e^{2\Phi}+\vert q\vert^2)^\lambda(\vert\partial_{q_i}\partial_x^\alpha\partial_q^\beta g\vert^2+A\vert\partial_{x_i}\partial_x^\alpha\partial_q^\beta g\vert^2){\rm d}x{\rm d}q\nonumber\\
	&	+\int_{\mathbb R^6}(e^{2\Phi}+\vert q\vert^2)^{\lambda}\vert\partial_{x_i}\partial_x^\alpha\partial_q^\beta g\vert\vert\partial_{q_i}\partial_x^\alpha\partial_q^\beta g\vert{\rm d}x{\rm d}q\nonumber\\
	&+\eta\int_{\mathbb R^6}(e^{2\Phi}+\vert q\vert^2)^\lambda\vert\partial_{p_i}\partial_x^\alpha\partial_q^\beta g\vert^2{\rm d}x{\rm d}q+C_\eta\sum_{\vert\alpha'\vert+\vert\beta'\vert\leq\vert\alpha\vert+\vert\beta\vert}\int_{\mathbb R^6}(e^{2\Phi}+\vert q\vert^2)^\lambda\vert\partial_x^{\alpha'}\partial_q^{\beta'}g\vert^2{\rm d}x{\rm d}q\nonumber\\
	&+C\sum_{ m+n\leq\vert\alpha\vert+\vert\beta\vert}\tilde{\mathcal D}_{m,n}^\lambda(t)^2+Ce^{(\frac{3}{2}-\lambda)\bar\phi}(\left\|\Phi\right\|_{L^2_x}^2
+\left\|\nabla_x\Phi\right\|_{H^{\vert\alpha\vert+1}_x}^2)\notag\\
	&+C\tilde{\mathcal E}^\lambda(t)\sum_{\vert\alpha'\vert+\vert\beta'\vert\leq\vert\alpha\vert+\vert\beta\vert+1}\int_{\mathbb R^6}(e^{2\Phi}+\vert q\vert^2)^\lambda\vert\partial_x^{\alpha'}\partial_q^{\beta'}g\vert^2{\rm d}x{\rm d}q\notag\\
	&+Ce^{(\frac{3}{2}-\gamma)\bar\phi}\int_{\mathbb R^6}(e^{2\Phi}+\vert q\vert^2)^{\lambda}(\vert\partial_{q_i}\partial_x^\alpha\partial_q^\beta g\vert^2+A\vert\partial_{x_i}\partial_x^\alpha\partial_q^\beta g\vert^2){\rm d}x{\rm d}q\notag\\
	&+C_\eta\sum_{\vert\beta'\vert\leq\vert\beta\vert+1} e^{\bar\phi}\left\|\Phi\right\|_{H^{\vert\alpha\vert+1}_x}^2\int_{\mathbb R^3}(1+\vert q\vert^2)^{\lambda-\frac{1}{2}}(\vert\nabla_q\partial^{\beta'}_{q}\bar G\vert^2
+\vert q\cdot\nabla_q\partial_{q}^{\beta'}\bar G\vert^2){\rm d}q,	\label{hg2}
\end{align}
 and
\begin{align}
	\frac{d}{dt}\int_{\mathbb R^3}(\vert\partial_t\partial_{x_i}\partial_x^\alpha\Phi\vert^2+\vert\nabla_x\partial_{x_i}\partial_x^\alpha\Phi\vert^2){\rm d}x\leq Ce^{\bar\phi}\sum_{ m\leq\vert\alpha\vert+1}\tilde{\mathcal E}_{m,0}^{\lambda}(t).\label{hg3}
\end{align}
Finally, by applying Gr\"{o}nwall's inequality to \eqref{hg1}, \eqref{hg2} and \eqref{hg3} and taking a linear combination  with \eqref{6.11} and \eqref{6.26},
we get \eqref{eng-es-ss} and \eqref{eng-de-ss}. This completes the proof of Theorem \ref{thm2}.

\end{proof}

\section{Appendix}\label{ad}
In this Appendix, we collect some  estimates used in the previous sections. The first lemma is about some  properties of the spatially homogeneous solution $(\bar F,\bar\phi)$ of \eqref{vnpf-ho} and \eqref{ho-id} constructed in \cite{ACP-2014}.

\begin{lemma}\label{lemma A1}
Assume $(\bar F,\bar\phi)$ is the spatially homogeneous solution of $(\bar F,\bar\phi)$ of \eqref{vnpf-ho} and \eqref{ho-id} given  in \cite[Theorem 2.1, pp.3703]{ACP-2014}. For any  $\gamma>0$, if the initial data satisfies
\begin{equation*}
	\int_{\mathbb R^3}\big((1+\vert p\vert^2)^{\gamma+1}\vert\nabla_p^3\bar F_{in}\vert^2+(1+\vert p\vert^2)^{\gamma+2}\vert\nabla_p^4\bar F_{in}\vert^2\big){\rm d}p<\infty,
\end{equation*}
then
\begin{equation}\label{A1}
\int_{\mathbb R^3}(e^{2\bar\phi}+\vert p\vert^2)^{\gamma}\vert\bar F\vert^2{\rm d}p<\infty,
\end{equation}
and
\begin{equation}\label{A2}
	\int_{\mathbb R^3}(e^{2\bar\phi}+\vert p\vert^2)^{\gamma_k}\vert\nabla^k_p\bar F\vert^2{\rm d}p<C(1+t)^{k-1},\ \textrm{for}\ k=1,2,3,4.
\end{equation}
Here, $\ga_1=\ga_2=\ga$, $\ga_3=\ga+1$ and $\ga_4=\ga+2$.
Moreover, it holds that
\begin{align}
\lim\limits_{t\rightarrow+\infty}\frac{d\bar{\phi}}{dt}\ \textrm{exists and is negative}.\notag
\end{align}

\end{lemma}
\begin{proof}
	We only need to prove \eqref{A2} for $k\ge 3$, since the other statements were  proved in \cite{ACP-2014}. Note that $(\bar F,\bar\phi)$ satisfies
	\begin{equation}\label{A6}
\left\{
\begin{array}{l}
\partial_t\bar F=e^{2\bar\phi}\nabla_p\cdot(\Lambda_{\bar\phi}\nabla_p\bar F),\\
\frac{d^2\bar{\phi}}{dt^2}=-e^{2\bar\phi}\int_{\mathbb R^3}\frac{\bar F}{\sqrt{e^{2\bar\phi}+\vert p\vert^2}}{\rm d}p.
\end{array}\right.
	\end{equation}
For $\vert\beta\vert=3$, applying $\partial_p^\beta$ to $\eqref{A6}_1$, multiplying the resulting identity by $(e^{2\bar\phi}+\vert p\vert^2)^{\gamma+1}\partial_p^\beta\bar F$, then the resulting equality yields that
\begin{align*}
\frac{1}{2}\frac{d}{dt}\int_{\mathbb R^3}&(e^{2\bar\phi}+\vert p\vert^2)^{\gamma+1}\vert\partial_p^\beta\bar F\vert^2{\rm d}p\nonumber\\
=&(\gamma+1)\int_{\mathbb R^3}\bar\phi'(t)(e^{2\bar\phi}+\vert p\vert^2)^{\gamma}\vert\partial_p^\beta\bar F\vert^2{\rm d}p+\int_{\mathbb R^3}e^{2\bar\phi}\partial_p^\beta\nabla_p\cdot(\Lambda_{\bar\phi}\nabla_p\bar F)(e^{2\bar\phi}+\vert p\vert^2)^{\gamma+1}\partial_p^\beta\bar F{\rm d}p\nonumber\\
\leq& (\gamma+1)\int_{\mathbb R^3}\bar\phi'(t)(e^{2\bar\phi}+\vert p\vert^2)^{\gamma}\vert\partial_p^\beta\bar F\vert^2{\rm d}p-\int_{\mathbb R^3}e^{2\bar\phi}(e^{2\bar\phi}+\vert p\vert^2)^{\gamma-\frac{1}{2}}(e^{2\bar\phi}\vert\nabla_p\partial_p^\beta\bar F\vert^2+\vert p\cdot\nabla_p\partial_p^\beta\bar F\vert){\rm d}p\nonumber\\
&+\int_{\mathbb R^3}e^{\bar\phi}(e^{2\bar\phi}+\vert p\vert^2)^{\frac{\gamma}{2}}\vert p\cdot\nabla_p\partial_p^\beta\bar F\vert (e^{2\bar\phi}+\vert p\vert^2)^{\frac{\gamma}{2}+\frac{1}{2}}\vert\partial_p^\beta\bar F\vert{\rm d}p\nonumber\\
&+\int_{\mathbb R^3}e^{2\bar\phi}\Big((e^{2\bar\phi}+\vert p\vert^2)^{-1}\vert\nabla_p\bar F\vert\vert\nabla_p^2\bar F\vert+(e^{2\bar\phi}+\vert p\vert^2)^{-\frac{1}{2}\vert\nabla_p\bar F\vert^2}\Big)(e^{2\bar\phi}+\vert p\vert^2)^{\gamma+1}{\rm d}p\nonumber\\
&+\int_{\mathbb R^3}e^{2\bar\phi}\Big((e^{2\bar\phi}+\vert p\vert^2)^{-1}\vert\nabla_p^2\bar F\vert\vert\nabla_p^2\bar F\vert+(e^{2\bar\phi}+\vert p\vert^2)^{-\frac{1}{2}}\vert\nabla^3_p\bar F\vert\Big)(e^{2\bar\phi}+\vert p\vert^2)^{\gamma+1}\vert\partial_p^\beta\bar F\vert{\rm d}p\nonumber\\
\leq&(\gamma+1)\int_{\mathbb R^3}\bar\phi'(t)(e^{2\bar\phi}+\vert p\vert^2)^{\gamma}\vert\partial_p^\beta\bar F\vert^2{\rm d}p\notag\\
&-\int_{\mathbb R^3}e^{2\bar\phi}(e^{2\bar\phi}+\vert p\vert^2)^{\gamma-\frac{1}{2}}(e^{2\bar\phi}\vert\nabla_p\partial_p^\beta\bar F\vert^2+\vert p\cdot\nabla_p\partial_p^\beta\bar F\vert){\rm d}p\nonumber\\
&+e^{\bar\phi}\int_{\mathbb R^3}(e^{2\bar\phi}+\vert p\vert^2)^{\gamma+1}\vert\partial_p^\beta\bar F\vert^2{\rm d}p+\int_{\mathbb R^3}(e^{2\bar\phi}+\vert p\vert^2)^{\gamma}(\vert\nabla_p\bar F\vert^2+\vert\nabla_p^2\bar F\vert^2){\rm d}p.	
	\end{align*}
By using \eqref{A1} and \eqref{A2} with $k=1,2$, one has
\begin{equation}
	\begin{split}
\frac{d}{dt}&\int_{\mathbb R^3}(e^{2\bar\phi}+\vert p\vert^2)^{\gamma+1}\vert\partial_p^\beta\bar F\vert^2{\rm d}p\\
&\leq	(\gamma+1)\int_{\mathbb R^3}\bar\phi'(t)(e^{2\bar\phi}+\vert p\vert^2)^{\gamma}\vert\partial_p^\beta\bar F\vert^2{\rm d}p+e^{\bar\phi}\int_{\mathbb R^3}(e^{2\bar\phi}+\vert p\vert^2)^{\gamma+1}\vert\partial_p^\beta\bar F\vert^2{\rm d}p+C(1+t).\notag
\end{split}
\end{equation}
The Gr\"{o}nwall's inequality then gives
\begin{equation}
\int_{\mathbb R^3}(e^{2\bar\phi}+\vert p\vert^2)^{\gamma+1}\vert\partial_p^\beta\bar F\vert^2{\rm d}p\leq C(1+t)^2.\notag
\end{equation}
This shows \eqref{A2} holds for $k= 3$. The case for $k=4$ is similar.
This completes the proof of Lemma \ref{lemma A1}.
\end{proof}

We now proceed to derive the corresponding estimates for the spatially homogeneous solution $\bar G(t,q)$, which represents the self-similar form of $\bar F(t,q)$.
Note that
\begin{equation*}
\bar F(t,p)=e^{-3\bar\phi}\bar G(t,e^{-\bar\phi}p)=:e^{-3\bar\phi}\bar G(t,q),
\end{equation*}
and
$(\bar G(t,q),\bar \phi(t))$ satisfies
\begin{equation}\label{A8}
	\left\{
	\begin{array}{l}
		\partial_t\bar G-3\bar\phi' \bar G-\bar\phi' q\cdot\nabla_q\bar G=e^{\bar\phi}\nabla_q\cdot(\Lambda_{0,q}\nabla_q\bar G),\\
		\frac{d^2\bar\phi}{dt^2}=-e^{\bar\phi}\int_{\mathbb R^3}\frac{\bar G}{\sqrt{1+\vert q\vert^2}}{\rm d}q,
	\end{array}\right.
\end{equation}
with
\begin{equation}\label{A9}
\bar G(0,q)=\bar G_0(q),\ \bar\phi(0)=\phi_{in},\ \partial_t\bar\phi(0)=\psi_{in}.
\end{equation}
The estimates on $\bar{G}$ are given in the following lemma.

\begin{lemma}\label{lemma A2}
Let $\frac{1}{2}<\lambda<\frac{3}{2}$. Assume  the initial data satisfies
	\begin{equation*}
		\sum_{\vert\beta\vert\leq N+1}\int_{\mathbb R^3}(1+\vert q\vert^2)^{\lambda}\vert\partial_q^{\beta} \bar G_0\vert^2{\rm d}p<\infty.
		\end{equation*}
Then the solution to \eqref{A8} and \eqref{A9} satisfies
	\begin{equation}\label{A10}
		\sum_{\vert\beta\vert\leq N+1}\int_{\mathbb R^3}(1+\vert q\vert^2)^{\lambda}\vert\partial_q^{\beta} \bar G\vert^2{\rm d}p\leq Ce^{2(\frac{3}{2}-\lambda)\bar\phi},
	\end{equation}
and
	\begin{equation}\label{A11}
		\sum_{\vert\beta\vert\leq N+1}\int_0^te^{(2\gamma-2)\bar\phi}\int_{\mathbb R^3}(1+\vert q\vert^2)^{\lambda-\frac{1}{2}}(\vert\nabla_q\partial_q^\beta\bar G\vert^2+\vert q\cdot\nabla_q\partial_q^\beta\bar G\vert^2){\rm d}q{\rm d}s\leq C.
	\end{equation}
Here, $C$ is a constant that depends only on $\lambda$ and the initial data.
\end{lemma}
\begin{proof}
By multiplying \eqref{A8} by $(1+\vert q\vert^2)^\lambda\bar G$ and integrating  over $\mathbb R^3$, we have
	\begin{align}\label{A12}
\frac{1}{2}\frac{d}{dt}\int_{\mathbb R^3}&(1+\vert q\vert^2)^\lambda\vert\bar G\vert^2{\rm d}q\nonumber\\
		=&\int_{\mathbb R^3}\bar\phi'(t)(3 \bar G+q\cdot\nabla_q\bar G)(1+\vert q\vert^2)^{\gamma}\bar G{\rm d}q+\int_{\mathbb R^3}e^{\bar\phi}\nabla_q\cdot(\Lambda_{0,q}\nabla_q\bar G)(1+\vert q\vert^2)^{\gamma}\bar G{\rm d}q\nonumber\\
		=&\bar\phi'(t)(\frac{3}{2}-\gamma)\int_{\mathbb R^3}(1+\vert q\vert^2)^\gamma\vert\bar G\vert^2{\rm d}q+\gamma\partial_t\bar\phi\int_{\mathbb R^3}(1+\vert q\vert^2)^{\gamma-1}\vert\bar G\vert^2{\rm d}q\nonumber\\
		&-e^{\bar\phi}\int_{\mathbb R^3}(1+\vert q\vert^2)^{\gamma-\frac{1}{2}}(\vert\nabla_q\bar G\vert^2+\vert q\cdot\nabla_q\bar G\vert^2){\rm d}q-2\gamma e^{\bar\phi}\int_{\mathbb R^3}(1+\vert q\vert^2)^{\gamma-\frac{1}{2}}q\cdot\nabla_q\bar G\bar G{\rm d}q\nonumber\\
		\leq& \Big(\bar\phi'(t)(\frac{3}{2}-\gamma)+(\partial_t\bar\phi)_++C_\eta e^{\bar\phi})\int_{\mathbb R^3}(1+\vert q\vert^2)^\gamma\vert\bar G\vert^2{\rm d}q\notag\\&-(1-\eta)e^{\bar\phi}\int_{\mathbb R^3}(1+\vert q\vert^2)^{\gamma-\frac{1}{2}}(\vert\nabla_q\bar G\vert^2+\vert q\cdot\nabla_q\bar G\vert^2){\rm d}q.\notag
	\end{align}
	For $\frac{1}{2}<\lambda<\frac{3}{2}$, applying Gr\"{o}nwall's inequality to the above inequality yields that
	\begin{equation}
		e^{-2(\frac{3}{2}-\lambda)\bar\phi}\int_{\mathbb R^3}(1+\vert q\vert^2)^\gamma\vert\bar G\vert^2{\rm d}q+\int_0^te^{(2\lambda-2)\bar\phi}\int_{\mathbb R^3}(1+\vert q\vert^2)^{\gamma-\frac{1}{2}}(\vert\nabla_q\bar G\vert^2+\vert q\cdot\nabla_q\bar G\vert^2){\rm d}q{\rm d}s\leq C.\notag
	\end{equation}
Thus \eqref{A10} and \eqref{A11} hold  for $\beta=0$.

Next  assume $1\leq\vert\beta\vert\leq N+1$. By applying $\partial_q^\beta:=\partial_{q_i}\partial_q^{\beta-e_i}$ to \eqref{A8}, we have
	\begin{equation}
		\begin{split}
			\partial_t\partial_q^\beta \bar G&-4\partial_t\bar\phi\partial_q^\beta \bar G
-\partial_t\bar\phi q\cdot\nabla_q\partial_q^\beta\bar G\\&	=e^{\bar\phi}\nabla_q\cdot(\Lambda_{0,q}\nabla_q\partial_q^\beta\bar G)+e^{\bar\phi}\partial_q^\beta\big(\nabla_q\cdot(\Lambda_{0,q}\nabla_q\bar G)\big)
-\nabla_q\cdot(\Lambda_{0,q}\nabla_q\partial_q^\beta\bar G).\notag
		\end{split}
	\end{equation}
Similarly, we can show that
\begin{align}\label{A16}
\frac{1}{2}\frac{d}{dt}\int_{\mathbb R^3}&(1+\vert q\vert^2)^\lambda\vert\partial_q^\beta\bar G\vert^2{\rm d}q\nonumber\\
		\leq& \Big(\bar\phi'(t)(\frac{5}{2}-\lambda)+(\partial_t\bar\phi)_++C_\eta e^{\bar\phi})\int_{\mathbb R^3}(1+\vert q\vert^2)^\lambda\vert\partial_q^\beta\bar G\vert^2{\rm d}q\notag\\&-(1-\eta)e^{\bar\phi}\int_{\mathbb R^3}(1+\vert q\vert^2)^{\lambda-\frac{1}{2}}(\vert\nabla_q\partial_q^\beta\bar G\vert^2+\vert q\cdot\nabla_q\partial_q^\beta\bar G\vert^2){\rm d}q\nonumber\\
		&+\int_{\mathbb R^3}e^{\bar\phi}\Big(\partial_q^\beta\big(\nabla_q\cdot(\Lambda_{0,q}\nabla_q\bar G)\big)-\nabla_q\cdot(\Lambda_{0,q}\nabla_q\partial_q^\beta\bar G)\Big)(1+\vert q\vert^2)^\lambda\partial_q^\beta\bar G{\rm d}q.
	\end{align}
	By a direct calculation, the last terms in \eqref{A16} can be bounded as
	\begin{equation*}
		\int_{\mathbb R^3}e^{\bar\phi}\Big(\partial_q^\beta\big(\nabla_q\cdot(\Lambda_{0,q}\nabla_q\bar G)\big)-\nabla_q\cdot(\Lambda_{0,q}\nabla_q\partial_q^\beta\bar G)\Big)(1+\vert q\vert^2)^\gamma\partial_q^\beta\bar G{\rm d}q\leq C^{\bar\phi}\sum_{\vert\beta'\vert\leq \vert\beta\vert}\int_{\mathbb R^3}(1+\vert q\vert^2)^\gamma\vert\partial_q^{\beta'}\bar G\vert^2{\rm d}q.
	\end{equation*}
Consequently, \eqref{A10} and \eqref{A11} for $\beta>0$ can be obtained by a standard induction argument. This completes the proof Lemma \ref{lemma A2}.
	\end{proof}


The estimates in the following three lemmas have been used in the proof of Theorem \ref{thm1}.

\begin{lemma}\label{lemma A3}
	Assume $(f,\Phi)$ is a solution to \eqref{pt-vnfp} and \eqref{pt-id}, and $(\bar F,\bar\phi)$ is a solution to \eqref{vnpf-ho} and \eqref{ho-id}. If $\vert\alpha\vert+\vert\beta\vert\leq4$ with $\vert\beta\vert<4$, then for $\ga>0$, we have
	\begin{equation}\label{A017}
		\begin{split}
\int_{\mathbb R^6}&e^{(\vert\alpha\vert+3\vert\beta\vert)\bar\phi}\partial_x^\alpha\partial_p^\beta\big(\nabla_x\sqrt{e^{2\phi}+\vert p\vert^2}\cdot\nabla_p\bar F\big)	(e^{2\phi}+\vert p\vert^2)^\gamma\partial_x^\alpha\partial_p^\beta f{\rm d}x{\rm d}p\\
\leq& Ce^{\bar\phi}\int_{\mathbb R^6}e^{(\vert\alpha\vert+3\vert\beta\vert)\bar\phi}(e^{2\phi}+\vert p\vert^2)^\gamma\vert\partial_x^\alpha\partial_p^\beta f\vert^2{\rm d}x{\rm d}p+Ce^{\bar\phi}\left\|\nabla_x\Phi\right\|_{H^{\vert\alpha\vert}_x}^2
\end{split}
	\end{equation}
for all $t>0$.
\end{lemma}
\begin{proof} By using Lemma \ref{lemma A1} and H\"{o}lder's inequality, it follows that for $\vert\beta\vert<2$
	\begin{align*}
\int_{\mathbb R^6}&e^{(\vert\alpha\vert+3\vert\beta\vert)\bar\phi}\partial_x^\alpha\partial_p^\beta\big(\nabla_x\sqrt{e^{2\phi}+\vert p\vert^2}\cdot\nabla_p\bar F\big)	(e^{2\phi}+\vert p\vert^2)^\gamma\partial_x^\alpha\partial_p^\beta f{\rm d}x{\rm d}p\nonumber\\
	\leq& C\int_{\mathbb R^6}e^{\frac{1}{2}(\vert\alpha\vert+3\vert\beta\vert)\bar\phi}(e^{2\phi}+\vert p\vert^2)^{\frac{\gamma}{2}}\vert\partial_x^\alpha\partial_p^\beta f\vert\sum_{0\leq\vert\alpha'\vert\leq\vert\alpha\vert}\vert\nabla_x\partial_x^{\alpha'}\Phi\vert\nonumber\\
		&\times(e^{2\bar\phi}+\vert p\vert^2)^{\frac{\gamma}{2}}\sum_{0\leq\vert\beta'\vert\leq\vert\beta\vert}\vert\nabla_p\partial_p^{\beta'}\bar F\vert e^{2\bar\phi}{\rm d}x{\rm d}p\nonumber\\
		\leq& Ce^{\bar\phi}\int_{\mathbb R^6}e^{(\vert\alpha\vert+3\vert\beta\vert)\bar\phi}(e^{2\phi}+\vert p\vert^2)^\gamma\vert\partial_x^\alpha\partial_p^\beta f\vert^2{\rm d}x{\rm d}p+Ce^{\bar\phi}\left\|\nabla_x\Phi\right\|_{H^{\vert\alpha\vert}_x}^2,
	\end{align*}
 and for $\vert\beta\vert\geq2$
	\begin{align*}
\int_{\mathbb R^6}&e^{(\vert\alpha\vert+3\vert\beta\vert)\bar\phi}\partial_x^\alpha\partial_p^\beta\big(\nabla_x\sqrt{e^{2\phi}+\vert p\vert^2}\cdot\nabla_p\bar F\big)	(e^{2\phi}+\vert p\vert^2)^\gamma\partial_x^\alpha\partial_p^\beta f{\rm d}x{\rm d}p\nonumber\\
	\leq& C\int_{\mathbb R^6}e^{\frac{1}{2}(\vert\alpha\vert+3\vert\beta\vert)\bar\phi}(e^{2\phi}+\vert p\vert^2)^{\frac{\gamma}{2}}\vert\partial_x^\alpha\partial_p^\beta f\vert\sum_{0\leq\vert\alpha'\vert\leq\vert\alpha\vert}\vert\nabla_x\partial_x^{\alpha'}\Phi\vert\\
		&\times\Big((e^{2\bar\phi}+\vert p\vert^2)^{\frac{\gamma}{2}}\sum_{0\leq\vert\beta'\vert\leq1}\vert\nabla_p\partial_p^{\beta'}\bar F\vert+\sum_{2\leq|\beta'|\leq|\beta|}(e^{2\bar\phi}+\vert p\vert^2)^{\frac{\gamma}{2}+\frac{1}{2}(|\beta|-|\beta'|)}\vert\nabla_p\partial_p^\beta f\vert\Big) e^{2\bar\phi}{\rm d}x{\rm d}p
\\ \leq& Ce^{\bar\phi}\int_{\mathbb R^6}e^{(\vert\alpha\vert+3\vert\beta\vert)\bar\phi}(e^{2\phi}+\vert p\vert^2)^\gamma\vert\partial_x^\alpha\partial_p^\beta f\vert^2{\rm d}x{\rm d}p+Ce^{\bar\phi}\left\|\nabla_x\Phi\right\|_{H^{\vert\alpha\vert}_x}^2.
	\end{align*}
Thus, \eqref{A017} holds. This completes the proof Lemma \ref{lemma A3}.
	\end{proof}

\begin{lemma}\label{lemma A4}
	Assume $(f,\Phi)$ is a solution to \eqref{pt-vnfp} and \eqref{pt-id}, and $(\bar F,\bar\phi)$ is a solution to \eqref{vnpf-ho} and \eqref{ho-id}. If $1\leq\vert\alpha\vert+\vert\beta\vert\leq4$ with $\vert\beta\vert<4$, then for $\ga>0$, it holds that
\begin{align}\label{A 17}
\int_{\mathbb R^6}&e^{(\vert\alpha\vert+3\vert\beta\vert)\bar\phi}\partial_x^\alpha\partial_p^\beta\big(e^{2\phi}\nabla_p\cdot(\Lambda_{\phi,p}\nabla_pf)\big)(e^{2\phi}+\vert p\vert^2)^\gamma\partial_x^\alpha\partial_p^\beta f{\rm d}x{\rm d}p\nonumber\\
&\leq -(1-\eta)\int_{\mathbb R^6}e^{(\vert\alpha\vert+3\vert\beta\vert)\bar\phi+2\phi}(e^{2\phi}+\vert p\vert^2)^{\gamma-\frac{1}{2}}\big(e^{2\phi}\vert\nabla_p\partial_x^\alpha\partial_p^\beta f\vert^2+\vert p\cdot\nabla_p\partial_x^\alpha\partial_p^\beta f\vert^2\big){\rm d}x{\rm d}p\nonumber\\
&\quad+Ce^{\bar\phi}\sum_{\substack{m+n\leq\vert\alpha\vert+\vert\beta\vert\\\ n<4}}\mathcal E^\gamma_{m,n}(t)+C\sum_{m+n<\vert\alpha\vert+\vert\beta\vert}\mathcal D_{m,n}^\gamma(t).
\end{align}
\end{lemma}
\begin{proof}
First of all, note that
	\begin{align}\label{A 18}
\int_{\mathbb R^6}&e^{(\vert\alpha\vert+3\vert\beta\vert)\bar\phi}\partial_x^\alpha\partial_p^\beta\big(e^{2\phi}\nabla_p\cdot(\Lambda_{\phi,p}\nabla_pf)\big)(e^{2\phi}+\vert p\vert^2)^\gamma\partial_x^\alpha\partial_p^\beta f{\rm d}x{\rm d}p\nonumber\\
=&\int_{\mathbb R^6}e^{(\vert\alpha\vert+3\vert\beta\vert)\bar\phi}\Lambda_{\phi,p}\partial_x^\alpha\partial_p^\beta(e^{2\phi}\nabla_pf)\big(\nabla_p(e^{2\phi}+\vert p\vert^2)^\gamma\partial_x^\alpha\partial_p^\beta f+(e^{2\phi}+\vert p\vert^2)^\gamma\nabla_p\partial_x^\alpha\partial_p^\beta f\big){\rm d}x{\rm d}p\nonumber\\
&+\int_{\mathbb R^6}e^{(\vert\alpha\vert+3\vert\beta\vert)\bar\phi}\Big(\partial_x^\alpha\partial_p^\beta\big(\nabla_p\cdot(\Lambda_{\phi,p}e^{2\phi}\nabla_pf)\big)-\nabla_p\cdot(\Lambda_{\phi,p}e^{2\phi}\nabla_p\partial_x^\alpha\partial_p^\beta f)\Big)\notag\\
&\quad\times(e^{2\phi}+\vert p\vert^2)^\gamma\partial_x^\alpha\partial_p^\beta f{\rm d}x{\rm d}p.
	\end{align}
The first term on the right hand side of \eqref{A 18} can be bounded as
\begin{align}\label{A 19}
\int_{\mathbb R^6}&e^{(\vert\alpha\vert+3\vert\beta\vert)\bar\phi}\Lambda_{\phi,p}\partial_x^\alpha\partial_p^\beta(e^{2\phi}\nabla_pf)\big(\nabla_p(e^{2\phi}+\vert p\vert^2)^\gamma\partial_x^\alpha\partial_p^\beta f+(e^{2\phi}+\vert p\vert^2)^\gamma\nabla_p\partial_x^\alpha\partial_p^\beta f\big){\rm d}x{\rm d}p\nonumber\\
\leq&-\int_{\mathbb R^6}e^{(\vert\alpha\vert+3\vert\beta\vert)\bar\phi+2\phi}(e^{2\phi}+\vert p\vert^2)^{\gamma-\frac{1}{2}}(e^{2\phi}\vert\nabla_p\partial_x^\alpha\partial_p^\beta f\vert^2+\vert p\cdot\nabla_p\partial_x^\alpha\partial_p^\beta f\vert^2){\rm d}x{\rm d}p\nonumber\\
&+\int_{\mathbb R^6}e^{\frac{1}{2}(\vert\alpha+3\vert\beta\vert)\bar\phi+\phi}(e^{2\phi}+\vert p\vert^2)^{\frac{\gamma}{2}-\frac{1}{4}}\vert p\cdot\nabla_p\partial_x^\alpha\partial_p^\beta f\vert e^{\frac{\bar\phi}{2}}(e^{\phi}+\vert p\vert^2)^{\frac{\gamma}{2}}\vert\partial_x^\alpha\partial_p^\beta f\vert{\rm d}x{\rm d}p\nonumber\\
&+Ce^{\bar\phi}\sum_{\substack{\vert\alpha'\vert+\vert\beta'\vert\leq\vert\alpha\vert+\vert\beta\vert\\|\beta'|<4}}\left\|e^{\frac{1}{2}(\vert\alpha\vert+3\vert\beta\vert)\bar\phi}\partial_x^\alpha\partial_p^\beta f\right\|_{L^2_{x,v}}\left\|e^{\frac{1}{2}(\vert\alpha'\vert+3\vert\beta'\vert)\bar\phi}\partial_x^{\alpha'}\partial_p^{\beta'}f\right\|_{L^2_{x,p}}\nonumber\\
&+C\sum_{\vert\alpha'\vert+\vert\beta'\vert<\vert\alpha\vert+\vert\beta\vert}\int_{\mathbb R^6}e^{\frac{1}{2}(\vert\alpha\vert+3\vert\beta\vert)\bar\phi+\frac{1}{2}(\vert\alpha'\vert+3\vert\beta'\vert)\bar\phi+2\phi}(e^{2\phi}+\vert p\vert^2)^{\gamma-\frac{1}{2}}\nonumber\\
&\times\Big(e^{2\phi}\vert\nabla_p\partial_x^\alpha\partial_p^\beta f\vert\vert\nabla_p\partial_x^{\alpha'}\partial_p^{\beta'} f\vert+\vert p\cdot\nabla_p\partial_x^\alpha\partial_p^\beta f\vert\vert p\cdot\nabla_p\partial_x^{\alpha'}\partial_p^{\beta'} f\vert\Big){\rm d}x{\rm d}p\nonumber\\
\leq& -(1-\eta)\int_{\mathbb R^6}e^{(\vert\alpha\vert+3\vert\beta\vert)\bar\phi+2\phi}(e^{2\phi}+\vert p\vert^2)^{\gamma-\frac{1}{2}}(e^{2\phi}\vert\nabla_p\partial_x^\alpha\partial_p^\beta f\vert^2+\vert p\cdot\nabla_p\partial_x^\alpha\partial_p^\beta f\vert^2){\rm d}x{\rm d}p\nonumber\\
&+Ce^{\bar\phi}\sum_{\substack{m+n\leq\vert\alpha\vert+\vert\beta\vert\\n<4}}\mathcal E^\gamma_{m,n}(t)+C\sum_{m+n<\vert\alpha\vert+\vert\beta\vert}\mathcal D_{m,n}^\gamma(t).
	\end{align}
As for the second term on the right hand side of \eqref{A 18}, it is straightforward to check that
\begin{equation}\label{A 20}
	\begin{split}
\int_{\mathbb R^6}&e^{(\vert\alpha\vert+3\vert\beta\vert)\bar\phi}\Big(\partial_x^\alpha\partial_p^\beta\big(\nabla_p\cdot(\Lambda_{\phi,p}e^{2\phi}\nabla_pf)\big)-\nabla_p\cdot\big(\Lambda_{\phi,p}\partial_x^\alpha\partial_p^\beta(e^{2\phi}\nabla_p f)\big)\Big)(e^{2\phi}+\vert p\vert^2)^\gamma\partial_x^\alpha\partial_p^\beta f{\rm d}x{\rm d}p\\
\leq& Ce^{\bar\phi}\sum_{\substack{m+n\leq\vert\alpha\vert+\vert\beta\vert\\n<4}}\mathcal E^\gamma_{m,n}(t).
\end{split}
\end{equation}
Substituting \eqref{A 19} and \eqref{A 20} into \eqref{A 18} gives \eqref{A 17}.
This completes the proof of  Lemma \ref{lemma A4}.
\end{proof}

\begin{lemma}\label{lemma A5}
	Assume $(f,\Phi)$ is a solution to \eqref{pt-vnfp} and \eqref{pt-id}, and $(\bar F,\bar\phi)$ is a solution to \eqref{vnpf-ho} and \eqref{ho-id}. If $\vert\alpha\vert+\vert\beta\vert\leq4$ with $\vert\beta\vert<4$, then for $\ga>0$, it holds that
	\begin{align}\label{A 22}
\int_{\mathbb R^6}&e^{(\vert\alpha\vert+3\vert\beta\vert)\bar\phi}\partial_x^\alpha\partial_p^\beta\Big(e^{2\phi}\nabla_p\cdot(\Lambda_{\phi,p}\nabla_p\bar F)-e^{2\bar\phi}\nabla_p\cdot(\Lambda_{\bar\phi,p}\nabla_p\bar F)\Big)(e^{2\phi}+\vert p\vert^2)^\gamma \partial_x^\alpha\partial_p^\beta f{\rm d}x{\rm d}p\nonumber\\
\leq& \eta\int_{\mathbb R^6}e^{(\vert\alpha\vert+3\vert\beta\vert)\bar\phi+2\phi}(e^{2\phi}+\vert p\vert^2)^{\gamma-\frac{1}{2}}(e^{2\phi}\vert\nabla_p\partial_x^\alpha\partial_p^\beta f\vert^2+\vert p\cdot\nabla_p\partial_x^\alpha\partial_p^\beta f\vert^2){\rm d}x{\rm d}p\nonumber\\
&+Ce^{\bar\phi}\int_{\mathbb R^6}e^{(\vert\alpha\vert+3\vert\beta\vert)\bar\phi}(e^{2\phi}+\vert p\vert^2)^\gamma\vert\partial_x^\alpha\partial_p^\beta f\vert^2{\rm d}x{\rm d}p+C_\eta e^{\bar\phi}\left\|\Phi\right\|_{H^{\vert\alpha\vert}_x}^2.
	\end{align}
\end{lemma}

\begin{proof}
If $\vert\alpha\vert=0$, we first rewrite the term on the left hand side  of \eqref{A 22} as
	\begin{align}
\int_{\mathbb R^6}&e^{3\vert\beta\vert\bar\phi}\partial_p^\beta\Big(e^{2\phi}\nabla_p\cdot(\Lambda_{\phi,p}\nabla_p\bar F)-e^{2\bar\phi}\nabla_p\cdot(\Lambda_{\bar\phi,p}\nabla_p\bar F)\Big)(e^{2\phi}+\vert p\vert^2)^\gamma \partial_p^\beta f{\rm d}x{\rm d}p\nonumber\\
=&\int_{\mathbb R^6}e^{3\vert\beta\vert\bar\phi}\partial_p^\beta\Big(e^{2\phi}\nabla_p\cdot(\Lambda_{\phi,p}\nabla_p\bar F)-e^{2\bar\phi}\nabla_p\cdot(\Lambda_{\phi,p}\nabla_p\bar F)\Big)(e^{2\phi}+\vert p\vert^2)^\gamma \partial_p^\beta f{\rm d}x{\rm d}p\nonumber\\
&+\int_{\mathbb R^6}e^{3\vert\beta\vert\bar\phi}\partial_p^\beta\Big(e^{2\bar\phi}\nabla_p\cdot(\Lambda_{\phi,p}\nabla_p\bar F)-e^{2\bar\phi}\nabla_p\cdot(\Lambda_{\bar\phi,p}\nabla_p\bar F)\Big)(e^{2\phi}+\vert p\vert^2)^\gamma \partial_p^\beta f{\rm d}x{\rm d}p
\notag\\:=&\CJ_1+\CJ_2.\notag
	\end{align}
We now  estimate $\CJ_1$ and $\CJ_2$ separately. For $\CJ_1$, by mean value theorem, there exists   $\theta\in(0,1)$ such that
\begin{align}\label{A 24}
\CJ_1=&2\int_{\mathbb R^6}\Phi e^{3\vert\beta\vert\bar\phi}e^{2\bar\phi+2\theta\Phi}\partial_p^\beta\big( \nabla_p\cdot(\Lambda_{\phi,p}\nabla_p\bar F)\big)(e^{2\phi}+\vert p\vert^2)^\gamma \partial_p^\beta f{\rm d}x{\rm d}p\nonumber\\
=&-2\int_{\mathbb R^6}\Phi e^{3\vert\beta\vert\bar\phi}e^{2\bar\phi+2\theta\Phi} \Lambda_{\phi,p}\nabla_p\partial_p^\beta\bar F\nabla_p\Big((e^{2\phi}+\vert p\vert^2)^\gamma \partial_p^\beta f\Big){\rm d}x{\rm d}p\nonumber\\
&+2\int_{\mathbb R^6}\Phi e^{3\vert\beta\vert\bar\phi}e^{2\bar\phi+2\theta\Phi}\nabla_p\cdot\Big(\partial_p^\beta(\Lambda_{\phi,p}\nabla_p\bar F)-\Lambda_{\phi,p}\nabla_p\partial_p^\beta\bar F\Big)(e^{2\phi}+\vert p\vert^2)^\gamma \partial_p^\beta f{\rm d}x{\rm d}p\nonumber\\
	\leq& \eta\int_{\mathbb R^6}e^{3\vert\beta\vert\bar\phi+2\phi}(e^{2\phi}+\vert p\vert^2)^{\gamma-\frac{1}{2}}(e^{2\phi}\vert\nabla_p\partial_p^\beta f\vert^2+\vert p\cdot\nabla_p\partial_p^\beta f\vert^2){\rm d}x{\rm d}p\nonumber\\
	&+Ce^{\bar\phi}\int_{\mathbb R^6}e^{3\vert\beta\vert\bar\phi}(e^{2\phi}+\vert p\vert^2)^\gamma\vert\partial_p^\beta f\vert^2{\rm d}x{\rm d}p+Ce^{2\bar\phi}\left\|\Phi\right\|_{L^2_x}^2.
\end{align}
Similarly, for $\CJ_2$, we have
\begin{align}\label{A 25}
\CJ_2=&-\int_{\mathbb R^6}e^{3\vert\beta\vert\bar\phi+2\bar\phi}\nabla_p\cdot\Big(\Lambda_{\phi,p}\nabla_p\partial_p^\beta\bar F-\Lambda_{\bar\phi,p}\nabla_p\partial_p^\beta\bar F\Big)\nabla_p\big((e^{2\phi}+\vert p\vert^2)^\gamma \partial_p^\beta f\big){\rm d}x{\rm d}p\nonumber\\
&+\int_{\mathbb R^6}e^{3\vert\beta\vert\bar\phi+2\bar\phi}\nabla_p\cdot\Big(\partial_p^\beta\big(\Lambda_{\phi,p}\nabla_p\bar F-\Lambda_{\bar\phi,p}\nabla_p\bar F\big)-\big(\Lambda_{\phi,p}\nabla_p\partial_p^\beta\bar F-\Lambda_{\bar\phi,p}\nabla_p\partial_p^\beta\bar F\big)\Big)\nonumber\\
&\quad\times(e^{2\phi}+\vert p\vert^2)^\gamma \partial_p^\beta f{\rm d}x{\rm d}p\nonumber\\
\leq& C\int_{\mathbb R^6}e^{3\vert\beta\vert\bar\phi+2\bar\phi}\vert\sqrt{e^{2\phi}+\vert p\vert^2}-\sqrt{e^{2\bar\phi}+\vert p\vert^2}\vert\vert p\cdot\nabla_p\partial_p^\beta\bar F\vert(e^{2\phi}+\vert p\vert^2)^{\gamma-1}\vert f\vert{\rm d}x{\rm d}p\nonumber\\
&+\int_{\mathbb R^6}e^{3\vert\beta\vert\bar\phi+2\bar\phi}\Big(\frac{e^{2\phi}\nabla_p\partial_p^\beta f\cdot\nabla_p\partial_p^\beta\bar F+p\cdot\nabla_p\partial_p^\beta fp\cdot\nabla_p\partial_p^\beta \bar F}{\sqrt{e^{2\phi}+\vert p\vert^2}}\nonumber\\
&\qquad\qquad\qquad\qquad-\frac{e^{2\bar\phi}\nabla_p\partial_p^\beta f\cdot\nabla_p\partial_p^\beta \bar F+p\cdot\nabla_p\partial_p^\beta fp\cdot\nabla_p\partial_p^\beta \bar F}{\sqrt{e^{2\bar\phi}+\vert p\vert^2}}\Big){\rm d}x{\rm d}p\nonumber\\
&+Ce^{\bar\phi}\int_{\mathbb R^6}(e^{2\phi}+\vert p\vert^2)^\gamma\vert\partial_p^\beta f\vert^2{\rm d}x{\rm d}p+Ce^{2\bar\phi}\left\|\Phi\right\|_{L^2_x}^2\nonumber\\
\leq& C\int_{\mathbb R^6}e^{\frac{3\vert\beta\vert}{2}\bar\phi}(e^{2\phi}+\vert p\vert^2)^{\frac{\gamma}{2}}\vert\partial_p^\beta f\vert(e^{2\bar\phi}+\vert p\vert^2)^{\frac{\gamma}{2}+\frac{1}{2}}\vert\nabla_p\partial_p^\beta\bar F\vert e^{2\bar\phi}\vert\Phi\vert{\rm d}x{\rm d}p\nonumber\\
&+C\int_{\mathbb R^6}e^{\frac{3\vert\beta\vert}{2}\bar\phi+2\phi}(e^{2\phi}+\vert p\vert^2)^{\frac{\gamma}{2}-\frac{1}{4}}\vert\nabla_p\partial_p^\beta f\vert(e^{2\bar\phi}+\vert p\vert^2)^{\frac{\gamma}{2}+\frac{1}{2}}\vert\nabla_p\partial_p^\beta\bar F\vert e^{\frac{3\bar\phi}{2}}\vert\Phi\vert{\rm d}x{\rm d}p\nonumber\\
&+C\int_{\mathbb R^6}e^{\frac{3\vert\beta\vert}{2}\bar\phi+\phi}(e^{2\phi}+\vert p\vert^2)^{\frac{\gamma}{2}-\frac{1}{4}}\vert p\cdot\nabla_p\partial_p^\beta f\vert(e^{2\phi}+\vert p\vert^2)^{\frac{\gamma}{2}+\frac{1}{4}}\vert\nabla_p\partial_p^\beta\bar F\vert e^{\bar\phi}\vert\Phi\vert{\rm d}x{\rm d}p\nonumber\\
\leq&\eta\int_{\mathbb R^6}e^{3\vert\beta\vert\bar\phi+2\phi}(e^{2\phi}+\vert p\vert^2)^{\gamma-\frac{1}{2}}(e^{2\phi}\vert\nabla_p\partial_p^\beta f\vert^2+\vert p\cdot\nabla_p\partial_p^\beta f\vert^2){\rm d}x{\rm d}p\nonumber\\
&+Ce^{\bar\phi}\int_{\mathbb R^6}e^{3\vert\beta\vert\bar\phi}(e^{2\phi}+\vert p\vert^2)^\gamma\vert\partial_p^\beta f\vert^2{\rm d}x{\rm d}p+C_\eta e^{\bar\phi}\left\|\Phi\right\|_{L^2_x}^2.
	\end{align}
Next, we consider the case $\vert\alpha\vert>0$. Similar to \eqref{A 24}, we have
\begin{align}\label{A 26}
\int_{\mathbb R^6}&e^{(\vert\alpha\vert+3\vert\beta\vert)\bar\phi}\partial_x^\alpha\partial_p^\beta\Big(e^{2\phi}\nabla_p\cdot(\Lambda_{\phi,p}\nabla_p\bar F)\Big)(e^{2\phi}+\vert p\vert^2)^\gamma \partial_x^\alpha\partial_p^\beta f{\rm d}x{\rm d}p\nonumber\\
	\leq& Ce^{\bar\phi}\int_{\mathbb R^6}e^{(\vert\alpha\vert+3\vert\beta\vert)\bar\phi}(e^{2\phi}+\vert p\vert^2)^\gamma\vert\partial_x^\alpha\partial_p^\beta f\vert^2{\rm d}x{\rm d}p+Ce^{\bar\phi}\left\|\Phi\right\|_{H^{\vert\alpha\vert}_x}^2.
\end{align}
From \eqref{A 24}, \eqref{A 25} and \eqref{A 26}, we obtain \eqref{A 22}. This completes the proof of Lemma \ref{lemma A5}.
\end{proof}

Finally, similar to Lemma \ref{lemma A3}, Lemma \ref{lemma A4} and Lemma \ref{lemma A5}, by using Lemma \ref{lemma A2}, we have the following lemma.

\begin{lemma}\label{lemma A6}
	Assume $(g,\Phi)$ is a solution to \eqref{g-eq} and \eqref{g-id}, and $(\bar G,\bar\phi)$ is a solution to \eqref{A8} and \eqref{A9}. If $\vert\alpha\vert+\vert\beta\vert\leq1$ , then for $\frac{1}{2}<\la<\frac{3}{2}$, the following estimates hold
\begin{equation}\label{g-tp}
	\begin{split}
		\int_{\mathbb R^6}&\partial_x^\alpha\partial_q^\beta\big(\nabla_x\sqrt{e^{2\Phi}+\vert q\vert^2}\cdot\nabla_q\bar G\big)	(e^{2\Phi}+\vert q\vert^2)^\la\partial_x^\alpha\partial_q^\beta g{\rm d}x{\rm d}q\\
		\leq& Ce^{(\frac{3}{2}-\la)\bar\phi}\int_{\mathbb R^6}(e^{2\Phi}+\vert q\vert^2)^\la\vert\partial_x^\alpha\partial_q^\beta g\vert^2{\rm d}x{\rm d}q+Ce^{(\frac{3}{2}-\la)\bar\phi}\left\|\nabla_x\Phi\right\|_{H^{\vert\alpha\vert}_x}^2,
		\end{split}
\end{equation}
 \begin{align}\label{g-df}
	\int_{\mathbb R^6}&\partial_x^\alpha\partial_q^\beta\big(e^{2\Phi}\nabla_q\cdot(\Lambda_{\Phi,q}\nabla_qg)\big)(e^{2\Phi}+\vert q\vert^2)^\la\partial_x^\alpha\partial_q^\beta g{\rm d}x{\rm d}q\nonumber\\
	\leq& -(1-\eta)\int_{\mathbb R^6}e^{\bar\phi+2\Phi}(e^{2\Phi}+\vert p\vert^2)^{\la-\frac{1}{2}}\big(e^{2\Phi}\vert\nabla_q\partial_x^\alpha\partial_q^\beta g\vert^2+\vert q\cdot\nabla_q\partial_x^\alpha\partial_q^\beta g\vert^2\big){\rm d}x{\rm d}q\nonumber\\
	&+Ce^{\bar\phi}\int_{\mathbb R^6}(e^{2\Phi}+|q|^2)^\la|\partial_x^\alpha\partial_q^\beta g|^2{\rm d}x{\rm d}q+C\sum_{m+n<\vert\alpha\vert+\vert\beta\vert}\tilde{\mathcal D}_{m,n}^\gamma(t),
\end{align}
and
\begin{align}\label{G-co}
	\int_{\mathbb R^6}&\partial_x^\alpha\partial_q^\beta\Big(e^{\bar\phi+2\Phi}\nabla_q\cdot(\Lambda_{\Phi,q}\nabla_q\bar G)-e^{\bar\phi}\nabla_p\cdot(\Lambda_{0,q}\nabla_q\bar G)\Big)(e^{2\Phi}+\vert q\vert^2)^\la \partial_x^\alpha\partial_q^\beta g{\rm d}x{\rm d}q\nonumber\\
	\leq& \eta\int_{\mathbb R^6}e^{\bar\phi+2\Phi}(e^{2\Phi}+\vert q\vert^2)^{\la-\frac{1}{2}}(e^{2\Phi}\vert\nabla_q\partial_x^\alpha\partial_q^\beta g\vert^2+\vert q\cdot\nabla_q\partial_x^\alpha\partial_q^\beta g\vert^2){\rm d}x{\rm d}q\nonumber\\
	&+Ce^{\bar\phi}\int_{\mathbb R^6}(e^{2\Phi}+\vert q\vert^2)^\la\vert\partial_x^\alpha\partial_q^\beta g\vert^2{\rm d}x{\rm d}q\notag\\
	&+C_\eta\sum_{\vert\beta'\vert\leq\vert\beta\vert} e^{\bar\phi}\left\|\Phi\right\|_{H^{\vert\alpha\vert}_x}^2\int_{\mathbb R^3}(1+\vert q\vert^2)^{\lambda-\frac{1}{2}}(\vert\nabla_q\partial^{\beta'}_{q}\bar G\vert^2
	+\vert q\cdot\nabla_q\partial_{q}^{\beta'}\bar G\vert^2){\rm d}q.
\end{align}
\end{lemma}
\begin{proof}
We will only prove \eqref{g-tp}, as the proofs of \eqref{g-df} and \eqref{G-co} are  similar to those of \eqref{A 17}  and \eqref{A 22}, respectively. By applying Lemma \ref{lemma A2} and H\"{o}lder's inequality, we obtain
	\begin{align*}
		\int_{\mathbb R^6}&\partial_x^\alpha\partial_q^\beta\big(\nabla_x\sqrt{e^{2\Phi}+\vert q\vert^2}\cdot\nabla_q\bar G\big)	(e^{2\Phi}+\vert q\vert^2)^\la\partial_x^\alpha\partial_q^\beta g{\rm d}x{\rm d}q\nonumber\\
		\leq& C\int_{\mathbb R^6}(e^{2\Phi}+\vert q\vert^2)^{\frac{\la}{2}}\vert\partial_x^\alpha\partial_q^\beta g\vert\sum_{\vert\alpha'\vert\leq\vert\alpha\vert}\vert\nabla_x\partial_x^{\alpha'}\Phi\vert
		(1+\vert q\vert^2)^{\frac{\la}{2}}\sum_{\vert\beta'\vert\leq\vert\beta\vert}\vert\nabla_q\partial_q^{\beta'}\bar G\vert {\rm d}x{\rm d}q\nonumber\\
		\leq& Ce^{(\frac{3}{2}-\la)\bar\phi}\int_{\mathbb R^6}(e^{2\Phi}+\vert q\vert^2)^\la\vert\partial_x^\alpha\partial_q^\beta g\vert^2{\rm d}x{\rm d}q+Ce^{(\frac{3}{2}-\la)\bar\phi}\left\|\nabla_x\Phi\right\|_{H^{\vert\alpha\vert}_x}^2.
	\end{align*}
This completes the proof of Lemma \ref{lemma A6}.
\end{proof}

\noindent {\bf Acknowledgements:}
SQL was supported by grants from the National Natural Science Foundation of China under contract 12325107.
TY was supported by a fellowship award from the Research Grants Council of the Hong Kong Special Administrative Region, China (Project no. SRFS2021-1S01).  TY would also like to thank the Kuok  foundation for its generous support. CSC extends his gratitude to the Department of Applied Mathematics at The Hong Kong Polytechnic University for their hospitality during his visit in 2023-2024. SQL also thanks the Department of Mathematics and IMS at The Chinese University of Hong Kong for their hospitality during his visit in the summer of 2024. And the authors would also like to thank the support by the Research Centre for Nonlinear Analysis at The Hong Kong Polytechnic University.

\end{document}